\documentclass[11pt,a4paper, reqno]{amsart}
\usepackage[utf8]{inputenc}               
\usepackage[a4paper,
left=4cm,
right=4cm,
top=2.9cm,
bottom=2.9cm]{geometry}
\usepackage{natbib}
\usepackage{mathtools}
\usepackage{dsfont}
\usepackage{booktabs}
\usepackage{float}
\usepackage{stmaryrd}

\DeclareFontFamily{U}{mathx}{}
\DeclareFontShape{U}{mathx}{m}{n}{<-> mathx10}{}
\DeclareSymbolFont{mathx}{U}{mathx}{m}{n}
\DeclareMathAccent{\widecheck}{0}{mathx}{"71}

\makeatletter
\newcases{crcases}{\quad}{%
	\hfil$\m@th\displaystyle{##}$\hfil}{\hfil$\m@th\displaystyle{##}$}{\lbrace}{.}
\makeatother

\usepackage{amsmath, amsthm, amssymb}
\usepackage{setspace}
\usepackage{tikz-cd}
\usepackage{tikz,tikz-3dplot, pgfmath, pgfplots}
\pgfplotsset{compat=1.18}
\usetikzlibrary{
	knots,
	hobby,
	decorations.pathreplacing,
	shapes.geometric,
	calc,
	decorations.markings
}

\usepackage{etoolbox}

\makeatletter
\let\old@tocline\@tocline
\let\section@tocline\@tocline

\newcommand{\subsection@dotsep}{4.5}
\newcommand{\subsubsection@dotsep}{4.5}
\patchcmd{\@tocline}
{\hfil}
{\nobreak
	\leaders\hbox{$\m@th
		\mkern \subsection@dotsep mu\hbox{.}\mkern \subsection@dotsep mu$}\hfill
	\nobreak}{}{}
\let\subsection@tocline\@tocline
\let\@tocline\old@tocline

\patchcmd{\@tocline}
{\hfil}
{\nobreak
	\leaders\hbox{$\m@th
		\mkern \subsubsection@dotsep mu\hbox{.}\mkern \subsubsection@dotsep mu$}\hfill
	\nobreak}{}{}
\let\subsubsection@tocline\@tocline
\let\@tocline\old@tocline

\let\old@l@subsection\l@subsection
\let\old@l@subsubsection\l@subsubsection

\def\@tocwriteb#1#2#3{%
	\begingroup
	\@xp\def\csname #2@tocline\endcsname##1##2##3##4##5##6{%
		\ifnum##1>\c@tocdepth
		\else \sbox\z@{##5\let\indentlabel\@tochangmeasure##6}\fi}%
	\csname l@#2\endcsname{#1{\csname#2name\endcsname}{\@secnumber}{}}%
	\endgroup
	\addcontentsline{toc}{#2}%
	{\protect#1{\csname#2name\endcsname}{\@secnumber}{#3}}}%

\newlength{\@tocsectionindent}
\newlength{\@tocsubsectionindent}
\newlength{\@tocsubsubsectionindent}
\newlength{\@tocsectionnumwidth}
\newlength{\@tocsubsectionnumwidth}
\newlength{\@tocsubsubsectionnumwidth}
\newcommand{\settocsectionnumwidth}[1]{\setlength{\@tocsectionnumwidth}{#1}}
\newcommand{\settocsubsectionnumwidth}[1]{\setlength{\@tocsubsectionnumwidth}{#1}}
\newcommand{\settocsubsubsectionnumwidth}[1]{\setlength{\@tocsubsubsectionnumwidth}{#1}}
\newcommand{\settocsectionindent}[1]{\setlength{\@tocsectionindent}{#1}}
\newcommand{\settocsubsectionindent}[1]{\setlength{\@tocsubsectionindent}{#1}}
\newcommand{\settocsubsubsectionindent}[1]{\setlength{\@tocsubsubsectionindent}{#1}}

\renewcommand{\l@section}{\section@tocline{1}{\@tocsectionvskip}{\@tocsectionindent}{\@tocsectionnumwidth}{\@tocsectionformat}}%
\renewcommand{\l@subsection}{\subsection@tocline{1}{\@tocsubsectionvskip}{\@tocsubsectionindent}{\@tocsubsectionnumwidth}{\@tocsubsectionformat}}%
\renewcommand{\l@subsubsection}{\subsubsection@tocline{1}{\@tocsubsubsectionvskip}{\@tocsubsubsectionindent}{\@tocsubsubsectionnumwidth}{\@tocsubsubsectionformat}}%
\newcommand{\@tocsectionformat}{}
\newcommand{\@tocsubsectionformat}{}
\newcommand{\@tocsubsubsectionformat}{}
\expandafter\def\csname toc@1format\endcsname{\@tocsectionformat}
\expandafter\def\csname toc@2format\endcsname{\@tocsubsectionformat}
\expandafter\def\csname toc@3format\endcsname{\@tocsubsubsectionformat}
\newcommand{\settocsectionformat}[1]{\renewcommand{\@tocsectionformat}{#1}}
\newcommand{\settocsubsectionformat}[1]{\renewcommand{\@tocsubsectionformat}{#1}}
\newcommand{\settocsubsubsectionformat}[1]{\renewcommand{\@tocsubsubsectionformat}{#1}}
\newlength{\@tocsectionvskip}
\newcommand{\settocsectionvskip}[1]{\setlength{\@tocsectionvskip}{#1}}
\newlength{\@tocsubsectionvskip}
\newcommand{\settocsubsectionvskip}[1]{\setlength{\@tocsubsectionvskip}{#1}}
\newlength{\@tocsubsubsectionvskip}
\newcommand{\settocsubsubsectionvskip}[1]{\setlength{\@tocsubsubsectionvskip}{#1}}

\patchcmd{\tocsection}{\indentlabel}{\makebox[\@tocsectionnumwidth][l]}{}{}
\patchcmd{\tocsubsection}{\indentlabel}{\makebox[\@tocsubsectionnumwidth][l]}{}{}
\patchcmd{\tocsubsubsection}{\indentlabel}{\makebox[\@tocsubsubsectionnumwidth][l]}{}{}

\newcommand{\@sectypepnumformat}{}
\renewcommand{\contentsline}[1]{%
	\expandafter\let\expandafter\@sectypepnumformat\csname @toc#1pnumformat\endcsname%
	\csname l@#1\endcsname}
\newcommand{\@tocsectionpnumformat}{}
\newcommand{\@tocsubsectionpnumformat}{}
\newcommand{\@tocsubsubsectionpnumformat}{}
\newcommand{\setsectionpnumformat}[1]{\renewcommand{\@tocsectionpnumformat}{#1}}
\newcommand{\setsubsectionpnumformat}[1]{\renewcommand{\@tocsubsectionpnumformat}{#1}}
\newcommand{\setsubsubsectionpnumformat}[1]{\renewcommand{\@tocsubsubsectionpnumformat}{#1}}
\renewcommand{\@tocpagenum}[1]{%
	\hfill {\mdseries\@sectypepnumformat #1}}

\let\oldappendix\appendix
\renewcommand{\appendix}{%
	\leavevmode\oldappendix%
	\addtocontents{toc}{%
		\protect\settowidth{\protect\@tocsectionnumwidth}{\protect\@tocsectionformat\sectionname\space}%
		\protect\addtolength{\protect\@tocsectionnumwidth}{2em}}%
}
\makeatother

\makeatletter
\settocsectionnumwidth{2em}
\settocsubsectionnumwidth{2.5em}
\settocsubsubsectionnumwidth{3em}
\settocsectionindent{1pc}%
\settocsubsectionindent{\dimexpr\@tocsectionindent+\@tocsectionnumwidth}%
\settocsubsubsectionindent{\dimexpr\@tocsubsectionindent+\@tocsubsectionnumwidth}%
\makeatother

\settocsectionvskip{10pt}
\settocsubsectionvskip{0pt}
\settocsubsubsectionvskip{0pt}

\settocsectionformat{\bfseries}%
\settocsubsectionformat{\mdseries}
\settocsubsubsectionformat{\mdseries}
\setsectionpnumformat{}
\setsubsectionpnumformat{\mdseries}
\setsubsubsectionpnumformat{\mdseries}

\theoremstyle{plain}
\theoremstyle{definition}
\newtheorem{Theorem}{Theorem}[section]
\definecolor{shadecolor}{RGB}{217, 255, 217}

\newtheorem{Lemma}[Theorem]{Lemma}
\newtheorem*{theorem*}{Theorem}

\newtheorem{Proposition}[Theorem]{Proposition}

\newtheorem{Corollary}[Theorem]{Corollary}
\newtheorem{definition}[Theorem]{Definition}
\newtheorem{Remark}[Theorem]{Remark}

\numberwithin{equation}{section}

\DeclareMathOperator{\supp}{supp}

\def\pdop2{\ensuremath{\scalebox{1.125}{\dh}}}
\def\pdop{\ensuremath{\P\llbracket\mathfrak{d}\rrbracket}}

\usepackage{mathrsfs}

\usepackage[final]{pdfpages}

\usepackage[colorlinks=blue, urlcolor=blue, citecolor=black]{hyperref}

\renewcommand{\arraystretch}{2}

\renewcommand{\P}{\mathsf{P}}

\newcommand{\C}{\mathbb{C}}

\newcommand{\R}{\mathbb{R}}

\newcommand{\N}{\mathbb{N}}

\title[An inverse spectral problem for the wave equation]{An inverse spectral problem for the wave equation with a varying Robin parameter at the boundary}
\author[A. Hannukainen, P. Kulikov, M. Lassas, L. Oksanen]{Antti Hannukainen, Petr Kulikov, Matti Lassas, Lauri Oksanen}

\begin{document}	
\maketitle
\thispagestyle{empty}
\begin{quote}
	\footnotesize
	{\sc Abstract.} An inverse spectral problem for a $(1+1)$-dimensional wave equation with wave speed $c(x)$ is considered. The spectral data consist of the eigenvalues of the Neumann--Dirichlet problem for the operator $-c^2(x)\partial_x^2$ together with their derivatives with respect to a Robin parameter at the boundary. We propose a two-step recovery procedure. In the first step, the Neumann-to-Dirichlet map $\Lambda$ is reconstructed from the spectral data using the results of \cite{kurylev_lassas_weder}. In the second step, the reconstructed $\Lambda$ is used within the regularization strategy for the boundary control method presented in \cite{korpela_lassas_oksanen}. As an improvement, the regularization of the control problems employs a weighted combination of $L^2$ and $H^1$ penalty terms, and we prove a uniform convergence result for the family of optimal controls, thereby allowing the regularization parameters to be tuned simultaneously.

	We also prove that the forward spectral map, which assigns to a wave speed its spectral data, is Lipschitz continuous, and that the inverse spectral map underlying the reconstruction procedure is H\"older continuous on a subclass of wave speeds.

	The method is implemented numerically and validated on both smooth and piecewise smooth wave speed profiles. Numerical experiments demonstrate that the recovery of the wave speed exhibits high accuracy and robustness.
\end{quote}

\medskip


\section{Introduction}
\subsection{Problem setting and main results}\label{intro}
Let $\Omega = (0,1)$. The wave speed function $c\in C^2(\mathbb{R})$ satisfies
$$c \geqslant C_{min} > 0,
\qquad
c|_{\mathbb{R}\setminus\Omega} = 1,
\qquad
\|c\|_{C^2(\mathbb{R})} \leqslant C_{max},$$
for some constants $C_{min}\in(0,1]$ and $C_{max}>0$. Let $T>0$, and let $u:[0,2T]\times\Omega\to\mathbb{R}$ and $f\in L^2(0,2T)$. The $(1+1)$-dimensional wave equation problem reads
$$(\mathcal{W}_f)\;:\;\begin{cases}
	\partial_t^2 u - c^2(x)\,\partial_x^2 u = 0, & \text{in }(0,2T) \times\Omega,\\
	\partial_x u(t,0) = f(t), & t\in (0,2T),\\
	u(t,1) = 0, & t\in (0,2T),\\
	u(0,x)=\partial_tu(0,x) = 0, & \text{in }\Omega.
\end{cases}$$
The function $c$ is assumed to equal $1$ outside $\Omega$, reflecting the physical setting where the medium is homogeneous beyond the domain of interest. The condition on $T$ guarantees that waves carry
information about the whole of $\Omega$. The \emph{Neumann-to-Dirichlet map} $\Lambda: L^2(0,2T)\to L^2(0,2T)$ is defined by
$$\Lambda\,:\,f\mapsto u|_{x=0},\qquad f\in L^2(0,2T),$$
where $u$ is the solution to $(\mathcal{W}_f)$. The boundary measurements are encoded in the Neumann-to-Dirichlet map
$\Lambda$, which records the Dirichlet trace of the solution at
$x=0$ in response to a prescribed Neumann datum $f$ at the same boundary point. A standard inverse problem is to recover the wave speed function $c(x)$ given the Neumann-to-Dirichlet map $\Lambda$.

An inverse spectral problem is the task of reconstructing the wave speed function $c(x)$ from a spectral data of the differential operator $-c^2(x)\partial_x^2$. A natural approach to the forward problem for wave equation on $\Omega$ is the separation of variables, which decomposes the problem into a time oscillation equation and a spatial spectral problem. Since the time oscillation equation has constant coefficients, it carries no information about the medium, whereas all dependence on $c(x)$ enters through the spatial operator $-c^2(x)\partial_x^2$. The spatial part solves a Sturm–Liouville type spectral problem. Thus, for spectral data, the spatial spectral problem is the object of interest. 

In the unperturbed case, we solve the Neumann--Dirichlet spectral problem,
$$(\mathcal{S})\;:\;\begin{cases}
	-c^2(x)\,\partial_x^2\phi = \lambda\,\phi, & \text{in }\Omega,\\
	\partial_x\phi(0) = 0,\\
	\phi(1) = 0.
\end{cases}$$
To obtain a richer data set, we proceed in analogy with the classical Borg--Levinson ``two-spectra'' problem and its multidimensional extensions. In this problem, an unknown coefficient of a differential operator is recovered from two spectra corresponding to two different boundary conditions. We therefore introduce a one-parameter family of Robin-type boundary conditions at \(x=0\). That is, for \(\omega\in(0,1)\), we study the perturbed spectral problems
$$(\mathcal{S}_\omega)\;:\;\begin{cases}
	-c^2(x)\,\partial_x^2\phi = \lambda\,\phi, & \text{in }\Omega,\\
	\partial_x\phi(0)+\omega\,\phi(0) = 0,\\
	\phi(1) = 0.
\end{cases}$$
For $\omega=0$, a perturbed problem reduces to the Neumann--Dirichlet spectral problem $(\mathcal{S})$. Following physical literature, the coefficient $\omega$ in the Robin boundary condition above is often called the \emph{impedance}. Robin boundary conditions are also called impedance boundary conditions from their application in electromagnetic problems. The eigenvalues $\{\lambda_n(\omega)\}_{n\in\N}$ of the problems $(\mathcal{S}_\omega)$, $\omega\in [0,1)$, form the spectral data from which we seek to reconstruct $c(x)$. In particular, the spatial operator $-c^2(x)\partial_x^2$, together with the boundary conditions in $(\mathcal{S}_\omega)$, fully encodes the dependence on $c(x)$. Consequently, the family of spectral problems $(\mathcal{S}_\omega)$ is the gateway between the unknown wave speed $c(x)$ and the spectral data used in our inverse problem.

In Section \ref{isp} we show that for each $n\in\N$ the corresponding eigenvalue curve $\lambda_n\,:\,\omega\mapsto \lambda_n(\omega)$ of $(\mathcal{S}_\omega)$ is differentiable with respect to $\omega$, and the derivatives at some $\omega\in [0,1)$ are denoted by $\lambda_{n}'(\omega)$. Moreover, these derivatives encode the boundary behaviour of the eigenfunctions at $x=0$ and thus provide a formally determined data set for the inverse spectral problem. Let $\mathcal{C}$ be the admissible class of wave speeds,
$$\mathcal{C} := \left\{c\in C^2(\R) :
c\geqslant C_{min},\ \|c\|_{C^2(\R)}\leqslant C_{max},\ c|_{\R\setminus\Omega}\equiv 1 \right\}.$$
For every $\omega\in[0,1)$ each $c\in\mathcal{C}$ fully determines the eigenvalues $\{\lambda_n(\omega)\}_{n\in\N}$ of $(\mathcal{S}_\omega)$ and hence the sequence $\{\lambda_{n}'(0)\}_{n\in\N}$. Define the \textit{forward spectral map} by 
$$\Phi(c) = \left\{\lambda_n(0),\lambda_{n}'(0)\right\}_{n\in\N}.$$ 
Let $\mathcal{D}=\Phi(\mathcal{C})$ be the corresponding set of admissible spectral data. 

The goal of the paper is to show that the sequences $\{\lambda_n(0)\}_{n\in\N}$ and $\{\lambda_{n}'(0)\}_{n\in\N}$ uniquely determine the wave speed $c(x)$ in $\Omega$, that is, that the inverse map $\Phi^{-1}:\mathcal{D}\to\mathcal{C}$ exists, and to quantify the stability of this correspondence in both directions. 

Now, we present the main results of the paper.
\begin{Theorem}\label{3mainres}
	\textit{The forward spectral map $\Phi:\mathcal{C}\to\mathcal{D}$ is bijective.}
\end{Theorem}
Let $\ell^\infty_s$, $s\in\R$, be the space of real sequences $a=\{a_n\}_{n\in\N}$ with 
$$\|a\|_{\ell^\infty_s}=\sup \{n^{-s}|a_n|\,:\,n\in\N\}<+\infty.$$ 
For nonnegative quantities $A$ and $B$ we write $A\lesssim B$ if $A\leqslant CB$ with a constant $C>0$ depending only on $C_{min}$, $C_{max}$ and $T$, and $A\asymp B$ if $A\lesssim B$ and $B\lesssim A$. When the constant is allowed to depend on an additional parameter $\theta$, we write $A\lesssim_\theta B$. 

We show that the forward spectral map $\Phi$ is Lipschitz continuous as a map from $\mathcal{C}$ to $\ell^\infty_2\times\ell^\infty_0$, where $\mathcal{C}$ is equipped with the $C^2(\overline\Omega)$ norm and $\ell^\infty_2\times\ell^\infty_0$ with the sum of the norms.
\begin{Proposition}\label{4mainres}
\textit{For any $c_1,c_2\in\mathcal{C}$, $\|\Phi(c_1)-\Phi(c_2)\|_{\ell^\infty_2\times\ell^\infty_0}\lesssim\|c_1-c_2\|_{C^2(\overline\Omega)}$.}
\end{Proposition}

We call $\Phi^{-1}$ the \textit{inverse spectral map}. Fix $d\in (0,1)$ such that $d> \tfrac{C'}{C'+1}$, where $C'=C_{min}^{-1}-C_{max}^{-1}$. It is possible to choose such $d$, since $C'\geqslant0$ due to $C_{min}\leqslant 1\leqslant C_{max}$. For the inverse direction we work on the subclass 
$$\mathcal{C}'=\{c\in\mathcal{C}\cap C^3(\R):\|c\|_{C^3(\R)}\leqslant C_{max},c|_{\R\setminus(0,1-d)}\equiv 1 \},\quad \mathcal{D}'=\Phi(\mathcal{C}').$$
The additional derivative and the buffer $[1-d,1]$ are both required by the regularization strategy of \cite{korpela_lassas_oksanen}, which is formulated on the half line; in our case the Dirichlet condition at $x=1$ reflects the waves, so the reconstruction must stop before the reflected signal returns, and the lower bound on $d$ is what allows a single stopping time chosen a priori to serve every $c\in\mathcal{C}'$; see Section~\ref{invstabr}. We show that $\Phi^{-1}|_{\mathcal{D}'}$ (below simply $\Phi^{-1}$) is H\"older continuous with every exponent below $\kappa/2$, where $\kappa=1/54$ is the exponent in the stability estimate of \cite{korpela_lassas_oksanen} for the recovery of $c$ from $\Lambda$. Thus small perturbations of the spectral data lead to controlled perturbations of the reconstructed wave speed.
\begin{Theorem}\label{5mainres}
\textit{The inverse spectral map $\Phi^{-1}:\mathcal{D}'\to\mathcal{C}'$ is H\"older continuous from $\ell^\infty_2\times\ell^\infty_0$ to $C^2(\overline\Omega)$ with every exponent $\theta<\frac{\kappa}{2}$, that is, for any $c_1,c_2\in\mathcal{C}'$,
$$\|c_1-c_2\|_{C^2(\overline\Omega)}\lesssim_{\theta}\|\Phi(c_1)-\Phi(c_2)\|^{\theta}_{\ell^\infty_2\times\ell^\infty_0},\quad \theta\in \left(0,\frac\kappa2\right).$$
Moreover, if $\varepsilon=\|\Phi(c_1)-\Phi(c_2)\|_{\ell^\infty_2\times\ell^\infty_0}\leqslant\frac12$, then 
$$\|c_1-c_2\|_{C^2(\overline\Omega)}\lesssim\big(\varepsilon\log(1/\varepsilon)\big)^{\kappa/2}.$$}
\end{Theorem}
Recall that the H\"older space $C^{2,\eta}(\overline\Omega)$, $\eta\in(0,1)$, consists of $w\in C^2(\overline\Omega)$ with
$$\|w\|_{C^{2,\eta}(\overline\Omega)}=\|w\|_{C^2(\overline\Omega)}+\sup_{x,y\in\overline\Omega,\,x\neq y}\frac{|w''(x)-w''(y)|}{|x-y|^{\eta}}<+\infty.$$
\begin{Corollary}\label{6mainres}
\textit{For any $c_1,c_2\in\mathcal{C}'$ and $\eta\in[0,1)$,}
$$\|c_1-c_2\|_{C^{2,\eta}(\overline\Omega)}\lesssim_{\theta}
\|\Phi(c_1)-\Phi(c_2)\|^{\theta(1-\eta)}_{\ell^\infty_2\times\ell^\infty_0} ,\quad \theta\in \left(0,\frac\kappa2\right).$$
\end{Corollary}

\medskip

In applications, the spectral data $\{\lambda_{n}'(0)\}_{n\in\N}$ are not directly available. Instead, for a sufficiently small fixed parameter $\omega>0$ one can measure two finite sets of eigenvalues
$$\{\lambda_n(0)\}_{n=1}^N \quad \text{from } (\mathcal{S}_0),
\qquad
\{\lambda_n(\omega)\}_{n=1}^N \quad \text{from } (\mathcal{S}_\omega),$$
and approximate the derivatives by finite differences:
$$\frac{d\lambda_n}{d\omega}(0) \approx \frac{\lambda_n(\omega) - \lambda_n(0)}{\omega},\qquad n=1,\dots,N.$$
The numerical inverse problem can then be phrased as follows: given the two finite spectra $\{\lambda_n(0)\}_{n=1}^N$  and $\{\lambda_n(\omega)\}_{n=1}^N$ for a small yet fixed $\omega>0$, reconstruct an approximation of the wave speed $c(x)$ on $\Omega$ that is consistent with the reconstruction $\Phi^{-1}(\{\lambda_n(0),\lambda_n'(0)\}_{n\in\N})$. In other words, one seeks a stable, computationally feasible approximation of the inverse spectral map $\Phi^{-1}$ based on truncated and finite-difference–perturbed spectral data.

\medskip

The study of such inverse problems has a long history. The one-dimensional inverse problems for hyperbolic equations were pioneered by A.~S.~Blagove\v{s}\v{c}enski\u{\i}~\cite{blag1,blag2}, who derived integral identity formulations and introduced the operators that now bear his name. The \emph{Boundary Control method} provides a different point of view on
inverse problems. It reformulates them as controllability questions for the
wave equation and has become a conceptual basis for many boundary spectral
and time-domain reconstruction methods. In the multidimensional setting, this
approach was introduced by M.~I.~Belishev~\cite{bel3} and further developed
by M.~I.~Belishev and Y.~V.~Kurylev~\cite{belkur}. In particular, the method
yields uniqueness and reconstruction results for smooth wave speeds from
boundary response data, such as the Neumann-to-Dirichlet operator. Shortly thereafter, M.~I.~Belishev and Y.~V.~Kurylev applied this method to the reconstruction of a Riemannian manifold from its spectral data~\cite{belkur2}. Further developments, including the reconstruction of manifolds and metrics
and related inverse problems such as obstacle recovery, are discussed in
M.~I.~Belishev's review~\cite{bel4}. For a detailed account of subsequent progress in the Boundary Control method, we refer to M.~I.~Belishev's later survey~\cite{belrec}.

A multidimensional Borg--Levinson theorem for Riemann--Schr\"odinger operators with Robin boundary conditions was proven in 2005 in \cite{kurylev_lassas_weder}. Their result shows that local spectral data for varying impedance functions determine the manifold, the potential, and the background impedance, and is closely related to the spectral perturbation data used here. For a classical reference on inverse boundary spectral problems, see the monograph \cite{katchalov_kurylev_lassas}.

Numerical realizations of the boundary control method via minimization and time reversal were pioneered in \cite{oksanen_ip,first}, and a regularization strategy for the \(1+1\)-dimensional wave equation inverse problem was developed in \cite{korpela_lassas_oksanen,KLO1}. The eigenvalue computations needed for our spectral data are naturally carried out using finite element methods as surveyed by D.~Boffi~\cite{boffi}, and the underlying spectral theory relies on the classical work of M.~S.~Birman and M.~Z.~Solomyak~\cite{birmansol}. The boundary control method has also been combined with spectral data in the partial boundary setting in \cite{oksanen_spec}; the authors applied it to a partial data Borg–Levinson inverse spectral problem for Schrödinger operators . 

Within this framework, the present work fits into the development of inverse spectral and boundary control methods in four ways. First, it combines boundary spectral perturbation data with the boundary control method and travel-time geometry. Second, it makes explicit the inverse spectral map $\Phi^{-1}$ from spectral data to the wave speed $c(x)$. Third, it establishes stability of the inverse spectral map: a Lipschitz estimate for $\Phi$ and a H\"older estimate for $\Phi^{-1}$, the latter obtained by combining a hidden-regularity estimate for the Neumann-to-Dirichlet map with the regularization strategy of \cite{korpela_lassas_oksanen}. Fourth, it formulates the associated numerical inverse spectral problem in terms of two finite sets of perturbed eigenvalues, $\{\lambda_n(0)\}_{n\in\N}$ and $\{\lambda_n(\omega)\}_{n\in\N}$ for a sufficiently small $\omega>0$. Furthermore, we construct a numerical implementation of the inverse spectral map~$\Phi^{-1}$. 

\medskip

\subsection{Outline of the paper and further results}\label{potpamr}
In this section we outline the structure of the paper. We also state further results, which are steps in the proofs of the main results and are of independent interest.

In Section~\ref{isp}, we analyze the perturbed spectral problems $(\mathcal{S}_\omega)$ and show differentiability of the eigenvalue curves with respect to $\omega$. In Section \ref{wefpmain}, we prove existence and uniqueness of the strong solutions of the wave equation $(\mathcal{W}_f)$ for an appropriate $f$. Further, for $\lambda>0$ we introduce the convolution-type operator $S_\lambda$,
$$S_{\lambda}[f](t)=\int_0^t \frac{\sin(s\sqrt{\lambda})}{\sqrt{\lambda}}f(t-s)\,ds,\qquad f\in L^1_{\ell oc}(\R_+).$$
This operator is well defined on $L^2(0,T)$. In Subsection \ref{gboep2}, we prove the first result that identifies the Neumann-to-Dirichlet map $\Lambda$ in terms of eigenvalues $\{\lambda_n(0)\}_{n\in\N}$ and their derivatives $\{\lambda_n'(0)\}_{n\in\N}$. 
\begin{Proposition}\label{1mainres}
	\textit{The Neumann-to-Dirichlet map $\Lambda$ is a convolution-type operator and admits the representation}
	\[
	\Lambda[f] = \sum_{n\in \N}\lambda_n'(0)\, S_{\lambda_n(0)}[f] , \qquad f\in L^2(0,T).
	\]
\end{Proposition} 

The function $d$ denotes the travel-time distance generated by the wave speed~$c(x)$. More precisely, the time needed for a wave to travel from $0$ to a point $x\in\overline{\Omega}$~is
\[
\tau(x)=d(0,x)=\int_0^x \frac{dy}{c(y)}.
\]
Equivalently, the wave speed $c(x)$ induces the metric element $ds=c(x)^{-1}dx$. For $r\in[0,\tau(1)]$, the set $\Omega(r)=\{x\in\overline{\Omega}:d(0,x)\leqslant r\}$ is called the domain of influence. In other words, $\Omega(r)$ consists exactly of the points that can be reached by a signal emitted from the boundary point $x=0$ within time~$r$. Furthermore, we fix a final time $T\geqslant\tau(1)$, which guarantees that a signal emitted at $x=0$ reaches every point of $\Omega$ within time $T$.

In Section~\ref{premip}, we introduce the Blagove\v{s}\v{c}enski\u{\i} operator $K$, expressed in terms of the Neumann-to-Dirichlet map $\Lambda$. This operator is used to recover the indicator functions of the domains of influence $\Omega(r)$. Before stating the corresponding result, we collect the necessary notations.

We use the weighted space
$$L^2_c= L^2\bigl(\Omega, c^{-2}(x)\,dx\bigr)$$
consisting of equivalence classes of measurable functions $v:\Omega\to\R$ such that $\int_\Omega |v(x)|^2c^{-2}(x)\,dx < +\infty$. Its inner product and norm are $(v,w)_{L^2_c} := \int_\Omega v(x)\,w(x)c^{-2}(x)\,dx$ and $\|v\|_{L^2_c} := (v,v)_{L^2_c}^{1/2}$. Since $c(x)$ is uniformly bounded from above and bounded away from zero on $\Omega$, the norm $\|\cdot\|_{L^2_c}$ is equivalent to the usual $L^2(\Omega)$ norm. Hence $L^2_c$ coincides with $L^2(\Omega)$ as a set, but is equipped with the weighted inner product above.

We also introduce spaces of boundary controls localized in time. We set
\[
H(r)=\{f\in H^1(0,2T):\mathrm{supp}f\subset [T-r,T]\},\qquad r\in [0,T].
\]
The $H^1(0,2T)$ seminorm is given by $|f|_{H^1(0,2T)} = \|\partial_t f\|_{L^2(0,2T)}$. There exists a unique bounded and symmetric linear operator $S:H(r)\to (H(r))^*$ satisfying
\begin{equation*}\label{stiff0}
	\langle S f, h\rangle =(\partial_t f,\partial_t h)_{L^2(0,2T)},\qquad f,h\in H(r),
\end{equation*}
and we call it the stiffness operator. The orthogonal projector \(\mathcal{P}_r:L^2(0,2T)\to L^2(0,2T)\) is defined by $\mathcal{P}_r[f]=\mathds{1}_{[T-r,T]}f$, where \(\mathds{1}_{[T-r,T]}\) is the indicator function of \([T-r,T]\). Finally, given a Neumann boundary input $f$, let $u^f$ denote the solution of $(\mathcal{W}_f)$. We write $u^f(T)=u^f(T,\cdot)$ for the spatial profile of this solution at time $T$.

\begin{Proposition}\label{2mainres}
	\textit{Let $0\leqslant r\leqslant \tau(1)$ and $\alpha,\beta >0$. Then the regularized optimization problem
		\[
		\min_{f\in H(r)}\left(\|u^f(T)-\mathds{1}_{\Omega(r)}\|^2_{L^2_c}+\alpha\|f\|^2_{L^2(0,2T)}+\beta|f|^2_{H^1(0,2T)}\right)
		\]
		has a unique minimizer 
		$$f_{\alpha,\beta,r} = (\mathcal{P}_rK\mathcal{P}_r+\alpha I+ \beta S)^{-1}\mathcal{P}_rB1.$$
		Moreover, the map $r\mapsto f_{\alpha,\beta,r}$ belongs to $C([0,  \tau(1)];L^2(0,2T))$, and as $\alpha\to +0$ and $\beta\to +0$ one has uniform convergence
		\[
		\lim_{\alpha,\beta\to +0}\,\sup_{r\in [0,  \tau(1)]}\left\|u^{f_{\alpha,\beta,r}}(T)- \mathds{1}_{\Omega(r)}\right\|_{L^2_c} = 0.
		\]}
\end{Proposition}

Proposition~\ref{2mainres} refines Theorem~4 of \cite{korpela_lassas_oksanen} in two respects. The regularization combines an $L^2$ penalty with an $H^1$ seminorm penalty, which makes the minimization well posed in $H(r)$ and yields controls of higher regularity. More importantly, the convergence of $u^{f_{\alpha,\beta,r}}(T)$ to $\mathds{1}_{\Omega(r)}$ is uniform in $r\in[0,\tau(1)]$, whereas the pointwise convergence in $r$ obtained in \cite{korpela_lassas_oksanen} would in principle require regularization parameters depending on $r$. The uniformity, whose proof relies on the continuity of $r\mapsto f_{\alpha,\beta,r}$, means that one pair $(\alpha,\beta)$ serves for the whole family of domains of influence. This is what makes the numerical reconstruction of Section~\ref{compimpl} practical: the travel-time volume $V(r)$ is computed on a grid of values of $r$ with the regularization parameters fixed once and for all, rather than tuned separately for each $r$.

The two propositions above are the ingredients of the first of our main results, Theorem~\ref{3mainres}, which asserts the bijectivity of the forward spectral map $\Phi:\mathcal{C}\to\mathcal{D}$ and thereby the unique solvability of the inverse spectral problem.

Section~\ref{stabsec} is devoted to the stability of the spectral correspondence, where the remaining main results are proved. In the forward direction, Proposition~\ref{4mainres} states that $\Phi$ is Lipschitz continuous; in the inverse direction, Theorem~\ref{5mainres} and Corollary~\ref{6mainres} state that $\Phi^{-1}$ is H\"older continuous.

Finally, in Section~\ref{compimpl} we discuss the computational implementation. A finite element discretization of the forward problem produces the spectral data $\{\lambda_n(0)\}_{n=1}^L$ and $\{\lambda_n(\omega)\}_{n=1}^L$ for a prescribed wave speed $c(x)$ and a fixed $\omega\in(0,1)$. From these data we estimate the travel time $\tau(1)$, which fixes the observation time $T$, and build a truncated Neumann-to-Dirichlet map. The boundary control method then yields the volumes of the domains of influence, from which $c(x)$ is recovered. In Section~\ref{numres} we present numerical experiments. We verify, with an independent finite difference solver, that the computed boundary sources produce waves close to the indicators of the domains of influence. We then reconstruct smooth and piecewise constant wave speeds and quantify the errors. Finally, we study the sensitivity of the reconstruction to noise in the spectral data.

\medskip

\section{Impedance Spectral problem}\label{isp}
\subsection{Solutions to the impedance spectral problem}\label{sttisp}
Consider the spectral Sturm-Liouville problem with a fixed Robin parameter $\omega\in [0,1)$:
$$(\mathcal{S}_\omega)\; : \; \begin{cases}
	- c^2\, \partial_x^2 \phi = \lambda \phi , & \\
	\partial_x \phi(0)+\omega\phi(0) =0 , & \\
	\phi(1)=0, &
\end{cases}\text{ in }\Omega. $$
In this section, we simply write $\lambda$ omitting argument $\omega$ and we characterize all pairs $(\lambda,\phi)\in \R\times L^2(\Omega)$ satisfying $(\mathcal{S}_\omega)$ in the weak sense, as defined below. When we need to refer to the forward Sturm-Liouville problem (given $\lambda\in\R$, find $\phi$), we denote the problem by $(\mathcal{S}_{\omega,\lambda})$ and call it forward impedance spectral problem. 

\medskip

To begin with, we define the vector subspace of $H^1(\Omega)$ with the linear constraint $v(1)=0$ by
$$V=\{v\in H^1(\Omega): v(1)=0\}.$$
The subspace $V\subset H^1(\Omega)$ is a Hilbert space with respect to the inherited $H^1(\Omega)$-inner product.

\begin{definition}\label{defweak}
	We define the bilinear form $B\,:\, V\times V\to \mathbb{R}$ by
	$$B(\phi,v)=-\omega \phi(0) \overline{v(0)}+\int_\Omega \partial_x \phi\, \overline{\partial_xv}\,dx ,\qquad \phi,v\in V.$$
	We say that $\phi\in V$ is a weak solution to the problem $(\mathcal{S}_{\omega,\lambda})$ given $\lambda\in\R$ if and only if the following integral identity holds:
	$$ B(\phi,v)=\lambda( \phi, v)_{L^2_c},\qquad v\in V. $$
\end{definition}

The bilinear form $B\,:\, V\times V\to \mathbb{R}$ is symmetric, bounded and coercive (elliptic) for $\omega\in [0,1)$. By Lax-Milgram theorem, there is the linear operator $A\,:\, L^2(\Omega)\to V$ that maps $f\mapsto \phi$, where $\phi\in V$ is the unique solution to the problem $B(\phi, v) = (f, v)_{L^2_c}$, $v\in V$. In other words,
$$B(A[f], v) = ( f, v)_{L^2_c},\qquad v\in V, f\in L^2(\Omega),$$
Moreover, $A$ as an $L^2_c\to L^2_c$ operator is injective, self-adjoint, positive-semidefinite, and compact.

The spectral theorem applied to the infinite rank operator $A\in \mathscr{B}(L^2_c)$, the Banach algebra of bounded linear operators on $L^2_c$, yields the eigenfunction expansion of $A$. This result is also known as the Hilbert-Schmidt theorem, see e.g., Section~16.6 in~\cite{fitzpatrick} or Theorem~8.94 in~\cite{renardy}. Clearly, positive-semidefiniteness of $A$ entails that all eigenvalues of $A$ are real and non-negative. Upon combining these we obtain the following proposition.

\begin{theorem*}[The Hilbert-Schmidt theorem]\label{spectralth}
	\textit{The spectrum $\sigma(A)$ is countable (discrete), compact, real, positive. Moreover, $0$ belongs to $\sigma(A)$ and the point spectrum $\sigma_p(A)=\sigma(A)\setminus\{0\}$ consists only of eigenvalues $\{\lambda_n\}_{n\in\N}$ of $A$ satisfying $\lambda_1\geqslant \lambda_2 \geqslant \lambda_3\geqslant \dots > 0$ and $\lambda_n\to 0$ as $n\to +\infty$. There is an orthonormal basis for $L^2_c$ of corresponding eigenfunctions $\{\phi_n\}_{n\in\N}\subset V$ such that
		$$A = \sum\limits_{n=1}^{\infty} \lambda_n (\cdot ,\phi_n )_{L^2_c}\, \phi_n .$$
		The series on the right hand side converges in the operator norm of $\mathscr{B}(L^2_c)$.}
\end{theorem*}

Now, we define a bounded linear operator $P\,:\, V\to V^*$ by setting
$$\langle P[\phi] ,v\rangle =B(\phi, v),\qquad \phi,v\in V.$$
Remark that $A$ is essentially an inverse of $P$, since $\langle PA[f], v\rangle = ( f, v)_{L^2_c}$ for any $v\in V$, $f\in L^2(\Omega)$. Furthermore, the operator $P$ is a topological isomorphism due to the boundedness and coercivity of $B$, namely, there are $\alpha_1,\alpha_2>0$,
$$\alpha_1\|\phi\|_{H^1(\Omega)}\leqslant \|P[\phi]\|_{V^*}\leqslant \alpha_2\|\phi\|_{H^1(\Omega)},\qquad\phi\in V.$$

\begin{definition}
	The spectrum of the operator $P$ is defined by
	$$\sigma(P) = \C\setminus \{\lambda\in\C\mid P-\lambda I\,:\, V\to V^*\text{ is a topological isomorphism}\}.$$
	A spectrum element $\lambda\in \sigma(P)$ is called an \textit{eigenvalue of $P$} if the forward problem $(\mathcal{S}_{\omega,\lambda})$ has a nontrivial weak solution. In this case, we say that this solution is an eigenfunction corresponding to $\lambda$. The set of all eigenvalues of $P$ is called point spectrum of $P$ and denoted by $\sigma_{p}(P)$.		
\end{definition}

Next result is an application of a well-known fact about eigenvalues of symmetric elliptic operators, see Theorem~1 in Section~6.5 in~\cite{davies}.

\begin{Proposition}\label{eigveigf}
	\textit{The point spectrum of $P$ consists of reciprocals of eigenvalues of $A$:
		$$\sigma_p(P)=\{\lambda >0 : \lambda^{-1}\in\sigma_p(A)\}.$$
		Moreover, for any $\lambda\in\sigma_p(P)$, the set of all weak solutions to $(\mathcal{S}_{\omega,\lambda})$ is exactly the eigenspace $\mathrm{Ker}(A-\lambda^{-1}I)\subset V$. }
\end{Proposition}
\begin{proof}
	If $\lambda = 0$, then $P-\lambda I = P$ is a topological isomorphism, so $0\not\in\sigma(P)$.
	
	Suppose $\lambda\ne 0$. Let $\phi\in V\setminus\{0\}$ be a nontrivial weak solution to $(\mathcal{S}_{\omega,\lambda})$, that is, $B(\phi,v) =\lambda ( \phi,v )_{L^2_c}$ for any $v\in V$. Then, we get
	$$ B(\phi,v) =\lambda ( \phi,v )_{L^2_c} = \lambda B(A[\phi],v) ,\qquad v\in V.$$
	It entails that $B(\phi-\lambda A[\phi],v) = 0$ for any $v\in V$. Fix $v=\phi-\lambda A[\phi]$, then the coercivity estimate implies $0=B(v,v)\geqslant C\|v\|_{H^1(\Omega)}^2$, which, in turn, gives   
	\begin{equation*}\label{aphi_eig}
		A[\phi]=\frac{1}{\lambda}\phi.
	\end{equation*}
	
	For other direction, pick an eigenfunction $\phi\in V\setminus\{0\}$ of $A$ corresponding to an eigenvalue $\lambda\ne 0$ and repeat the steps above backwards. Thus, a function $\phi\in V\setminus\{0\}$ is a weak solution to $(\mathcal{S}_{\omega,\lambda})$ if and only if it is an eigenfunction of $A$ corresponding to the eigenvalue $\lambda^{-1}$.
	
	Now, pick $\lambda\in\sigma_p(P)$, so $\lambda^{-1}\in\sigma_p(A)$. Then, clearly, $\mathrm{Ker}(A-\lambda^{-1}I)\ne\{0\}$. This implies that any $\phi\in\mathrm{Ker}(A-\lambda^{-1}I)$ is a weak solution to $(\mathcal{S}_{\omega,\lambda})$.
\end{proof}

Consider the Gelfand triple $V\hookrightarrow L^2_c \hookrightarrow V^*$. Note that $V$ is reflexive, since $H^1(\Omega)$ is reflexive and every closed subspace of a reflexive space is reflexive. Of course, $V$ is dense in $L^2_c$, and also $L^2_c$ is isomorphic to $(L^2_c)^*$. So $L^2_c$ is dense in $V^*$. In this case the Gelfand triple is called dense Gelfand triple and the embeddings are denoted as
\begin{equation}\label{gelfand}
	V\overset{\text{dense}}{\hookrightarrow} L^2_c \overset{\text{dense}}{\hookrightarrow} V^*.
\end{equation}
Identifying $L^2_c$ as the pivot space in the dense Gelfand triple, the dual pairing $\langle\cdot,\cdot\rangle$ is the unique continuous extension of the $L^2_c$ the inner product.

\begin{definition}\label{deffr}
	We define the Friedrichs extension of $P$ by 
	$$P_F[\phi]=P[\phi],\quad \phi\in \mathrm{Dom}(P_F),$$
	where $\mathrm{Dom}(P_F)=\{\phi\in V: P[\phi]\in L^2_c\}$ is the domain of $P_F$.
\end{definition}

Define the point spectrum of $P_F$ by
$$\sigma_p(P_F)=\C\setminus\{\lambda\in\C:\mathrm{Ker}(P_F-\lambda I)=0\},$$
its elements we call eigenvalues of $P_F$. We say that a nonzero element $\phi\in\mathrm{Dom}(P_F)$ satisfying $P_F[\phi]=\lambda\phi$ in $L^2_c$ is called eigenfunction of $P_F$ corresponding to eigenvalue $\lambda\in\sigma_p(P_F)$. The set of all eigenfunctions of $P_F$ corresponding to the same eigenvalue $\lambda$, together with zero, is called an eigenspace of $P_F$ associated with that eigenvalue and is denoted by $E_\lambda(P_F)$.

Note that $A$ is the inverse of $P_F$, that is, $A=P_F^{-1}$ is a bijection map. Furthermore, one can show that the operator $P_F$ is self-adjoint and the point spectrum $\sigma_p(P_F)$ coincides with $\sigma_p(P)$ and $E_\lambda(P_F)=\mathrm{Ker}(A-\lambda^{-1}I)$ for any $\lambda\in\sigma_p(P_F)$.

Recall that for any self-adjoint operator the boundedness of the operator is equivalent to boundedness of its spectrum. This yields that the operator $P_F$ is unbounded, because its point spectrum is unbounded. Finally, we notice that $P_F$ is exactly the differential operator $-c^2(x)\partial_x^2$ and its domain corresponds to the spectral problem $(\mathcal{S}_\omega)$, namely
$$\mathrm{Dom}(P_F)=\{\phi\in H^2(\Omega):\phi(1)=0,\,\partial_x \phi(0)+\omega \phi(0)=0\}.$$

\medskip

\subsection{Analyticity of perturbations}
For each problem $(\mathcal{S}_\omega)$, $0\leqslant \omega <1$, Definition~\ref{defweak}, Definition~\ref{deffr} and the first representation theorem (see Chapter~VI, Section~2.1, Theorem~2.1, in~\cite{kato}) give rise to the bilinear form $B_\omega$ on $V$ and to the unique self-adjoint linear operator $P_\omega\,:\, \mathrm{Dom}(P_\omega)\to L^2_c$,
$$(P_\omega[\phi], v)_{L^2_c} = B_\omega(\phi,v),\qquad \phi\in\mathrm{Dom}(P_\omega),\, v\in V,$$
and $\mathrm{Dom}(P_\omega)=\{\phi\in H^2(\Omega):\phi(1)=0,\,\partial_x \phi(0)+\omega \phi(0)=0\}$. Each $P_\omega$ is an unbounded operator $-c(x)^2\partial_x^2$ corresponding to the impedance spectral problem $(\mathcal{S}_\omega)$ in the strong sense and we call $P_\omega$ an associated operator to the form $B_\omega$. The unperturbed form $B_0$, operator $P^0$, and spectral problem $(\mathcal{S}_0)$ for $\omega=0$ are denoted by $B$, $P$, and $(\mathcal{S})$. 

Now, recall that the spectral theorem, see e.g., Theorem~VIII.6 in \cite{reedsimon} or Chapter~XII, Section~2, in~\cite{dunford_schwartz} or Theorem~1.2, Section~1.5 in~\cite{pde6}, gives a one-to-one correspondence between self-adjoint operators and projection-valued measures. Let $\mathcal{M}$ be a collection of all Borel subsets of $\R$. Then, for each fixed $\omega\in [0,1)$ there is a unique projection-valued measure $\{\mathbb{P}_\omega(E)\}_{E\in\mathcal{M}}$ such~that
$$P_\omega = \int_{\sigma(P_\omega)}\lambda\,\mathbb{P}_\omega(d\lambda).$$
This projection-valued measure is called spectral measure. Any $\mathbb{P}_\omega(E)$ is an orthogonal projection in $L^2_c$ and for a fixed $f,g\in L^2_c$, the function $\mu^\omega_{f,g}(E)=(\mathbb{P}_\omega(E)f,g)_{L^2_c}\,:\,\mathcal{M}\to\C$ is a complex-valued Borel measure. We call $\mathbb{P}_\omega(E)$ the spectral projector.

Fix arbitrary $\omega\in [0,1)$ and $n\in\N$. Let $\lambda_n(\omega)\in\sigma(P_\omega)$ denote the $n$-th eigenvalue of $P_\omega$. Since the singleton $\{\lambda_n(\omega)\}\subset\R$ is a Borel (measurable) set, the spectral projector $\mathbb{P}_\omega(\{\lambda_n(\omega)\})$ is well-defined and we denote it by~$\mathbb{P}_{\omega}^{\, n}$. By the definition of spectral measure, we have $\mathrm{Ran}(\mathbb{P}_{\omega}^{\, n}) = E_{\lambda_n(\omega)}(P_\omega)$.

\begin{Lemma}\label{pndformula} \textit{
	Let $0\leqslant\omega <1$. The spectral projector $\mathbb{P}_{\omega}^{\, n}$ admits the Bochner integral formula:
	\begin{equation}\label{pnd1}
		\mathbb{P}_{\omega}^{\, n}[f]=\frac{1}{2\pi i}\,\int_{\Gamma_n}(P_\omega-z)^{-1}[f]\,dz,\qquad f\in L^2_c,
	\end{equation}
	where $\Gamma_n\subset\C$ is a contour oriented clockwise of sufficiently small diameter enclosing only $n$-th eigenvalue $\lambda_n(\omega)$ such that $\Gamma_n\cap\sigma_p(P_{\omega})=\varnothing$. }
\end{Lemma}
\begin{proof}
	To begin with, notice that the function $\lambda\mapsto (\lambda-z)^{-1}$, $z\in\Gamma_n$ is bounded on the spectrum $\sigma(P_\omega)$, since 
	$$\inf\left\{|\lambda-z|:\lambda\in\sigma(P_\omega),\,z\in\Gamma_n\right\}=\mathrm{dist}(\sigma(P_\omega),\Gamma_n)>0,$$
	where the positivity follows from the fact that $\sigma(P_\omega)$ is closed and $\Gamma_n$ is compact. Also, this function is Borel measurable due to continuity in $\lambda\in\R\setminus\Gamma_n$. Therefore, applying the functional calculus (see e.g., Chapter~XII, Section~2, Theorem~6 in~\cite{dunford_schwartz} or Theorem~VIII.5 in~\cite{reedsimon}) to the function $\lambda\mapsto (\lambda-z)^{-1}$ gives the identity 
	\begin{equation}\label{pm1upb}
		\|(P_\omega - z)^{-1}\|_{\mathscr{B}(L^2_c)}=\mathrm{dist}(\sigma(P_\omega),z)^{-1},\qquad z\in\Gamma_n,
	\end{equation}
	and, moreover, for any $f,g\in L^2_c$ and all $z\in\Gamma_n$,
	$$(P_\omega - z)^{-1}=\int\limits_{\sigma(P_\omega)}\frac{\mathbb{P}_\omega(d\lambda)}{\lambda-z},\quad ((P_\omega - z)^{-1}f,g)_{L^2_c} = \int\limits_{\sigma(P_\omega)}\frac{\mu^\omega_{f,g}(d\lambda)}{\lambda-z} .$$
	
	Now, we show that the Bochner integral in (\ref{pnd1}) is well-defined. The function $z\mapsto (P_\omega-z)^{-1}$ is continuous in operator norm of $\mathscr{B}(L^2_c)$ on $\Gamma_n$, which follows from the identity
	$$(P_\omega-z)^{-1}-(P_\omega-z_0)^{-1}=(z-z_0)(P_\omega-z)^{-1}(P_\omega-z_0)^{-1},\qquad z,z_0\in\Gamma_n.$$
	Whence, for any $f\in L^2_c$ the map $z\mapsto (P_\omega-z)^{-1}f$ is Bochner measurable. Furthermore, since the contour $\Gamma_n$ is compact, continuity of the resolvent on~$\Gamma_n$ implies uniform continuity in $z\in\Gamma_n$. In particular, the upper bound (\ref{pm1upb}) gives Bochner integrability of $z\mapsto (P_\omega-z)^{-1}f$ for a fixed $f\in L^2_c$ due to the following upper bound,
	$$\int\limits_{\Gamma_n}\|(P_\omega-z)^{-1}f\|_{L^2_c}\,dz\leqslant \frac{2\pi\|f\|_{L^2_c}}{\mathrm{dist}(\sigma(P_\omega),\Gamma_n)},\qquad f\in L^2_c.$$
	
	We are ready to prove the identity (\ref{pnd1}). Fix arbitrary $f,g\in L^2_c$. Since the linear functional $(\cdot,g)_{L^2_c}$ is bounded, it commutes with Bochner integrals, and, therefore,
	$$\left(\,\int_{\Gamma_n}(P_\omega-z)^{-1}f\,dz,g\right)_{L^2_c} = \int\limits_{\Gamma_n}((P_\omega-z)^{-1}f,g)_{L^2_c}\,dz = \int\limits_{\Gamma_n}\int\limits_{\sigma(P_\omega)}\frac{\mu^\omega_{f,g}(d\lambda)}{\lambda-z}\,dz.$$
	The modulus of the latter term is finite due to boundedness of the integrand on $\Gamma_n\times\sigma(P_\omega)$ and finiteness of $|\mu^\omega_{f,g}(\R)|=|(f,g)_{L^2_c}|<+\infty$. Hence, by the Fubini-Tonelli theorem one can exchange the order of integration. Cauchy's integral formula entails that $\int_{\Gamma_n}(\lambda-z)^{-1}\,dz$ equals $2\pi i$ whenever $\lambda$ is in the interior of the contour $\Gamma_n$ and $0$ otherwise. And again, since $\sigma(P_\omega)$ is discrete and only $\lambda_n(\omega)$ is inside $\Gamma_n$ by the definition of the contour, we obtain
	$$\frac{1}{2\pi i}\left(\,\int_{\Gamma_n}(P_\omega-z)^{-1}f\,dz,g\right)_{L^2_c}=\mu^\omega_{f,g}(\{\lambda_n(\omega)\})=(\mathbb{P}_{\omega}^{\, n}[f],g)_{L^2_c}.$$
	This concludes the lemma.
\end{proof}

By the definition of the Bochner integral of an operator-valued function,  $\left(\int_{\Gamma_n}(P_\omega-z)^{-1}\,dz\right) (f) = \int_{\Gamma_n}(P_\omega-z)^{-1}f\,dz$, $f\in L^2_c$. Therefore, we write
$$\mathbb{P}_{\omega}^{\, n} =\frac{1}{2\pi i}\int_{\Gamma_n}(P_\omega-z)^{-1}\,dz,$$
which is itself a well-defined Bochner integral for Banach algebra $\mathscr{B}(L^2_c)$. 

The spectral projector definition $\mathbb{P}_{\omega}^{\, n}$ provides no handle on the $\omega$~dependence. Nevertheless, Lemma~\ref{pndformula} gives a convenient formula to examine differentiability with respect to $\omega$, which, in turn, implies the following result.

\begin{Proposition}\label{lem1}
	\textit{For $n\in \N$, the spectral projector $\mathbb{P}_{\omega}^{\, n}$ is analytic in $\omega$ on the interval $[0,1)$.}
\end{Proposition}
\begin{proof}
	Let the trace operator $\gamma\,:\,V\to V^*$ be defined via $\langle\gamma[\phi],v\rangle=\phi(0)v(0)$, so $\gamma$ is bounded. Due to Definition~\ref{deffr} the operators $P_\omega$ can be understood as operators $V\to V^*$ via $\langle P_\omega[\phi],v\rangle = B_\omega(\phi,v)$, $\phi,v\in V$. We notice that $P_\omega=P-\omega\gamma$ holds in $\mathscr{B}(V,V^*)$ and for any $z\in\C\setminus\sigma(P_\omega)$ the resolvent operator $(P_\omega-z)^{-1}$ belongs to $\mathscr{B}(V^*,V)$. 
	
	In the view of Chapter~VII, Section~4.2, Remark~4.6, in~\cite{kato}, we show that the resolvent $(P_\omega-z)^{-1}\,:\, V^*\to V$ is analytic in $\omega$ near any fixed $\omega_0\in [0,1)$, for any fixed $z\in\C\setminus\sigma(P_{\omega_0})$. Indeed, writing $P_\omega=P_{\omega_0}-(\omega-\omega_0)\gamma$,
	$$(P_\omega-z)^{-1}=\bigl(I-(\omega-\omega_0)(P_{\omega_0}-z)^{-1}\gamma\bigr)^{-1}(P_{\omega_0}-z)^{-1},$$
	and the Neumann series
	$$\bigl(I-(\omega-\omega_0)(P_{\omega_0}-z)^{-1}\gamma\bigr)^{-1}=I+\sum_{k\in\N}(\omega-\omega_0)^k\bigl((P_{\omega_0}-z)^{-1}\gamma\bigr)^k$$
	converges in $\mathscr{B}(V)$ for $|\omega-\omega_0|<\|(P_{\omega_0}-z)^{-1}\gamma\|^{-1}_{\mathscr{B}(V)}$. Whence, $(P_\omega-z)^{-1}$ inherits analyticity in a neighborhood of $\omega_0$ in $\C$ as a power series convergent in operator norm~$\mathscr{B}(V^*,V)$. Since $[0,1)$ is connected, on overlapping neighborhoods the power series agree by uniqueness of analytic continuation, see the Identity Principle in Chapter~II, Section~5, Proposition~5.7, in~\cite{mujica}. 
	
	Consider the composition $\imath (P_\omega-z)^{-1}\jmath$, where $\imath\,:\,V\hookrightarrow L^2_c$ and $\jmath\,:\,L^2_c\hookrightarrow V^*$ are continuous embeddings from the Gelfand triple (\ref{gelfand}). This operator coincides with the resolvent $(P_\omega-z)^{-1}\in\mathscr{B}(L^2_c)$, since both solve the same variational problem and the solution is unique by the Lax-Milgram theorem. Composing holomorphic operators with bounded operators preserves holomorphy, see Chapter II, Section~5, Exercise 5A in~\cite{mujica}, and it yields that for $z\in\C\setminus\bigcup_{\omega\in[0,1)}\sigma(P_\omega)$ the resolvent $(P_\omega-z)^{-1}$ depends analytically on $\omega$ on the whole interval $[0,1)$.
	
	Analyticity of $\mathbb{P}_{\omega}^{\, n}$ in $\omega$ is a consequence of Lemma~\ref{pndformula} and analyticity of the resolvent $(P_\omega-z)^{-1}$ under the Bochner integral sign. Indeed, fix $\omega_0\in [0,1)$ and choose a contour $\Gamma_n\subset \C\setminus\sigma(P_{\omega_0})$ enclosing $\lambda_n(\omega_0)\in\sigma(P_{\omega_0})$. Analyticity of the resolvent implies continuity of $\omega\mapsto\|(P_\omega-z)^{-1}\|_{\mathscr{B}(L^2_c)}$ on $\Gamma_n$. Hence, the radius of convergence of the Neumann series above, taken at $\omega_0$, is bounded from below by $\delta_n=\inf\{\|(P_{\omega_0}-z)^{-1}\gamma\|^{-1}_{\mathscr{B}(V)}:z\in\Gamma_n\}>0$, so there is an interval $I\subset[0,1)$ containing $\omega_0$, the same for all $z\in\Gamma_n$, on which $\Gamma_n\subset\C\setminus\sigma(P_\omega)$ and the resolvent $(P_\omega-z)^{-1}$ is analytic in $\omega$ for every $z\in\Gamma_n$. Shrinking $I$ if necessary, the relation (\ref{pm1upb}) and the compactness of $\Gamma_n$ give
	$$M=\sup\left\{\|(P_\omega-z)^{-1}\|_{\mathscr{B}(L^2_c)}:z
	\in\Gamma_n,\omega\in I\right\}< +\infty .$$
	
	Fix $f,g\in L^2_c$ and $z\in\Gamma_n$. Then, the function $\omega\mapsto ((P_{\omega}-z)^{-1}f,g)_{L^2_c}$ is analytic, since it is the evaluation of the analytic resolvent against the fixed bounded bilinear form $(\cdot,\cdot)_{L^2_c}$. Also, this function is uniformly bounded,
	$$|((P_{\omega}-z)^{-1}f,g)_{L^2_c}|\leqslant M \|f\|_{L^2_c}\|g\|_{L^2_c},\qquad \omega\in I,\,z\in\Gamma_n.$$
	That is, for each fixed $f$ and $g$, this function is dominated by an integrable function in $z\in\Gamma_n$. Hence, by the classical complex analysis we obtain that for all $f,g\in L^2_c$,
	$$(\mathbb{P}_{\omega}^{\, n}[f],g)_{L^2_c}=\frac{1}{2\pi i}\,\int\limits_{\Gamma_n}((P_\omega-z)^{-1}f,g)_{L^2_c}\,dz$$
	is analytic in $\omega\in I$. Furthermore, the map $\omega\mapsto\mathbb{P}_{\omega}^{\, n}$ is locally bounded in operator norm due to
	$$\|\mathbb{P}_{\omega}^{\, n}\|_{\mathscr{B}(L^2_c)}\leqslant \frac{|\Gamma_n|}{2\pi}\sup\left\{\|(P_\omega-z)^{-1}\|_{\mathscr{B}(L^2_c)}:z
	\in\Gamma_n\right\} \lesssim M,\qquad \omega\in I.$$
	A standard characterization of analytic operators yields that the spectral projector $\mathbb{P}_{\omega}^{\, n}$ is analytic in $I$. Finally, since the choice of $\omega_0$ is arbitrary, covering $[0,1)$ by overlapping intervals implies analyticity of the spectral projector $\mathbb{P}_{\omega}^{\, n}$ on all of $[0,1)$.
\end{proof}

Now, we prove an auxiliary lemma. 

\begin{Lemma}\label{uniqS}
	\textit{Let $0\leqslant\omega < 1$. For any $\lambda\in \sigma_p(P_\omega)$, $\mathrm{dim}\,E_\lambda(P_\omega)=1$, that is, all eigenvalues of $P_\omega$ are simple.}
\end{Lemma}
\begin{proof}
By the discussion following Definition~\ref{deffr}, $E_\lambda(P_\omega)=E_\lambda(P_F)$ and $\mathrm{Dom}(P_F)\subset H^2(\Omega)$. Let $\phi\in E_\lambda(P_F)$. By the one-dimensional Sobolev embedding $H^2(\Omega)\hookrightarrow C^1(\overline\Omega)$, the values $\phi(1)$ and $\phi'(1)$ are well defined, and $\phi(1)=0$ since $\phi\in V$. Moreover, the eigenvalue equation $P_F[\phi]=\lambda\phi$ is the differential equation of $(\mathcal{S}_{\omega,\lambda})$, a linear second-order equation with continuous coefficients on $\overline\Omega$, so by uniqueness for its Cauchy problem a solution with $\phi(1)=\partial_x\phi(1)=0$ vanishes identically.

Consider the linear map
$$E_\lambda(P_F)\to\R,\qquad\phi\mapsto\partial_x\phi(1).$$
If $\partial_x\phi(1)=0$, then $\phi(1)=\partial_x\phi(1)=0$ and so $\phi=0$. Hence the map is injective and $\dim E_\lambda(P_F)\leqslant1$. Since $\lambda$ is an eigenvalue, $\dim E_\lambda(P_\omega)=1$.
\end{proof}

Pick arbitrary $\lambda\in \sigma_p(P_\omega)$. Due to Lemma~\ref{uniqS}, there is $\phi\in V$ such that $\|\phi\|_{L^2_c}=1$ and both $\pm\phi\in E_\lambda(P_\omega)$ are solutions to the impedance spectral problem $(\mathcal{S}_{\omega,\lambda})$. Denote $\{v\in V:\|v\|_{L^2_c}=1,v(0)>0\}$ by $\widetilde{V}$. Hence, for any $\lambda\in \sigma_p(P_\omega)$ there is one and only one weak solution $\phi\in E_\lambda(P_\omega)\cap \widetilde{V}$ to the impedance spectral problem $(\mathcal{S}_{\omega,\lambda})$.

\begin{definition}
	Consider the family of impedance spectral problems $(\mathcal{S}_\omega)$ for $\omega\in [0,1)$. Lemma~\ref{uniqS} allows us to define for each $\omega$ the family of pairs 
	$$\{(\lambda_n(\omega),\phi_{n,\omega})\}_{n\in \N}\subset \R_+\times \widetilde{V},$$ 
	where $\lambda_n(\omega)\in\sigma_p(P_\omega)$ and $\phi_{n,\omega}\in E_{\lambda_n(\omega)}(P_\omega)\cap \widetilde{V}$ are an eigenvalue and the corresponding unique eigenfunction of $P_\omega$. Let us omit the index $0$ in pairs $\{(\lambda_n(0),\phi_{n,0})\}_{n\in\N}$ and write simply $\{(\lambda_n,\phi_n)\}_{n\in \N}$. Note that the elements of the spectrum $\sigma_p(P_\omega)$ are discrete with no accumulation point, meaning that for a fixed $\omega$ all eigenvalues $\lambda_n(\omega)$ are isolated. Moreover, for any $n\in\N$ and any $\omega\in[0,1)$ one has $\mathrm{Ran}(\mathbb{P}_{\omega}^{\, n}) = \mathrm{span}\{\phi_{n,\omega}\}$.
\end{definition}

Notice that $V$ does not depend on $\omega$, while $\mathrm{Dom}(P_\omega)$ clearly changes with $\omega$. So the family of sesquilinear forms $\{B_\omega\}_{\omega\in [0,1)}$ share the common domain $V$; yet this is not the case for the collection of associated operators $\{P_\omega\}_{\omega\in [0,1)}$. However, clearly, for any fixed $\phi, v\in V$, the function $\omega\mapsto B_\omega(\phi,v)$ depends linearly on $\omega$, hence analytically. This makes the collection $\{P_\omega\}_{\omega\in [0,1)}$ a holomorphic family of type (B), in the sense of Kato, see Chapter~VII, Section~4.2, Theorem~4.8, in~\cite{kato}. By Kato’s spectral perturbation theory (Chapter VII), the discrete eigenvalues $\lambda_n(\omega)$ and their spectral projections $\mathbb{P}_{\omega}^{\, n}$ depend analytically on $\omega$. Since all eigenvalues are simple and the corresponding eigenfunction is uniquely normalized by Lemma~\ref{uniqS}, the eigenfunction $\phi_{n,\omega}$ can be chosen to depend analytically on $\omega$. Nevertheless, we develop a self-contained re-derivation of those results of Kato in our concrete setting, so we proceed with the final result of this section as follows.

\begin{Proposition}\label{lem2}
	\textit{For each $n\in \N$, the eigenvalue $\lambda_n(\omega)$ and the eigenfunction $\phi_{n,\omega}$ are analytic in $\omega$ for $0\leqslant\omega <1$.}
\end{Proposition}
\begin{proof}
	Fix arbitrary $\omega_0\in [0,1)$. Because $\mathbb{P}^{\,n}_{\omega_0}$ is the nonzero spectral projector onto 1 dimensional eigenspace, we fix a function $v\in\widetilde{V}$ such that $\mathbb{P}^{\,n}_{\omega_0} [v]\ne 0$. By Proposition~\ref{lem1}, the map $p(\omega)=(\mathbb{P}_{\omega}^{\, n} [v],\mathbb{P}_{\omega}^{\, n} [v])_{L^2_c}$ is analytic in $\omega\in [0,1)$, since the function $v$ is independent of $\omega$. Notice that $p(\omega_0) > 0$. By continuity of $p$, there is an interval $I$ containing $\omega_0$ such that $p(\omega) > 0$ for any $\omega\in I$. Therefore, the function
	$$\omega\mapsto\frac{1}{\|\mathbb{P}_{\omega}^{\, n} [v]\|_{L^2_c}} = \frac{1}{\sqrt{p(\omega)}}$$
	is analytic on $I$, since a nonvanishing analytic function has a locally analytic logarithm and an analytic reciprocal. Moreover, Proposition~\ref{lem1} implies that $\mathbb{P}_{\omega}^{\, n}[v]$ is analytic in $\omega\in[0,1)$, since $v$ is independent of $\omega$. We define 
	$$\widetilde{\phi}_{n,\omega}=\frac{\mathbb{P}_{\omega}^{\, n} [v]}{\|\mathbb{P}_{\omega}^{\, n} [v]\|_{L^2_c}},\qquad n\in\N,\;\omega\in I.$$
	As the product of two analytic functions is analytic, it follows that $\widetilde{\phi}_{n,\omega}\in\mathrm{Ran}(\mathbb{P}_{\omega}^{\, n})$ is analytic on $I$. Also, for any $\omega\in I$, $\widetilde{\phi}_{n,\omega}\in\mathrm{span}\{\phi_{n,\omega}\}$ and $\|\widetilde{\phi}_{n,\omega}\|_{L^2_c}=1$. 
	
	Now we note that at $\omega=\omega_0$, either $\widetilde{\phi}_{n,\omega_0}=\phi_{n,\omega_0}$ or $\widetilde{\phi}_{n,\omega_0}=-\phi_{n,\omega_0}$. By the obtained above analyticity, $\widetilde{\phi}_{n,\omega}(0)$ has the same sign as $\widetilde{\phi}_{n,\omega_0}(0)$ on a subinterval $\widetilde{I}\subset I$. Without loss of generality, let $\widetilde{\phi}_{n,\omega_0}=\phi_{n,\omega_0}$. Thus, for any $\omega\in\widetilde{I}$, the eigenfunction $\widetilde{\phi}_{n,\omega}$ belongs to $E_{\lambda_n(\omega)}(P_\omega)\cap \widetilde{V}$. Recall that Lemma~\ref{uniqS} implies that for all $\omega\in\widetilde{I}$ there is exactly one eigenfunction of $P_\omega$ corresponding to $\lambda_n(\omega)$ in the energy class $E_{\lambda_n(\omega)}(P_\omega)\cap \widetilde{V}$. This entails  
	$$\widetilde{\phi}_{n,\omega}=\phi_{n,\omega},\qquad \omega\in\widetilde{I},$$
	which gives analyticity of $\omega\mapsto\phi_{n,\omega}$ on the subinterval $\widetilde{I}$. Finally, because the choice of $\omega_0\in [0,1)$ is arbitrary, we repeat this argument for any basepoint $\omega_0$ and get analyticity on overlapping intervals.
	
	Analyticity of the eigenvalues is a straightforward consequence of the following identity,
	$$\lambda_n(\omega)=B_\omega(\phi_{n,\omega},\phi_{n,\omega}),$$
	and the fact that for any fixed $\phi, v\in V$, the function $\omega\mapsto B_\omega(\phi,v)$ depends linearly on $\omega\in [0,1)$. 
\end{proof}

Now, the Gateaux derivatives of eigenvalues $\lambda_n(\omega)$ and normalized eigenfunctions $\phi_{n,\omega}$ with respect to $\omega$ are well-defined. For each $n\in \N$, we call an element $\frac{d}{d\omega}\phi_{n,\omega}\in V$, $\omega\in (0,1)$, the Gateaux derivative of $\phi_{n,\omega}$ with respect to $\omega$ if and only if
$$\left\|\frac{\phi_{n,\omega +\varepsilon} - \phi_{n,\omega}}{\varepsilon}-\frac{d\phi_{n,\omega}}{d\omega}\right\|_{H^1(\Omega)}\to 0\;\text{ as }\; {\varepsilon\to 0}.$$
The derivative of the $n$-th eigenvalue curve $\{\lambda_n(\omega)\}_{\omega\in [0,1)}$ is the right derivative with respect to $\omega$, 
$$\lambda^{\prime}_n(\omega)=\lim\limits_{\varepsilon\to 0}\frac{\lambda_n(\omega+\varepsilon) - \lambda_n(\omega)}{\varepsilon} .$$  
At the endpoint $\omega=0$, both Gateaux derivatives are naturally defined by the right limit of the same difference quotients as $\varepsilon\to +0$. 

\medskip

\subsection{Generic behaviour of eigenvalues}
The following result was first obtained in a more general case in \cite{kurylev_lassas_weder}, see Lemma~3.1. We adjust their proof to our 1 dimensional case.
\begin{Proposition}\label{gbeig}
	\textit{For each $n\in \N$ the following formula holds:}
	$$\lambda^{\prime}_n(\omega) =-|\phi_{n,\omega}(0)|^2,\qquad \omega\in (0,1).$$
\end{Proposition}
\begin{proof}
	Fix arbitrary $n\in \N$ and $\omega\in (0,1)$. The eigenfunction $\phi_{n,\omega}$ satisfies $(\mathcal{S}_{\omega,\lambda_n(\omega)})$ in the weak sense, 
	$$( \partial_x\phi_{n,\omega} , \partial_x v)_{L^2(\Omega)} - \omega \phi_{n,\omega}(0)v(0) = \lambda_n(\omega) (\phi_{n,\omega}, v)_{L^2_c},\quad v\in V.$$
	Note that $\frac{d}{d\omega}\phi_{n,\omega}\in V$, so taking $v=\frac{d}{d\omega}\phi_{n,\omega}$ turns the identity above into
	\begin{equation}\label{eq1212}
		\left( \partial_x\phi_{n,\omega} , \partial_x \frac{d\phi_{n,\omega}}{d\omega}\right)_{L^2(\Omega)}  - \omega \phi_{n,\omega}(0)\frac{d\phi_{n,\omega}}{d\omega}(0) = \lambda_n(\omega) \left(\phi_{n,\omega},\frac{d\phi_{n,\omega}}{d\omega}\right)_{L^2_c}.
	\end{equation}
	Moreover, since $\phi_{n,\omega}\in \widetilde{V}$, one can pick $v=\phi_{n,\omega}$, which gives
	\begin{equation}\label{eq111}
		\| \partial_x\phi_{n,\omega} \|^2_{L^2(\Omega)} - \omega |\phi_{n,\omega}(0)|^2 = \lambda_n(\omega) \|\phi_{n,\omega}\|^2_{L^2_c}.
	\end{equation}
	
	Now, let $b^\omega[v]=B_\omega(v,v)$ be a quadratic form on $V$; this form is bounded, since $B_\omega$ is bounded. Also, the quadratic form $q^\omega=\lambda_n(\omega)\|\cdot\|^2_{L^2_c}$ is bounded. We calculate the needed Gateaux derivatives; for any $h\in V$, $$D_h\left(b^\omega[v]\right) = 2(\partial_xv , \partial_x h)_{L^2(\Omega)}  -  2\omega v(0)h(0),$$ $$D_h\left(q^\omega[v]\right) = 2\lambda_n(\omega)(v,h)_{L^2_c}.$$ 
	Finally, the functionals $\omega\mapsto b^\omega[\phi_{n,\omega}]$ and $\omega \mapsto q^\omega[\phi_{n,\omega}]$ are analytic in $\omega$ as a composition of analytic functions due to Proposition~\ref{lem2}. The identity (\ref{eq111}) turns into $b^\omega[\phi_{n,\omega}]= q^\omega[\phi_{n,\omega}]$. Differentiate it in the Gateaux sense applying the chain rule,	
	$$2\left(\partial_x\phi_{n,\omega} , \partial_x \frac{d\phi_{n,\omega}}{d\omega}\right)_{L^2(\Omega)}  -  |\phi_{n,\omega}(0)|^2 - 2\omega \phi_{n,\omega}(0)\frac{d\phi_{n,\omega}}{d\omega}(0) =$$
	$$=\lambda^{\prime}_n(\omega) \|\phi_{n,\omega}\|^2_{L^2_c} + 2\lambda_n(\omega) \left(\phi_{n,\omega}, \frac{d\phi_{n,\omega}}{d\omega}\right)_{L^2_c} .$$
	Hence, upon combining the latter, identity (\ref{eq1212}) multiplied by $2$, and the fact that $\|\phi_{n,\omega}\|_{L^2_c}=1$, we obtain the required formula.
\end{proof}

Now, the unperturbed spectral problem $(\mathcal{S}_0)$ for $\omega=0$ is the standard Sturm-liouville problem with Neumann boundary condition, namely
$$(\mathcal{S})\; : \; \begin{cases}
	- c^2\, \partial_x^2 \phi = \lambda \phi , & \\
	\partial_x \phi(0) =0 , \; \phi(1)=0, & 
\end{cases}\text{ in }\Omega. $$
Recall that we omit the index $0$ in pairs $\{(\lambda_n(0),\phi_{n,0})\}_{n\in\N}$ and write simply $\{(\lambda_n,\phi_n)\}_{n\in \N}$. Taking the limit $\omega \to +0$ in Proposition~\ref{gbeig} entails that one can recover the squared boundary value of an eigenfunction of $(\mathcal{S})$ at $x=0$ by calculating the derivative at $\omega=0$ of the corresponding eigenvalue curve for the impedance spectral problems $(\mathcal{S}_{\omega})$. 

\begin{Corollary}\label{gbeig0}
	\textit{For any $n\in \N$, one has $|\phi_n(0)|^2=-\lambda^{\prime}_n(0)$.}
\end{Corollary}
\begin{proof}
	Proposition~\ref{lem2} gives analyticity of eigenvalues and eigenfunctions. Since the trace operator $\gamma(v)=v(0)$, $v\in H^1(\Omega)$, is bounded, we get that both $\omega\mapsto\phi_{n,\omega}(0)$ and $\omega\mapsto\lambda_n^{\prime}(\omega)$ are continuous on $[0,1)$. They agree on $(0,1)$ by Proposition~\ref{gbeig}, hence also at $\omega=0$.
\end{proof}

\medskip

\section{Wave equation: forward problem}\label{wefpmain}
\subsection{Energy estimates and finite speed of propagation}\label{fpee}
Let $T>0$ and let $u\,:\,[0,T]\times\mathbb{R}\to\mathbb{R}$ and $f\in L^2(0,T)$. Consider the $(1+1)$-dimensional wave equation problem:
$$(\widetilde{\mathcal{W}}_f)\; : \; \begin{cases}
	\partial_t^2 u - c^2(x) \partial_x^2 u = 0, & \text{in }(0,T)\times\Omega,\\
	\partial_x u(t,0) = f(t),&t\in (0,T), \\
	u(t,1)=0,&t\in (0,T),\\
	u(0,x)=\partial_tu(0,x) = 0,& \text{in }\Omega.
\end{cases} $$
In this section, we show the energy estimates and the finite speed of propagation for the wave equation above. 

\medskip

To begin with, we note that we usually write $u(t)=u(t,\cdot)$, meaning that this is a function of the spatial variable. In the view of L. C. Evans' monograph~\cite{evans}, one can define a weak solution of the wave equation $(\widetilde{\mathcal{W}}_f)$ via the Gelfand triple (\ref{gelfand}). We say a function $u\in L^2(0,T;V)$ such that $\partial_tu\in L^2(0,T;L^2_c)$ and $\partial_t^2u\in L^2(0,T;V^*)$ is a weak solution of $(\mathcal{W}_f)$ provided that for almost every $t\in[0,T]$,
\begin{equation}\label{wswe}
	\langle\partial_t^2u(t), v\rangle + B(u(t),v)=-f(t)v(0),\qquad v\in V,
\end{equation}
with the initial conditions $u(0)=0$ and $\partial_tu(0)=0$ hold in $L^2_c$. Note that $B(v_1,v_2)=(\partial_xv_1,\partial_xv_2)_{L^2(\Omega)}$, $v_1,v_2\in V$; we take $\omega=0$ in Definition~\ref{defweak}. Furthermore,  $\partial_tu$ and $\partial_t^2u$ are understood as Sobolev time derivatives of $u$ in the Bochner sense, that is, as distributional derivatives in time that happen to be represented by an $L^2$ in time Banach-valued function.

By the Hilbert-Schmidt theorem, Theorem~\ref{spectralth}, and Lemma~\ref{uniqS}, the eigenfunctions $\{\phi_n\}_{n\in \N}\subset V$ of the spatial operator $-c^2(x)\partial_x^2$ form a complete orthonormal basis in $L^2_c$, because the sign convention at $0$ picks exactly one eigenfunction with $\phi_n(0)>0$. For each $N\in\N$ and almost every $t\in[0,T]$, the Galerkin approximations is defined by
$$u^N(t) = \sum_{n=1}^N a_n^N(t)\phi_n,$$ 
where for each $k\in\{1,\dots,N\}$ the coefficient functions $a_n^N(t)$ are chosen to satisfy 
$$(\partial_t^2u^N(t), \phi_k)_{L^2_c} + B(u^N(t),\phi_k)=-f(t)\phi_k(0),\quad \text{a.e. }t\in [0,T],$$
and $(u^N(0),\phi_k)_{L^2_c} = (\partial_tu^N(0),\phi_k)_{L^2_c}=0$. 

The classical PDE theory gives that for any $f\in L^2(0,T)$ there is a unique weak solution
$$u\in L^\infty(0,T;V)\cap H^1(0,T;L^2_c),\quad \partial_t^2u\in L^2(0,T;V^*),$$
of the wave equation $(\widetilde{\mathcal{W}}_f)$ satisfying
$$\|u\|_{L^\infty(0,T;V)}+\|u\|_{H^1(0,T;L^2_c)}+\|\partial_t^2u\|_{L^2(0,T;V^*)}\lesssim\|f\|_{L^2(0,T)}.$$
In particular, the solution operator $U\,:\, L^2(0,T)\to C([0,T];L^2_c)$, $f\mapsto u$, is a bounded linear operator. Well-posedness of the problem $(\widetilde{\mathcal{W}}_f)$ is a straightforward consequence of the latter. Moreover, it implies that the weak solution is stable in the sense that the solution operator $U$ continuously depends on the given data $f$.

Further, we improve regularity for the weak solution of $(\widetilde{\mathcal{W}}_f)$ and show that the weak solution is a strong solution. For this purpose, we define the subspace of Neumann data displaying the natural compatibility condition at $t=0$:
$$H^1_{\diamond}(0,T)=\{f\in H^1(0,T)\mid f(0)=0\}.$$
This is clearly a closed subspace of $H^1(0,T)$ as it is the kernel of the bounded trace operator evaluating at the origin in time. We say that a function $u\in H^2((0,T)\times\Omega)$ is a \textit{strong solution of} $(\widetilde{\mathcal{W}}_f)$ if the wave equation holds almost everywhere in the interior, and boundary and initial conditions hold in the trace sense. 

Let $f\in H^1_{\diamond}(0,T)$, then there is a unique strong solution $u\in H^2((0,T)\times\Omega)$ of the problem $(\widetilde{\mathcal{W}}_f)$ such that
$$\|u\|_{H^2((0,T)\times\Omega)}\lesssim\|f\|_{H^1(0,T)}.$$
In particular, the solution operator $U\,:\, H^1_{\diamond}(0,T)\to H^2((0,T)\times\Omega)$, $f\mapsto u$, is a bounded linear operator. Since for every $f\in H^1_{\diamond}(0,T)$ the strong solution exists and is unique, it coincides with the weak solution. Below we therefore omit mentioning which of the two notions is meant.

\begin{Remark}
Assume that the wave speed function $c$ belongs to $C^\infty(\R)$ and let the Neumann boundary data $f$ be a smooth function such that it is flat at the origin $t=0$:
$$f(0)=f'(0)=f''(0)=\dots=f^{(k)}(0)=\dots=0,\qquad k\in\N.$$
Following the classical higher-order regularity theory for second-order hyperbolic PDEs, see e.g., monograph \cite{evans} of L.~C.~Evans, one can show that the solution $u$ belongs to $C^\infty([0,T]\times\overline{\Omega})$. 
\end{Remark}

Finally, the finite speed of propagation holds. Recall some definitions from Subsection~\ref{potpamr}. The travel time distance is $\tau(x)=d(0,x)$. The domain of influence is defined by $\Omega(r)=\{x\in\overline{\Omega}:\tau(x)\leqslant r\}$, $r\in [0,T]$. Since $\tau$ is strictly increasing, its inverse $\chi\,:\,[0,\tau(1)]\to\overline{\Omega}$ is a strictly increasing $C^1$ function with $\chi(0)=0$. Differentiating $\tau(\chi(t))=t$ gives the characteristic equation $\chi'(t)=c(\chi(t))$. Let $f\in H(r)$. Then, the solution $u$ of $(\widetilde{\mathcal{W}}_f)$ satisfies 
$$\supp u(t)\subset\Omega(t-(T-r)),\qquad t\in [T-r,T].$$
In particular, the latter tells us that for any $r\in [0,T]$ the support of the solution $u$ of $(\widetilde{\mathcal{W}}_f)$, $\supp f\subset [T-r,T]$, at level $t=T$ is contained in $\Omega(r)$. This is schematically illustrated in Figure~\ref{fig:figure1}, where the light-grey shaded region represents the corresponding domain of influence.
\begin{figure}[H]
	\centering	
	\begin{tikzpicture}[scale=0.45]
		\fill[gray!20] (0,-2) -- (4,3) -- (0,3) -- cycle;
		\draw[->] (-0.5,-4) -- (8,-4) node[right] {$x$};
		\draw[->] (0,-4.5) -- (0,4) node[above] {$t$};
		\node [below left] at (0,-4) {$(0,0)$};
		\draw [] (6.75,3) -- (-0.25, 3) node [left]  {$t=T$};
		\draw [] (0.25, -2) -- (-0.25,-2) node [left] {$t=T-r$};
		\draw [dashed] (4,3) -- (4, -4.25) node [below]  {$x=\chi(r)$};
		\draw [] (6.75,3.5) -- (6.75, -4.25) node [below]  {$x=1$};
		\draw [dashed] (0,-2) -- (4, 3) ;
		\draw[decorate, decoration={brace}] 
		(0,3) -- (4,3) node[midway, above=1pt] {$\Omega(r)$};
	\end{tikzpicture}
	\caption{Illustration of the finite speed of propagation of waves for Neumann boundary data $f$ supported in $[T-r,T]$.}
	\label{fig:figure1}
\end{figure}
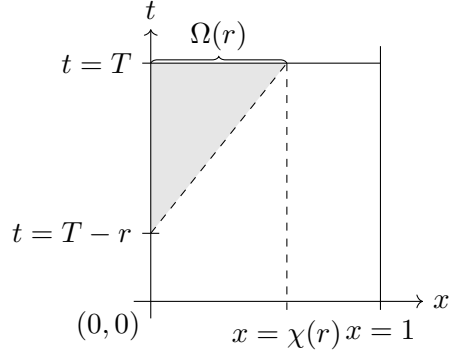

The geometry is purely schematic and is not meant to suggest constant propagation speed. In general, the resulting domain of influence need not be linear in time and may exhibit curved boundaries depending on the underlying wave operator.

\medskip

\subsection{Representation of the Neumann-to-Dirichlet map}\label{gboep2}
The Neumann-to-Dirichlet map $\Lambda\,:\,L^2(0,T)\to L^2(0,T)$ is defined by
$$\Lambda\,:\, f \mapsto u|_{x=0},\qquad f\in L^2(0,T),$$
where $u$ is the unique solution of the wave equation $(\widetilde{\mathcal{W}}_f)$. Taking into account well-posedness of the wave equation with boundary control yields that the Neumann-to-Dirichlet map $\Lambda$ is bounded. 

We now prove the representation formula for $\Lambda$ from Proposition~\ref{1mainres}, which is essentially a combination of Lemma~3.6 from the monograph \cite{katchalov_kurylev_lassas} and Corollary~\ref{gbeig0}. Recall that we want to show the following identity:
\begin{equation}\label{NDreprl}
	\Lambda[f] = \sum_{n\in \N}\lambda_n^\prime(0)\, S_{\lambda_n(0)}[f] , \qquad f\in L^2(0,T) .
\end{equation}
\begin{proof}
	Recall that eigenfunctions $\{\phi_n\}_{n\in\N}$ form a complete orthonormal basis in $L^2_c$. We claim that the coefficients $a_k^N(t)$ of the Galerkin approximations $u^N$ approximate the Fourier coefficients of $u$. Indeed, $u^N$ tends to $u$ in $C([0,T];L^2_c)$, and, therefore, for a fixed $k\in\{1,\dots,N\}$ the identity 
	$$(u^N(t),\phi_k)_{L^2_c}=\sum_{n=1}^N a_n^N(t)(\phi_n,\phi_k)_{L^2_c}=a_k^N(t),\qquad t\in [0,T],$$
	implies that $a_k^N(t)\to (u(t),\phi_k)_{L^2_c}$ as $N\to +\infty$ for any $t\in [0,T]$. Then, the Fourier series expansion of $u(t)$ for a fixed $t>0$ is given by
	$$u(t)=\sum_{n\in \N} u_n(t)\,\phi_n,\quad \text{with }\; u_n(t)=(u(t),\phi_n)_{L^2_c}. $$
	For any $t\in[0,T]$ the series converges in $L^2_c$. Moreover, the energy estimate of Section~\ref{fpee} applies to the Galerkin approximations $u_N(t)=\sum_{n\leqslant N}u_n(t)\phi_n$ uniformly in $N$, so $\|u_N(t)\|_{H^1(\Omega)}\lesssim\|f\|_{L^2(0,T)}$ for almost every $t\in[0,T]$. Together with the convergence in $L^2_c$ this gives $u_N(t)\rightharpoonup u(t)$ weakly in $H^1(\Omega)$ for almost every $t\in[0,T]$.
	
	Recall that the Dirac delta $\delta_0$ defined by $\langle\delta_0,v\rangle=v(0)$ is bounded in $V$. By~the definition of $\Lambda$ we see that $\Lambda[f](t)=\langle\delta_0,u(t)\rangle$, $f\in L^2(0,T)$. Since $u_N(t)\rightharpoonup u(t)$ weakly in $H^1(\Omega)$ and $u_N(t)\in V$ for every $N$, the weak limit stays in~$V$, and therefore
	$\langle\delta_0,u_N(t)\rangle\to\langle\delta_0,u(t)\rangle$ for a.e. $t\in[0,T]$. Thus, $\delta_0$ may be passed through the infinite sum, which gives
	\begin{equation}\label{NDrepr}
		\Lambda[f](t) = \sum_{n\in \N} u_n(t)\,\phi_n(0),\qquad \text{a.e. }t\in [0,T].
	\end{equation}
	
	Fix arbitrary $n\in\N$ and suppose $f\in H^1_{\diamond}(0,T)$. Section \ref{fpee} implies that $u_n\in H^2(0,T)$. By the definition of the weak solutions, the Fourier coefficients must obey
	$$\partial_t^2u_n(t)+ \lambda_n(0) u_n(t) = -f(t)\phi_n(0),\quad \text{a.e. }t\in [0,T],$$
	with the initial conditions $u_n(0)=\partial_tu_n(0)=0$. Whence, the standard ODE theory gives
	$$u_n(t) = -\phi_n(0)\int\limits_0^t  \frac{\sin(\tau\sqrt{\lambda_n(0)})}{\sqrt{\lambda_n(0)}}f(t-\tau)\,d\tau ,\qquad \text{a.e. }t\in [0,T].$$
	Substituting $u_n(t)$ in formula (\ref{NDrepr}) yields
	\begin{equation*}\label{lamh1d}
		\Lambda[f](t) = -\sum_{n\in \N}|\phi_n(0)|^2\, S_{\lambda_n(0)}[f](t) ,\qquad \text{a.e. }t\in [0,T],
	\end{equation*}
	and Corollary~\ref{gbeig0} entails the required formula~(\ref{NDreprl}).

	Finally, on $H^1_\diamond(0,T)$ the operator $\tilde{\Lambda}\,:\,f\mapsto\sum_{n\in\N}\lambda_n'(0)\, S_{\lambda_n(0)}[f]$ coincides with $\Lambda$ by the identity (\ref{NDreprl}) proved above for $H^1_\diamond(0,T)$ functions. It is therefore bounded with respect to $L^2(0,T)$ norm, that is, $\|\tilde{\Lambda}[f]\|_{L^2(0,T)}\lesssim\|f\|_{L^2(0,T)}$, $f\in H^1_\diamond(0,T)$. Since $H^1_\diamond(0,T)$ is dense in $L^2(0,T)$, it extends uniquely to an operator in $\mathscr{B}(L^2(0,T))$, and the extension is $\Lambda$. Indeed, given $f\in L^2(0,T)$, choose $\{f_k\}_{k\in\N}\subset H^1_\diamond(0,T)$ with $f_k\to f$ in $L^2(0,T)$; then
	$$\tilde{\Lambda}[f_k]=\sum_{n\in\N}\lambda_n'(0)\,S_{\lambda_n(0)}[f_k]=\Lambda[f_k]\to\Lambda[f] \qquad\text{in }L^2(0,T),$$
	both sides converging because $\Lambda$ is bounded. This concludes the proof.
\end{proof}

\medskip

\section{The Inverse Problem}\label{premip}	
\subsection{Preliminaries}\label{premip1}
Recall that $\Lambda$ is the Neumann-to-Dirichlet map. The Blagovestchenskii operators $B,J,R,K\in\mathscr{B}(L^2(0,2T))$ are defined by
$$Bf(t)=-\mathds{1}_{(0,T)}(t)\int\limits_t^T f(s)\,ds, \quad Jf(t)=\frac{1}{2}\,\mathds{1}_{(0,T)}(t)\int\limits_t^{2T-t} f(s)\,ds,$$
$$Rf(t)=f(2T-t),\quad	K=R\Lambda R J - J\Lambda.$$
We note that $K$ is compact. Indeed, the time-reversal operator \(R\) is an isometry. Also, \(J:L^2(0,2T)\to H^1(0,2T)\) is bounded, because the function $Jf$ vanishes on $[T,2T)$, is continuous at $t=T$ since $Jf(T)=0$, and satisfies 
\[
\partial_t (Jf)(t)=-\frac12 f(2T-t)-\frac12 f(t),\qquad t\in (0,T);
\]
whence $\|Jf\|_{H^1(0,2T)}\leqslant C_T\|f\|_{L^2(0,2T)}$ with $C_T>0$ depending only on $T$. The compactness of \(J:L^2(0,2T)\to L^2(0,2T)\) then follows from the compact
embedding \(H^1(0,2T)\hookrightarrow L^2(0,2T)\). Finally, since the Neumann-to-Dirichlet map \(\Lambda\) is bounded and the Blagove\v{s}\v{c}enski\u{\i} operator $K$ is defined by \(K=R\Lambda RJ-J\Lambda\), we obtain that \(K\) is compact.

Using the Blagove\v{s}\v{c}enski\u{\i} operators, we formulate the Blagove\v{s}\v{c}enski\u{\i} identities. These identities are a basic mechanism of the boundary control method: they allow one to compute inner products of interior wave fields in the geometry determined by $c(x)$, namely in $L^2_c$, using only boundary-time operators acting on the prescribed boundary inputs. Recall that given a Neumann boundary input $f$, we write $u^f(t)=u^f(t,\cdot)$ for the spatial profile of the solution of $(\mathcal{W}_f)$ at time $t\in [0,2T]$. 

\begin{Theorem}[The Blagove\v{s}\v{c}enski\u{i} identities]\label{blagid}
\textit{Let $T>0$. Then}
$$(u^f(T),u^h(T))_{L^2_c}=(f,Kh)_{L^2(0,2T)},\qquad f,h\in L^2(0,2T).$$
\textit{Moreover, if $0\leqslant r\leqslant\min\{\tau(1),T\}$, then}
$$(u^f(T),1)_{L^2_c} = (f, B1)_{L^2(0,2T)},\quad f\in H^1(0,2T),\, \mathrm{supp}f\subset [T-r,T].$$
\end{Theorem}
\begin{proof}[Sketch of proof]
The first identity is classical, see \cite{blag2,blag1}; we indicate the argument in the present setting. Let $w(t,s)=(u^f(t),u^h(s))_{L^2_c}$. Two integrations by parts in $x$, using $u^f(t,1)=u^h(s,1)=0$ and the Neumann conditions at $x=0$, give
$$(\partial_t^2-\partial_s^2)w=\Lambda[f](t)h(s)-f(t)\Lambda[h](s),\qquad t,s>0,$$
with $w=\partial_tw=0$ on $\{t=0\}$ and $w=\partial_sw=0$ on $\{s=0\}$. Solving this $(1+1)$-dimensional wave equation by d'Alembert and evaluating at $(T,T)$ integrates the right-hand side over the characteristic triangle $\{(t,s):s>t>0,\ t+s<2T\}$, which is precisely the kernel of $J$; collecting terms yields $K=R\Lambda RJ-J\Lambda$.

For the second identity, set $F(t)=(u^f(t),1)_{L^2_c}$. Integrating the equation~in~$x$,
$$F''(t)=\int\limits_0^1\partial_x^2u^f(t,x)\,dx=\partial_xu^f(t,1)-f(t).$$
Let $f\in H(r)$ with $r\leqslant\min\{\tau(1),T\}$; the bound $r\leqslant T$ ensures $[T-r,T]\subset[0,T]$, and the bound $r\leqslant\tau(1)$ is used below. By the finite speed of propagation of Section~\ref{fpee}, $u^f(s,\cdot)$ is supported in $\Omega(s-T+r)$ for $s\geqslant T-r$, and $u^f(s,\cdot)\equiv 0$ for $s\leqslant T-r$. Since $\Omega(r')=[0,\chi(r')]$ and $\chi$ is strictly increasing, for $s<T$ we have $s-T+r<r\leqslant\tau(1)$, hence $\chi(s-T+r)<1$ and $u^f(s,\cdot)$ vanishes on a neighbourhood of $x=1$. That is, the wave has not yet reached the endpoint $x=1$. In particular, $\partial_xu^f(s,1)=0$ for every $s<T$. Hence $F''=-f$ on $(0,T)$, and $F(0)=F'(0)=0$ gives
$$F(T)=-\int\limits_0^T(T-s)f(s)\,ds=(f,B1)_{L^2(0,2T)} .$$
\end{proof}

In the half-line model of \cite{korpela_lassas_oksanen} the boundary term  $\partial_xu^f(t,x)$ vanishes, since $u^f(t,\cdot)$ is compactly supported for every finite $t$; on the interval it is $\partial_xu^f(t,1)$, which is nonzero once the wave has reached the endpoint. The restriction $r\leqslant\tau(1)$ is caused by the Dirichlet condition at $x=1$: it is the exact condition under which the endpoint has not influenced the solution by the observation time $T$. 

Finally, we remark that the operator $K$ is self-adjoint and positive semidefinite on $L^2(0,2T)$. Indeed, the first Blagove\v{s}\v{c}enski\u{\i} identity implies $K=K^*$. Taking $h=f$ in the identity gives $(f,Kf)_{L^2(0,2T)}=\|u^f(T)\|^2_{L^2_c}\geqslant 0$.

\medskip

\subsection{Reconstruction of wave speed function}
In this section, we construct boundary sources $f$ that allows us to approximate domains of influence $\Omega(r)$ corresponding to a given wave speed function $c(x)$. For this purpose, we improve the regularization strategy proposed in 2016 in~\cite{korpela_lassas_oksanen}, see Theorem~4.

Let $0\leqslant r\leqslant T$ and recall that the spaces of boundary controls localized in time are denoted by
$$H(r)=L(r)\cap H^1(0,2T),\quad L(r)=\{f\in L^2(0,2T):\mathrm{supp}f\subset [T-r,T]\}.$$ 
For $\alpha,\beta >0$, we introduce the regularization norm $\|\cdot\|_{\alpha,\beta}$ defined by setting
$$\|f\|^2_{\alpha,\beta}=\alpha\|f\|^2_{L^2(0,2T)}+\beta|f|^2_{H^1(0,2T)},\qquad f\in H(r),$$ 
where $|f|_{H^1(0,2T)} = \|\partial_t f\|_{L^2(0,2T)}$. Clearly, the norm $\|\cdot\|_{\alpha,\beta}$ is equivalent to $\|\cdot\|_{H^1(0,2T)}$. Recall that the stiffness operator $S:H(r)\to (H(r))^*$ is defined~by
\begin{equation*}
	\langle S f, h\rangle =(\partial_t f,\partial_t h)_{L^2(0,2T)},\qquad f,h\in H(r),
\end{equation*}
and the orthogonal projector \(\mathcal{P}_r:L^2(0,2T)\to L^2(0,2T)\) is defined by $\mathcal{P}_r[f]=\mathds{1}_{[T-r,T]}f$, where \(\mathds{1}_{[T-r,T]}\) is the indicator function of \([T-r,T]\). Here the operator $(\mathcal{P}_rK\mathcal{P}_r+\alpha I+\beta S)^{-1}:(H(r))^*\to H(r)$, and $\mathcal{P}_rB1$ is regarded as an element of $(H(r))^*$ through the embedding $L^2(0,2T)\hookrightarrow (H(r))^*$.

For a functional $E:H(r)\to\R$, $f,h\in H(r)$, the directional derivative of $E$ at $f$ in the direction~$h$ is
$$D_hE(f)=\frac{d}{ds}E(f+sh)\Big|_{s=0}=\lim_{s\to0}\frac{E(f+sh)-E(f)}{s},$$
whenever the limit exists. If $f$ is a minimizer of $E$ and $D_hE(f)$ exists, then $D_hE(f)=0$, since $s\mapsto E(f+sh)$ attains its minimum at $s=0$.

Now, we are ready to prove Proposition~\ref{2mainres}, which is the second result of the paper. For convenience, we restate it.

\medskip
\noindent\textbf{Proposition~\ref{2mainres}.}
\textit{Let $0\leqslant r\leqslant \tau(1)$ and $\alpha,\beta >0$. Then the regularized optimization problem
\[
\min_{f\in H(r)}\left(\|u^f(T)-\mathds{1}_{\Omega(r)}\|^2_{L^2_c}+\|f\|^2_{\alpha,\beta}\right)
\]
has a unique minimizer
\[
f_{\alpha,\beta,r} = (\mathcal{P}_rK\mathcal{P}_r+\alpha I+ \beta S)^{-1}\mathcal{P}_rB1.
\]
Moreover, the map $r\mapsto f_{\alpha,\beta,r}$ belongs to $C([0, \tau(1)];L^2(0,2T))$, and as $\alpha\to +0$ and $\beta\to +0$ one has uniform convergence
\[
\lim_{\alpha,\beta\to +0}\,\sup_{r\in [0, \tau(1)]}\left\|u^{f_{\alpha,\beta,r}}(T)- \mathds{1}_{\Omega(r)}\right\|_{L^2_c} = 0.
\]}

\begin{proof}
\smallskip
\noindent\textit{Step 1: existence of a minimizer.}
Let $f\in H(r)$. Define the energy functional
$$E(f) = (f,Kf)_{L^2(0,2T)} - 2(f,B1)_{L^2(0,2T)}+\|f\|^2_{\alpha,\beta}.$$
By the finite speed of propagation, the solution $u^f$ to $(\mathcal{W}_f)$ is such that $u^f(T)$ is supported in $\Omega(r)$, see Section \ref{fpee}. Using the Blagove\v{s}\v{c}enski\u{\i} identities, Theorem~\ref{blagid}, we get
$$E(f)=\|u^f(T)-\mathds{1}_{\Omega(r)}\|^2_{L^2_c} - \|\mathds{1}_{\Omega(r)}\|^2_{L^2_c} +\|f\|^2_{\alpha,\beta}.$$
Pick a minimizing sequence $\{f_k\}_{k\in\N}\subset H(r)$, that is, $E(f_k)\to\inf\{E(f):f\in H(r)\}$ as $k\to +\infty$. This implies
$$\|f_k\|^2_{\alpha,\beta}\leqslant E(f_k) + \|\mathds{1}_{\Omega(r)}\|^2_{L^2_c},\qquad k\in\N,$$
so the sequence $\{f_k\}_{k\in\N}$ is bounded in $H(r)$. Passing to a subsequence, still denoted by $\{f_k\}_{k\in\N}$, we may assume that $f_k\rightharpoonup f_\infty$ weakly in $H(r)$ for some $f_\infty\in H(r)$.

Due to the energy estimate, the map $U_T : H^1(0,2T) \to V$, $U_T g = u^g(T)$, is bounded, since $H^1(0,2T)\hookrightarrow L^2(0,2T)$ is continuous. Recall that the embedding $V\hookrightarrow L^2_c$ is compact and, therefore, $U_T\,:\,H^1(0,2T)\to L^2_c$ is compact. Hence, passing to a further subsequence, still denoted by $\{f_k\}_{k\in\N}$, we have $u^{f_k}(T)\to u^{f_\infty}(T)$ in $L^2_c$ as $k\to +\infty$, and we get
$$E(f_\infty) = \lim\limits_{k\to +\infty}\|u^{f_k}(T)-\mathds{1}_{\Omega(r)}\|^2_{L^2_c} - \|\mathds{1}_{\Omega(r)}\|^2_{L^2_c} +\|f_\infty\|^2_{\alpha,\beta}.$$
The weak lower semicontinuity $\|f_\infty\|_{\alpha,\beta}\leqslant  \liminf\limits_{k\to +\infty}\|f_k\|_{\alpha,\beta}$ implies that
$$E(f_\infty)\leqslant\liminf\limits_{k\to +\infty} E(f_k)= \inf \{E(f):f\in H(r)\}.$$
This yields that $f_\infty$ is a minimizer of $E(f)$.	
	
	\smallskip
\noindent\textit{Step 2: uniqueness and the explicit formula.}
Let $f$ be a minimizer of $E$ and $h\in H(r)$. Since $E$ is the sum of a continuous quadratic form and a bounded linear functional on $H(r)$, the map $s\mapsto E(f+sh)$ is a polynomial of degree two in $s$, so $D_hE(f)$ exists and vanishes. Using $\mathcal{P}_rf=f$ and $\mathcal{P}_rh=h$, and the symmetry of $K$ given by the Blagove\v{s}\v{c}enski\u{\i} identities, a direct computation gives that $\tfrac12D_hE(f)$ equals
$$(h,\mathcal{P}_rK\mathcal{P}_rf)_{L^2(0,2T)}-(h,\mathcal{P}_rB1)_{L^2(0,2T)}+\alpha(h,f)_{L^2(0,2T)}+\beta\langle Sh,f\rangle,$$
so that $f$ solves $(\mathcal{P}_rK\mathcal{P}_r+\alpha I+\beta S)f=\mathcal{P}_rB1$ in $(H(r))^*$. The operator $\mathcal{P}_rK\mathcal{P}_r+\alpha I+\beta S:H(r)\to (H(r))^*$ is coercive, since positive semidefiniteness of $K$ yields
\begin{align*}
(K\mathcal{P}_rf,\mathcal{P}_rf)_{L^2(0,2T)}&+\alpha\|f\|^2_{L^2(0,2T)}+\beta\langle Sf,f\rangle \\
&\geqslant\alpha\|f\|^2_{L^2(0,2T)}+\beta|f|^2_{H^1(0,2T)}\geqslant\min\{\alpha,\beta\}\|f\|^2_{H^1(0,2T)} .
\end{align*}
By the Lax--Milgram theorem it is invertible. Hence every minimizer equals $(\mathcal{P}_rK\mathcal{P}_r+\alpha I+\beta S)^{-1}\mathcal{P}_rB1$, and, by Step~1, this is the unique minimizer~$f_{\alpha,\beta,r}=f_\infty$.
	
	\smallskip
\noindent\textit{Step 3: continuity in $r$.}
For $r\in[0,\tau(1)]$ and $f\in H(r)$ write
$$F_r(f)=\|u^f(T)-\mathds{1}_{\Omega(r)}\|^2_{L^2_c}+\|f\|^2_{\alpha,\beta}=E(f)+\|\mathds{1}_{\Omega(r)}\|^2_{L^2_c},$$ 
$$m_{\alpha,\beta}(r)=\min_{f\in H(r)}F_r(f)=F_r(f_{\alpha,\beta,r}) .$$
We show that $r\mapsto f_{\alpha,\beta,r}$ is continuous from $[0,\tau(1)]$ to $H^1(0,2T)$ and that $r\mapsto m_{\alpha,\beta}(r)$ is continuous on $[0,T]$; in particular, $(r\mapsto f_{\alpha,\beta,r})\in C([0,\tau(1)];L^2(0,2T))$. Note that $r\mapsto\mathds{1}_{\Omega(r)}$ is continuous from $[0,\tau(1)]$ to $L^2_c$, since $\Omega(r)=[0,\chi(r)]$ and $\chi$ is continuous.

Fix $r\in[0,\tau(1)]$ and a sequence $r_k\to r$ in $[0,\tau(1)]$; write $f_k=f_{\alpha,\beta,r_k}$ and $f_r=f_{\alpha,\beta,r}$. We first show $\limsup_km_{\alpha,\beta}(r_k)\leqslant m_{\alpha,\beta}(r)$. If $r_k\geqslant r$, then $H(r)\subset H(r_k)$ and we set $\tilde f_k=f_r$. If $r_k<r$, we set
$$\tilde f_k(t)=f_r\Big(T-\frac{r}{r_k}(T-t)\Big)\ \text{ for } t\in[T-r_k,T],\qquad\tilde f_k(t)=0\ \text{ otherwise};$$
the dilation maps $[T-r_k,T]$ onto $[T-r,T]$, so $\tilde f_k\in H(r_k)$. In both cases $\tilde f_k\to f_r$ in $H^1(0,2T)$ by continuity of dilations in $H^1$, hence $u^{\tilde f_k}(T)\to u^{f_r}(T)$ in $L^2_c$ by the energy estimate, and
$$m_{\alpha,\beta}(r_k)\leqslant F_{r_k}(\tilde f_k)\to F_r(f_r)=m_{\alpha,\beta}(r) .$$

Next, $\|f_k\|^2_{\alpha,\beta}\leqslant m_{\alpha,\beta}(r_k)\leqslant\|\mathds{1}_{\Omega(r_k)}\|^2_{L^2_c}$, so $\{f_k\}$ is bounded in $H^1(0,2T)$. Let $\bar f$ be the weak limit in $H^1(0,2T)$ of a subsequence, still denoted $\{f_k\}$. Since $\operatorname{supp}f_k\subset[T-r_k,T]$ and $r_k\to r$, we have $\operatorname{supp}\bar f\subset[T-r,T]$, that is, $\bar f\in H(r)$. By compactness of $U_T:H^1(0,2T)\to L^2_c$, $u^{f_k}(T)\to u^{\bar f}(T)$ in $L^2_c$, and by weak lower semicontinuity of $\|\cdot\|_{\alpha,\beta}$,
$$m_{\alpha,\beta}(r)\leqslant F_r(\bar f)\leqslant\liminf_kF_{r_k}(f_k)=\liminf_km_{\alpha,\beta}(r_k) .$$
Together with the previous paragraph, $m_{\alpha,\beta}(r_k)\to m_{\alpha,\beta}(r)$ and $F_r(\bar f)=m_{\alpha,\beta}(r)$, so $\bar f=f_r$ by uniqueness of the minimizer. Moreover, $\|u^{f_k}(T)-\mathds{1}_{\Omega(r_k)}\|^2_{L^2_c}\to\|u^{f_r}(T)-\mathds{1}_{\Omega(r)}\|^2_{L^2_c}$, so $F_{r_k}(f_k)\to F_r(f_r)$ gives $\|f_k\|_{\alpha,\beta}\to\|f_r\|_{\alpha,\beta}$; weak convergence together with convergence of norms yields $f_k\to f_r$ in $H^1(0,2T)$. As the limit does not depend on the subsequence, the whole sequence converges, which proves the claim.
	
	\smallskip
\noindent\textit{Step 4: convergence for fixed $r$.}
Since $\|u^{f_{\alpha,\beta,r}}(T)-\mathds{1}_{\Omega(r)}\|^2_{L^2_c}\leqslant m_{\alpha,\beta}(r)$, it suffices to show that $m_{\alpha,\beta}(r)\to0$ as $\alpha,\beta\to+0$ for each fixed $r\in[0,\tau(1)]$.

Note that $\{u^f(T):f\in H(r)\}$ is dense in $L^2(\Omega(r))$. Indeed, according to the proof of Theorem 4 in \cite{korpela_lassas_oksanen}, $\{u^f(T):f\in L(r)\}$ is dense in $L^2(\Omega(r))$, see also \cite{katchalov_kurylev_lassas} and \cite{tataru}. Since $H(r)$ is dense in $L(r)$ with respect to $L^2(0,2T)$, the image $U_T(H(r))$ is dense in $U_T(L(r))$, which is dense in $L^2(\Omega(r))$. Now, fix $\varepsilon>0$ and choose $f_\varepsilon\in H(r)$, $f_\varepsilon\neq0$, such that
$$\|u^{f_\varepsilon}(T)-\mathds{1}_{\Omega(r)}\|^2_{L^2_c}<\frac{\varepsilon^2}{3} .$$
Taking $\alpha$ and $\beta$ such that $0<\alpha<\|f_\varepsilon\|^{-2}_{L^2(0,2T)}\frac{\varepsilon^2}{3}$ and $0<\beta<|f_\varepsilon|^{-2}_{H^1(0,2T)}\frac{\varepsilon^2}{3}$, we obtain by minimality
\begin{align*}
m_{\alpha,\beta}(r)\leqslant F_r(f_\varepsilon)&=\|u^{f_\varepsilon}(T)-\mathds{1}_{\Omega(r)}\|^2_{L^2_c}+\alpha\|f_\varepsilon\|^2_{L^2(0,2T)}+\beta|f_\varepsilon|^2_{H^1(0,2T)}<\varepsilon^2 .
\end{align*}
Since $\varepsilon>0$ is arbitrary, $m_{\alpha,\beta}(r)\to0$ as $\alpha,\beta\to+0$.

\smallskip
\noindent\textit{Step 5: uniform convergence in $r$.}
By Step~4, it suffices to show $\sup_{r\in[0,\tau(1)]}m_{\alpha,\beta}(r)\to0$ as $\alpha,\beta\to+0$. We use three properties of $m_{\alpha,\beta}$. First, if $\alpha\leqslant\alpha'$ and $\beta\leqslant\beta'$, then $\|f\|^2_{\alpha,\beta}\leqslant\|f\|^2_{\alpha',\beta'}$ for every $f$, hence $m_{\alpha,\beta}(r)\leqslant m_{\alpha',\beta'}(r)$ for every $r\in[0,\tau(1)]$. Second, $m_{\alpha,\beta}(r)\to0$ as $\alpha,\beta\to+0$ for each fixed $r\in[0,\tau(1)]$ by Step~4. Third, $r\mapsto m_{\alpha,\beta}(r)$ is continuous on $[0,\tau(1)]$ for fixed $\alpha,\beta>0$ by Step~3.

Now let $\alpha_k=\beta_k=\frac1k$, $k\in\N$. By the three properties, $\{m_{\alpha_k,\beta_k}\}_{k\in\N}$ is a nonincreasing sequence of continuous functions on the compact interval $[0,\tau(1)]$ converging pointwise to the continuous function $0$. By Dini's theorem, the convergence is uniform: for every $\varepsilon>0$ there is $k\in\N$ with $\sup_{r\in[0,\tau(1)]}m_{\alpha_k,\beta_k}(r)<\varepsilon^2$. By monotonicity, $m_{\alpha,\beta}(r)\leqslant m_{\alpha_k,\beta_k}(r)<\varepsilon^2$ for all $r\in[0,\tau(1)]$ whenever $0<\alpha,\beta\leqslant\frac1k$. Therefore
$$\sup_{r\in[0,\tau(1)]}\|u^{f_{\alpha,\beta,r}}(T)-\mathds{1}_{\Omega(r)}\|_{L^2_c}<\varepsilon,\qquad 0<\alpha,\beta\leqslant\tfrac1k,$$
which is the required uniform convergence.
\end{proof}

Finally, we prove the main result, Theorem~\ref{3mainres}, by describing a reconstruction procedure. For convenience, we restate it.

\medskip
\noindent\textbf{Theorem~\ref{3mainres}.}
\textit{The forward spectral map $\Phi:\mathcal{C}\to\mathcal{D}$ is bijective.}

\begin{proof}
\smallskip
\noindent\textit{Step 1: the spectral data determine the Blagove\v{s}\v{c}enski\u{\i} operator.}
By Proposition~\ref{1mainres}, the Neumann-to-Dirichlet map $\Lambda$ is built from the sequence $\{\lambda_n(0),\lambda_n'(0)\}_{n\in\N}$ alone. Hence $\Lambda$ is a function of the spectral data. The Blagove\v{s}\v{c}enski\u{\i} operator is defined in terms of $\Lambda$ by $K=R\Lambda RJ-J\Lambda$, where $J$ and $R$ are known time-domain operators. Consequently, $K$ is a function of the spectral data as well.

\smallskip
\noindent\textit{Step 2: the operator $K$ determines the travel-time volume function.}
By Proposition~\ref{2mainres}, for each $r\in[0, \tau(1)]$ and $\alpha,\beta>0$, the regularized control
$$f_{\alpha,\beta,r}=(\mathcal{P}_rK\mathcal{P}_r+\alpha I+\beta S)^{-1}\mathcal{P}_rB1$$
is uniquely determined by $K$. Moreover, $u^{f_{\alpha,\beta,r}}(T)\to\mathds{1}_{\Omega(r)}$ in $L^2_c$ uniformly with respect to $r\in[0, \tau(1)]$ as $\alpha,\beta\to+0$. The Blagove\v{s}\v{c}enski\u{\i} identities give
$$(f_{\alpha,\beta,r},B1)_{L^2(0,2T)}=\bigl(u^{f_{\alpha,\beta,r}}(T),1\bigr)_{L^2_c}\to(\mathds{1}_{\Omega(r)},1)_{L^2_c}, \qquad \alpha,\beta\to +0,$$
uniformly in $r\in[0, \tau(1)]$. Hence the travel-time volume function
\[
V(r)=\int_{\Omega(r)}\frac{dx}{c^2(x)}= (\mathds{1}_{\Omega(r)},1)_{L^2_c}
\]
is uniquely determined by the spectral data.

\smallskip
\noindent\textit{Step 3: the volume function determines the wave speed.}
Recall that the travel time coordinate is $\tau(x)=\int_0^x c(y)^{-1}\,dy$. Since $c(x)>0$, the map $\tau$ is strictly increasing, and we denote its inverse by $\chi=\tau^{-1}$. Thus $\chi(0)=0$, $\chi(\tau(x))=x$, and $\tau(\chi(t))=t$. Differentiating the identity $\tau(\chi(t))=t$ gives
$$\chi'(t)=\frac{1}{\tau'(\chi(t))}=c(\chi(t)).$$
On the other hand, by changing variables $x=\chi(t)$ in the definition of $V$, we obtain
$$V(r)=\int_0^{\chi(r)}\frac{dx}{c^2(x)}=\int_0^r\frac{\chi'(t)}{c^2(\chi(t))}\,dt=\int_0^r\frac{dt}{c(\chi(t))},\quad r\in [0,\tau(1)].$$
The integrand above is continuous on $[0,\tau(1)]$ and bounded from below and above, so $V\in C^1([0,\tau(1)])$. Hence,
$$c(\chi(t))=\frac{1}{V'(t)},\qquad t\in[0,\tau(1)].$$
Furthermore, $\chi(t)=\int_0^t c(\chi(s))\,ds=\int_0^t \frac{1}{V'(s)}\,ds$ for $t\in[0,\tau(1)]$.

Thus both the coordinate change $t\mapsto \chi(t)$ and the value $c(\chi(t))$ are determined by $V$. Consequently, for $x=\chi(t)$ we recover
$$c(x)=\frac{1}{V'(\tau(x))},\qquad x\in\overline\Omega.$$
Therefore the values of $V$ on $[0,\tau(1)]$ determine $\chi$ on $[0,\tau(1)]$ and $c$ on $\overline\Omega$. 

\smallskip
\noindent\textit{Step 4: conclusion.}
Combining the steps, the spectral data determine wave speeds,
$$\left\{\lambda_n(0),\lambda_n^{\prime}(0)\right\}_{n\in\N}\mapsto\Lambda\mapsto K\mapsto V\mapsto c(x).$$
Let now $c_1,c_2\in\mathcal{C}$ have the same spectral data $\{\lambda_n(0),\lambda_n^{\prime}(0)\}_{n\in\N}$, and write $\tau_i,\chi_i,V_i$ for the corresponding travel time, its inverse and the travel-time volume. By Step~1, the two wave speeds give the same $K$, hence the same controls $f_{\alpha,\beta,r}$, so by Step~2 we get $V_1=V_2$ on $[0,r_\ast]$, where $r_\ast=\min\{\tau_1(1),\tau_2(1)\}$. By Step~3, $\chi_1=\chi_2$ on $[0,r_\ast]$ and $c_1(\chi_1(t))=c_2(\chi_2(t))$ for $t\in [0,r_\ast]$, since both functions are continuous. Thereby, $c_1=c_2$ on $[0,\chi_1(r_\ast)]$. Notice that $\chi_i(r_\ast)=1$ for at least one $i$, so $\chi_1(r_\ast)=\chi_2(r_\ast)=1$. Hence, $c_1=c_2$ on $[0,1]=\overline\Omega$. Thus $\Phi$ is injective; it is surjective by the definition $\mathcal{D}=\Phi(\mathcal{C})$, so $\Phi:\mathcal{C}\to\mathcal{D}$ is bijective.
\end{proof}

\medskip

\section{Stability results}\label{stabsec}
\subsection{Forward stability}
The spatial wave operator depends on $c(x)$, so we denote the eigenvalues corresponding to the unpertubed problem $(\mathcal{S}_{0})$ by $\lambda^c_n$ and its derivatives with respect to $\omega$ by $(\lambda_n^c)'(\omega)$. Analogously, the eigenfunctions are denoted by $\phi_n^c$.

We begin this section with a sequence of auxiliary technical lemmas.

\begin{Lemma}\label{weylsn2}
	\textit{For any $c_1,c_2\in\mathcal{C}$ and any $n\in\N$, $|\lambda_n^{c_1} - \lambda_n^{c_2}| \lesssim n^2 \|c_1 - c_2\|_{C^2(\overline\Omega)}$. Moreover, $\lambda_n^c\asymp n^2$ uniformly in $c\in\mathcal{C}$.}
\end{Lemma}
\begin{proof}
For any $c \in \mathcal{C}$, the $n$-th eigenvalue is characterized by the min-max principle, namely
$$\lambda_n^c = \min_{\dim E = n} \,\max_{v \in E\setminus\{0\}} \frac{B(v,v)}{\|v\|_{L^2_c}^2},$$
where the minimum is taken over all $n$-dimensional subspaces $E \subset V$; note that $V$ does not depend on $c$. Since $C_{min}\leqslant c\leqslant C_{max}$ on $\Omega$, one has
$$C_{max}^{-2}\|v\|_{L^2(\Omega)}^2\leqslant \|v\|_{L^2_c}^2\leqslant C_{min}^{-2}\|v\|_{L^2(\Omega)}^2,\qquad v\in V.$$
Recall that $\lambda_n^1=\left(\left(n-\frac{1}{2}\right)\pi\right)^2$ the eigenvalues of $(\mathcal{S})$ for $c\equiv 1$. Comparing the Rayleigh quotients termwise and taking the min-max gives the uniform two-sided bound $C_{min}^2\,\lambda_n^1\leqslant\lambda_n^c\leqslant C_{max}^2\,\lambda_n^1$ for any $n\in\N$ and $c\in\mathcal{C}$. Thereby, $\lambda_n^c\asymp n^2$ with constants depending only on the admissible class $\mathcal{C}$ definition.

By the mean value theorem,
$$|c_1(x)^{-2} - c_2(x)^{-2}| \leqslant 2C_{min}^{-3} |c_1(x) - c_2(x)| \leqslant 2C_{min}^{-3} \|c_1 - c_2\|_{C^2(\overline\Omega)}.$$
Consequently, using $\|v\|_{L^2(\Omega)}^2 \leqslant C_{max}^2 \|v\|_{L^2_{c_2}}^2$ we obtain
$$\left| \|v\|_{L^2_{c_1}}^2 - \|v\|_{L^2_{c_2}}^2 \right| = \left| \int_\Omega (c_1^{-2} - c_2^{-2}) |v|^2 \, dx \right| \leqslant C\|c_1 - c_2\|_{C^2(\overline\Omega)} \|v\|_{L^2_{c_2}}^2,$$
where $C = 2C_{min}^{-3} C_{max}^2$. Dropping the absolute value on the left yields $\|v\|_{L^2_{c_1}}^2 \leqslant \big(1 + C \|c_1 - c_2\|_{C^2(\overline\Omega)}\big) \|v\|_{L^2_{c_2}}^2$, and dividing gives
$$\frac{B(v,v)}{\|v\|_{L^2_{c_1}}^2} \geqslant \frac{B(v,v)}{\big(1 + C \|c_1 - c_2\|_{C^2(\overline\Omega)}\big) \|v\|_{L^2_{c_2}}^2},\qquad v\in V\setminus\{0\}.$$
Taking the min-max over all subspaces $E \subset V$ preserves this inequality for the eigenvalues,
$$ \lambda_n^{c_1} \geqslant \frac{\lambda_n^{c_2}}{1 + C\|c_1 - c_2\|_{C^2(\overline\Omega)}}.$$
Hence, $\lambda_n^{c_2}-\lambda_n^{c_1}\leqslant C \lambda_n^{c_1} \|c_1-c_2\|_{C^2(\overline\Omega)}$. By complete symmetry, we also get $\lambda_n^{c_1} - \lambda_n^{c_2} \leqslant C  \lambda_n^{c_2}\|c_1-c_2\|_{C^2(\overline\Omega)}$. Thus, we obtain
$$|\lambda_n^{c_1} - \lambda_n^{c_2}|\leqslant C\max(\lambda_n^{c_1}, \lambda_n^{c_2}) \|c_1-c_2\|_{C^2(\overline\Omega)}.$$
The uniform upper bound $\lambda_n^c\leqslant C_{max}^2\,\lambda_n^1\lesssim n^2$ entails the required estimate.
\end{proof}

\begin{Lemma}\label{eigderabs}
	\textit{There is a constant $C_{abs}>0$ such that $|(\lambda_n^c)'(0)|\leqslant C_{abs}$ for any $n\in\N$ and for any $c\in\mathcal{C}$.}
\end{Lemma}
\begin{proof}
Consider the following coordinate change $\tau(x) = \int_0^x c(y)^{-1}\,dy$ and $x=\tau^{-1}$. The differentials and derivatives transform as follows:
$$d\tau = \frac{dx}{c(x)},\quad dx = c(x(\tau))\,d\tau,\quad \partial_x = \frac{1}{c(x(\tau))}\partial_\tau.$$
Let $\phi$ be a weak solution to $(\mathcal{S})$, then we get that under the coordinate change the identity $(\partial_x\phi,\partial_xv)_{L^2(\Omega)}=\lambda (\phi,v)_{L^2_c}$, $v\in V$, transforms into
$$\int\limits_0^{\tau(1)}\frac{1}{\tilde{c}(\tau)}\partial_\tau\tilde{\phi}(\tau)\partial_\tau\tilde{v}(\tau)\,d\tau = \lambda \int\limits_0^{\tau(1)}\frac{1}{\tilde{c}(\tau)}\tilde{\phi}(\tau)\tilde{v}(\tau)\,d\tau,$$
where $\tilde{\phi}(\tau)=\phi(x(\tau))$, $\tilde{v}(\tau)=v(x(\tau))$, and $\tilde{c}(\tau)=c(x(\tau))$. Now we define the transformed eigenfunction $\psi(\tau)$ by rescaling $\phi(x)$, namely $\psi(\tau) = \tilde{c}(\tau)^{-1/2}\tilde{\phi}(\tau)$. Notice that $\psi(\tau(1)) = c(1)^{-1/2}\phi(1)=0$. Now, we evaluate
$$\partial_\tau\tilde{\phi}(\tau)=\partial_\tau\left(\tilde{c}(\tau)^{1/2}\psi(\tau)\right)= \tilde{c}(\tau)^{1/2}\left(\partial_\tau\psi(\tau)+\frac{\partial_\tau \tilde{c}(\tau)}{2\tilde{c}(\tau)}\psi(\tau)\right).$$
Similarly, defining $\eta(\tau) = \tilde{c}(\tau)^{-1/2}\tilde{v}(\tau)$ gives exactly the same identity. Insert all of the above into the identity:
$$\int\limits_0^{\tau(1)}\left(\partial_\tau\psi(\tau)+\frac{\partial_\tau \tilde{c}(\tau)}{2\tilde{c}(\tau)}\psi(\tau)\right)\left(\partial_\tau\eta(\tau)+\frac{\partial_\tau \tilde{c}(\tau)}{2\tilde{c}(\tau)}\eta(\tau)\right)\,d\tau = \lambda \int\limits_0^{\tau(1)}\psi(\tau)\eta(\tau)\,d\tau.$$
Expanding the product gives three terms:
$$\int\limits_0^{\tau(1)}\partial_\tau\psi\partial_\tau\eta\,d\tau + \int\limits_0^{\tau(1)} \left(\frac{\partial_\tau \tilde{c}}{2\tilde{c}}\right)^2 \psi\eta\,d\tau + \int\limits_0^{\tau(1)}\frac{\partial_\tau \tilde{c}}{2\tilde{c}}\partial_\tau(\psi\eta)\,d\tau. $$
Integrating the latter by parts gives
$$\left(\frac{\partial_\tau \tilde{c}}{2\tilde{c}}\psi\eta\right)\Bigg|_0^{\tau(1)} -\int\limits_0^{\tau(1)}\partial_\tau\left(\frac{\partial_\tau \tilde{c}}{2\tilde{c}}\right)\psi\eta\,d\tau.$$
The first term above vanishes due to the boundary conditions and the fact that $\partial_\tau\tilde{c}(0)=0$. Hence, we obtain
$$\int\limits_0^{\tau(1)}\partial_\tau\psi\partial_\tau\eta\,d\tau + \int\limits_0^{\tau(1)} \left(\left(\frac{\partial_\tau \tilde{c}}{2\tilde{c}}\right)^2 - \partial_\tau\left(\frac{\partial_\tau \tilde{c}}{2\tilde{c}}\right)\right)\psi\eta\,d\tau  =\lambda \int\limits_0^{\tau(1)}\psi\eta\,d\tau.$$
Moreover, the following 
$$\partial_\tau\psi(\tau)=c(x)\partial_x\left(c(x)^{-1/2}\phi(x)\right)= c(x)^{1/2}\partial_x\phi(x)-\frac{1}{2}c(x)^{-1/2}c'(x)\phi(x)$$
shows that the condition $\partial_x\phi(0)=0$ turns into $\partial_\tau\psi(0)=0$, since $c'(0)=0$ for any $c\in\mathcal{C}$. This matches the canonical weak formulation of the Schrödinger equation:
$$-\partial_\tau^2\psi + q\psi = \lambda\psi,\quad q=\frac{1}{4}\left(\frac{\partial_\tau \tilde{c}}{\tilde{c}}\right)^2 - \frac{1}{2}\partial_\tau\left(\frac{\partial_\tau \tilde{c}}{\tilde{c}}\right).$$

The potential $q$ is continuous and uniformly bounded due to the formula $q=(\frac14 (c')^2 -\frac12 cc'')(x(\tau))$, namely $\|q\|_{L^\infty(0, \tau(1))} \leqslant \frac34 C_{max}^2=: Q < \infty$. By variation of parameters, any solution to the Schrödinger equation satisfies the Volterra integral equation:
$$\psi_n(\tau)= \phi_n(0)\cos(\tau\sqrt{\lambda_n^c}) + S_{\lambda_n^c}[q \psi_n](\tau) .$$
Now fix arbitrary $c\in\mathcal{C}$. Note that $\|\cos(\cdot\sqrt{\lambda_n^c})\|_{L^2(0,\tau(1))}^2=\frac12 \tau(1) +(4\sqrt{\lambda_n^c})^{-1}\sin(2\tau(1)\sqrt{\lambda_{n}^c}) \geqslant \frac14 \tau(1) \geqslant(4C_{max})^{-1}=:\varepsilon$ for any $n\in\N$, since if $\tau(1)\sqrt{\lambda_{n}^c}\leqslant 1 <\frac\pi2$, the value of the sine function is nonnegative, and if $\tau(1)\sqrt{\lambda_{n}^c}> 1$, the absolute value of the sine function is at most $1$. Then the inequality
$$|\phi_n(0)|\cdot\|\cos(\cdot\sqrt{\lambda_n^c})\|_{L^2(0,\tau(1))}\leqslant\|\psi_n\|_{L^2(0,\tau(1))}+\|S_{\lambda_n^c}[q \psi_n]\|_{L^2(0,\tau(1))}$$
implies that there is a constant $C_{abs} >0$ depending only on $C_{min}$ and $C_{max}$ such that
$$|\phi_n(0)|^2\lesssim \frac{1}{\varepsilon}\left(1+\frac{\tau(1)Q}{\sqrt{\lambda_n^c}}\right)^2\|\psi_n\|^2_{L^2(0,\tau(1))}\leqslant C_{abs}$$
due to $\|\psi_n\|_{L^2(0,\tau(1))}=\|\phi_n\|_{L^2_c}=1$ and the uniform lower bound $\lambda^c_n\geqslant C^2_{min}\lambda^1_1=(C_{min}\pi/2)^2$ from the proof of Lemma~\ref{weylsn2}. The required follows from Corollary \ref{gbeig0}.
\end{proof}

\begin{Lemma}\label{eigderdiff0}
	\textit{For any $c_1,c_2\in\mathcal{C}$, $|(\lambda_n^{c_1})'(0) - (\lambda_n^{c_2})'(0)|\lesssim \|c_1-c_2\|_{C^2(\overline\Omega)}$ for every $n\in\N$.}
\end{Lemma}
\begin{proof}
Fix $c_1,c_2\in\mathcal{C}$, $n\in\N$, and write $\delta=\|c_1-c_2\|_{C^2(\R)}$, $\lambda_i=\lambda_n^{c_i}$, $i=1,2$. By Corollary~\ref{gbeig0},
$$(\lambda_n^{c_1})'(0) - (\lambda_n^{c_2})'(0) =(\phi_n^{c_2}(0) - \phi_n^{c_1}(0))\, (\phi_n^{c_1}(0) + \phi_n^{c_2}(0)),$$
and Lemma~\ref{eigderabs} gives $|\phi_n^{c_1}(0) + \phi_n^{c_2}(0)|\leqslant 2\sqrt{C_{abs}}$. Hence it suffices to show
\begin{equation}\label{eq:trace-diff}
	|\phi_n^{c_2}(0) - \phi_n^{c_1}(0)|\lesssim \delta .
\end{equation}

\smallskip
\noindent\textit{Step 1: common domain and normalized eigenfunctions.}
We place the eigenfunctions of $c_1$ and $c_2$ on a common interval, normalize them at $\tau=0$ instead of in $L^2$, and record uniform bounds on the normalized solutions and their derivatives.

For $c\in\mathcal{C}$ the travel time $\tau_c(x)=\int_0^x c^{-1}(y)\,dy$ is defined for all $x\in\R$ and its inverse $\chi_c$ is defined on $\R$, since $c\equiv 1$ on $\R\setminus\Omega$. Consequently the potential
$$q_c=\Big(\tfrac14 (c')^2-\tfrac12\, c\, c''\Big)\circ\chi_c$$
from the proof of Lemma~\ref{eigderabs} is continuous on $[0,+\infty)$, $\|q_c\|_{L^\infty(0,+\infty)}\leqslant Q$, and $q_c\equiv 0$ on $[\tau_c(1),+\infty)$ due to $c'(1)=c''(1)=0$. Set $L=C_{min}^{-1}$; then $\tau_c(1)\leqslant L$ for every $c\in\mathcal{C}$. Throughout the proof all functions of $\tau$ are considered on the common interval $[0,L]$. Let $v_c\in C^2[0,L]$ be the solution of the initial value problem
$$-\partial_\tau^2 v_c + q_c v_c=\lambda_n^c v_c \ \text{ on } (0,L),\qquad v_c(0)=1,\quad v_c'(0)=0 .$$
Let $\phi_n^c$ and $\psi_n^c$ be as in the proof of Lemma~\ref{eigderabs}; recall that $\psi_n^c(0)=\phi_n^c(0)>0$ and $\partial_\tau\psi_n^c(0)=0$. Uniqueness for the initial value problem yields $v_c=\psi_n^c/\phi_n^c(0)$ on $[0,\tau_c(1)]$. By variation of parameters, $v_c$ satisfies the Volterra equation
\begin{equation}\label{eq:volterra}
	v_c(\tau) = \cos(\tau\sqrt{\lambda_n^c}) + S_{\lambda_n^c}[q_c v_c](\tau),\qquad \tau\in[0,L].
\end{equation}
Note that for every $\lambda>0$ and $0<\tau\leqslant L$,
\begin{equation}\label{eq:S-op-bound}
	\|S_\lambda\|_{L^1(0,\tau)\to L^\infty(0,\tau)}\leqslant\lambda^{-1/2}.
\end{equation}
Hence \eqref{eq:volterra} gives $|v_c(\tau)|\leqslant 1+Q(\lambda_n^c)^{-1/2}\|v_c\|_{L^1(0,\tau)}$ for $\tau\in[0,L]$, and Gr\"onwall's inequality together with Lemma~\ref{weylsn2} yields
\begin{equation}\label{eq:v-bound}
	\|v_c\|_{L^\infty(0,L)}\leqslant e^{QL/\sqrt{\lambda_1^c}}\leqslant e^{2Q/(\pi C_{min}^2)}=:M_0 .
\end{equation}
Differentiating \eqref{eq:volterra} by the Leibniz rule, we get
$$v_c'(\tau)=-\sqrt{\lambda_n^c}\,\sin(\tau\sqrt{\lambda_n^c})+\int_0^\tau\cos\big((\tau-t)\sqrt{\lambda_n^c}\big)\,q_c(t)v_c(t)\,dt,$$
and hence, by \eqref{eq:v-bound},
\begin{equation}\label{eq:vprime-bound}
	\|v_c'\|_{L^\infty(0,L)}\leqslant \sqrt{\lambda_n^c}+LQM_0\lesssim n .
\end{equation}
Finally, by \eqref{eq:S-op-bound}, \eqref{eq:v-bound} and Lemma~\ref{weylsn2},
\begin{equation}\label{eq:rem-bound}
	\|S_{\lambda_n^c}[q_cv_c]\|_{L^\infty(0,L)}\leqslant (\lambda_n^c)^{-1/2}LQM_0\lesssim n^{-1}.
\end{equation}

\smallskip
\noindent\textit{Step 2: comparison of the coordinate changes and potentials.}
Write $\tau_i=\tau_{c_i}$ and $\chi_i=\chi_{c_i}$, $i=1,2$. We show that the travel-time coordinates differ by $O(\delta)$, and split $q_{c_2}-q_{c_1}$ into a part of size $O(\delta)$ and a derivative of a function of size $O(\delta)$; this splitting is needed for the integration by parts in Step~3.

For $x\in\overline{\Omega}$, $|\tau_1(x)-\tau_2(x)|\leqslant\int_0^1|c_1^{-1}-c_2^{-1}|\,dy\leqslant C_{min}^{-2}\delta$, and the difference $\tau_1-\tau_2$ is constant on $[1,+\infty)$ and vanishes on $(-\infty,0]$, so $\|\tau_1-\tau_2\|_{L^\infty(\R)}\leqslant C_{min}^{-2}\delta$. Since $\chi_1'=c_1\circ\chi_1\leqslant C_{max}$, for $\tau\geqslant 0$ and $x=\chi_2(\tau)$,
$$|\chi_1(\tau)-\chi_2(\tau)|=|\chi_1(\tau_2(x))-\chi_1(\tau_1(x))|\leqslant C_{max}\|\tau_1-\tau_2\|_{L^\infty(\R)}\leqslant C_{max}C_{min}^{-2}\delta .$$
Write $q_{c_i}=a_i^2-\partial_\tau a_i$ with $a_i=\tfrac12 c_i'\circ\chi_i\in C^1[0,L]$, and set $b=a_2-a_1$. Since $c_1'$ is $C_{max}$-Lipschitz, for $\tau\in[0,L]$,
$$2|b(\tau)|\leqslant |c_2'(\chi_2(\tau))-c_1'(\chi_2(\tau))|+|c_1'(\chi_2(\tau))-c_1'(\chi_1(\tau))|\leqslant \delta + C_{max}^2 C_{min}^{-2}\delta,$$
so $\|b\|_{L^\infty(0,L)}\lesssim\delta$. Moreover $\|a_1+a_2\|_{L^\infty(0,L)}\leqslant C_{max}$, hence
$$\|a_2^2-a_1^2\|_{L^\infty(0,L)}\leqslant\|b\|_{L^\infty(0,L)}\|a_1+a_2\|_{L^\infty(0,L)}\lesssim\delta .$$
Finally, $b(0)=\tfrac12(c_2'(0)-c_1'(0))=0$.

\smallskip
\noindent\textit{Step 3: the difference $w=v_{c_2}-v_{c_1}$.}
Subtracting the Volterra equations \eqref{eq:volterra} gives, on $[0,L]$,
\begin{align*}
w ={}& \cos(\cdot\,\sqrt{\lambda_2}) - \cos(\cdot\,\sqrt{\lambda_1})+ (S_{\lambda_2} - S_{\lambda_1})[q_{c_2} v_{c_2}] \\
&+ S_{\lambda_1}\big[(q_{c_2} - q_{c_1}) v_{c_2}\big] + S_{\lambda_1}[q_{c_1} w] .
\end{align*}
We estimate the first three terms in $L^\infty(0,L)$.

First, by the mean value theorem and Lemma~\ref{weylsn2}, for $\tau\in[0,L]$,
\begin{equation}\label{eq:cos-diff}
|\cos(\tau\sqrt{\lambda_2}) - \cos(\tau\sqrt{\lambda_1})| \leqslant \tau|\sqrt{\lambda_2}-\sqrt{\lambda_1}|= \frac{\tau |\lambda_2 - \lambda_1|}{\sqrt{\lambda_1}+\sqrt{\lambda_2}} \lesssim n\delta .
\end{equation}

Second, for $s\in[0,2L]$ the function $\lambda\mapsto \frac{\sin(s\sqrt\lambda)}{\sqrt\lambda}$ satisfies
$$\Big|\partial_\lambda\frac{\sin(s\sqrt{\lambda})}{\sqrt\lambda}\Big| = \Big|\frac{s \cos(s\sqrt{\lambda})}{2\lambda} - \frac{\sin(s\sqrt{\lambda})}{2\lambda^{3/2}}\Big|\leqslant \frac{L}{\lambda}+\frac{1}{2\lambda^{3/2}},$$
and hence, by the mean value theorem in $\lambda$ and Lemma~\ref{weylsn2},
\begin{equation}\label{eq:kernel-diff}
\sup_{s\in[0,2L]}\Big|\frac{\sin(s\sqrt{\lambda_2})}{\sqrt{\lambda_2}}-\frac{\sin(s\sqrt{\lambda_1})}{\sqrt{\lambda_1}}\Big|
\leqslant |\lambda_2-\lambda_1|\sup_{\lambda\in[\min\lambda_i,\max\lambda_i]}\Big(\frac{L}{\lambda}+\frac{1}{2\lambda^{3/2}}\Big).
\end{equation}
The latter is bounded by a constant multiple of $n^{-2}\cdot n^2\delta=\delta$. In particular $\|S_{\lambda_2}-S_{\lambda_1}\|_{L^1(0,L)\to L^\infty(0,L)}\lesssim\delta$, and together with \eqref{eq:v-bound},
\begin{align*}
\|(S_{\lambda_2} - S_{\lambda_1})[q_{c_2} v_{c_2}]\|_{L^\infty(0,L)}& \leqslant \|S_{\lambda_2}-S_{\lambda_1}\|_{L^1(0,L)\to L^\infty(0,L)}\,\|q_{c_2}v_{c_2}\|_{L^1(0,L)}\\
& \lesssim LQM_0\,\delta .
\end{align*}

Third, $q_{c_2}-q_{c_1}=(a_2^2-a_1^2)-\partial_\tau b$. By Step 2, \eqref{eq:v-bound} and \eqref{eq:S-op-bound},
\begin{align*}
	\|S_{\lambda_1}[(a_2^2-a_1^2) v_{c_2}]\|_{L^\infty(0,L)} & \leqslant \lambda_1^{-1/2}\|(a_2^2-a_1^2)v_{c_2}\|_{L^1(0,L)}\\
	& \leqslant L\lambda_1^{-1/2}M_0\|a_2^2-a_1^2\|_{L^\infty(0,L)}\lesssim n^{-1}\delta .
\end{align*}
For the remaining part we integrate by parts. For $\tau\in[0,L]$,
\begin{align*}
S_{\lambda_1}[b'\,v_{c_2}](\tau)=&\int_0^\tau \frac{\sin((\tau-t)\sqrt{\lambda_1})}{\sqrt{\lambda_1}}\,b'(t)\,v_{c_2}(t)\,dt\\ 
=&\Big[\frac{\sin((\tau-t)\sqrt{\lambda_1})}{\sqrt{\lambda_1}}\,b(t)\,v_{c_2}(t)\Big]_{t=0}^{t=\tau}\\ 
&-\int_0^\tau b(t)\,\partial_t\Big(\frac{\sin((\tau-t)\sqrt{\lambda_1})}{\sqrt{\lambda_1}}v_{c_2}(t)\Big)\,dt .
\end{align*}
The boundary term vanishes, since $\sin 0=0$ and $b(0)=0$. Furthermore, by \eqref{eq:v-bound} and \eqref{eq:vprime-bound}, for $0\leqslant t\leqslant\tau\leqslant L$,
\begin{align*}
\partial_t\Big(\frac{\sin((\tau-t)\sqrt{\lambda_1})}{\sqrt{\lambda_1}}v_{c_2}(t)\Big)
=&-\cos((\tau-t)\sqrt{\lambda_1})\,v_{c_2}(t)\\ &+\frac{\sin((\tau-t)\sqrt{\lambda_1})}{\sqrt{\lambda_1}}\,v_{c_2}'(t).
\end{align*}
The modulus of the latter is upper bounded by $M_0+\lambda_1^{-1/2}(\sqrt{\lambda_2}+LQM_0)\lesssim 1$, where $\sqrt{\lambda_2/\lambda_1}\leqslant C_{max}/C_{min}$ by Lemma~\ref{weylsn2}. Hence
$$\|S_{\lambda_1}[b'\, v_{c_2}]\|_{L^\infty(0,L)}\lesssim L\|b\|_{L^\infty(0,L)}\lesssim\delta .$$

Finally, by \eqref{eq:S-op-bound} on $(0,\tau)$, $|S_{\lambda_1}[q_{c_1} w](\tau)|\leqslant Q\lambda_1^{-1/2}\|w\|_{L^1(0,\tau)}$ for $\tau\in[0,L]$. Collecting the estimates, for $\tau\in[0,L]$, $|w(\tau)|\lesssim n\,\delta + Q\lambda_1^{-1/2}\|w\|_{L^1(0,\tau)}$, and Gr\"onwall's inequality together with $QL/\sqrt{\lambda_1}\leqslant 2Q/(\pi C_{min}^2)$ yields
\begin{equation}\label{eq:w-bound}
	\|v_{c_2}-v_{c_1}\|_{L^\infty(0,L)}\lesssim n\delta .
\end{equation}
By \eqref{eq:S-op-bound} and \eqref{eq:w-bound}, $\|S_{\lambda_1}[q_{c_1}w]\|_{L^\infty(0,L)}\leqslant\lambda_1^{-1/2}LQ\|w\|_{L^\infty(0,L)}\lesssim\delta$, and hence the three estimates above give
\begin{equation}\label{eq:rem-diff}
	\big\|S_{\lambda_2}[q_{c_2}v_{c_2}]-S_{\lambda_1}[q_{c_1}v_{c_1}]\big\|_{L^\infty(0,L)}\lesssim\delta .
\end{equation}

\smallskip
\noindent\textit{Step 4: boundary traces.}
Since $\|\psi_n^c\|_{L^2(0,\tau_c(1))}=\|\phi_n^c\|_{L^2_c}=1$ and $v_c=\psi_n^c/\phi_n^c(0)$ on $[0,\tau_c(1)]$, we have
$$N_c:=\|v_c\|_{L^2(0,\tau_c(1))}=\frac{1}{\phi_n^c(0)},\qquad C_{abs}^{-1/2}\leqslant N_c\leqslant \sqrt{L}\,M_0 ,$$
by Lemma~\ref{eigderabs} and \eqref{eq:v-bound}. Denote $N_i=N_{c_i}$. Then
$$|\phi_n^{c_2}(0)-\phi_n^{c_1}(0)|=\frac{|N_1-N_2|}{N_1N_2}= \frac{|N_1^2-N_2^2|}{N_1N_2(N_1+N_2)}\leqslant \frac{C_{abs}^{3/2}}{2}\,|N_1^2-N_2^2| .$$
Assume without loss of generality that $\tau_1(1)\geqslant\tau_2(1)$. Then
\begin{equation}\label{eq:N-split}
|N_1^2-N_2^2|
\leqslant \Big|\int_0^{\tau_2(1)} (v_{c_1}^2 - v_{c_2}^2) \, d\tau\Big| + \int_{\tau_2(1)}^{\tau_1(1)} |v_{c_1}|^2 \, d\tau ,
\end{equation}
where the second term is upper bounded by $M_0^2|\tau_1(1)-\tau_2(1)|\leqslant M_0^2C_{min}^{-2}\delta$ by Step 2. For the first term we substitute \eqref{eq:volterra} for $v_{c_1}$ and $v_{c_2}$ and expand:
\begin{align*}
v_{c_1}^2-v_{c_2}^2={}&\cos^2(\cdot\sqrt{\lambda_1})-\cos^2(\cdot\sqrt{\lambda_2})\\
&+2\big(\cos(\cdot\sqrt{\lambda_1})-\cos(\cdot\sqrt{\lambda_2})\big)S_{\lambda_1}[q_{c_1}v_{c_1}]\\
&+2\cos(\cdot\sqrt{\lambda_2})\big(S_{\lambda_1}[q_{c_1}v_{c_1}]-S_{\lambda_2}[q_{c_2}v_{c_2}]\big)\\
&+\big(S_{\lambda_1}[q_{c_1}v_{c_1}]-S_{\lambda_2}[q_{c_2}v_{c_2}]\big)\big(S_{\lambda_1}[q_{c_1}v_{c_1}]+S_{\lambda_2}[q_{c_2}v_{c_2}]\big).
\end{align*}
The first line is integrated explicitly: since $\cos^2(\tau\sqrt\lambda)=\tfrac12(1+\cos(2\tau\sqrt\lambda))$,
$$\int\limits_0^{\tau_2(1)}\big(\cos^2(\tau\sqrt{\lambda_1})-\cos^2(\tau\sqrt{\lambda_2})\big)\,d\tau
=\frac14\Big(\frac{\sin(2\tau_2(1)\sqrt{\lambda_1})}{\sqrt{\lambda_1}}-\frac{\sin(2\tau_2(1)\sqrt{\lambda_2})}{\sqrt{\lambda_2}}\Big),$$
and $\tau_2(1)\leqslant L$, so its absolute value is bounded by a constant multiple of $\delta$ due to \eqref{eq:kernel-diff}. The remaining three lines are bounded in $L^\infty(0,\tau_2(1))$ by \eqref{eq:cos-diff} and \eqref{eq:rem-bound}, by \eqref{eq:rem-diff}, and by \eqref{eq:rem-diff} and \eqref{eq:rem-bound}, respectively:
$$n\delta\cdot n^{-1}=\delta,\qquad \delta,\qquad \delta\cdot n^{-1}\leqslant\delta .$$
Integrating over $(0,\tau_2(1))$ and inserting into \eqref{eq:N-split} gives $|N_1^2-N_2^2|\lesssim\delta$. This proves \eqref{eq:trace-diff} and the lemma.
\end{proof}

Recall from Section~\ref{intro} that $\ell^\infty_s$, $s\in\R$, is the space of real sequences $a=\{a_n\}_{n\in\N}$ with
$$\|a\|_{\ell^\infty_s}=\sup \{n^{-s}|a_n|\,:\, n\in\N\}<+\infty .$$
Note that $\ell^\infty_s\hookrightarrow\ell^\infty_{s'}$ continuously for $s\leqslant s'$. For $\omega\in[0,1)$ define the \textit{eigenvalue map}
$$\Phi^\perp_\omega:\mathcal{C}\to\ell^\infty_2,\qquad \Phi^{\perp}_\omega(c)=\{\lambda_n^c(\omega)\}_{n\in\N}.$$
It is well defined. Indeed, $B_\omega\leqslant B_0$ on $V$, so the min-max principle gives $\lambda_n^c(\omega)\leqslant\lambda_n^c(0)\lesssim n^2$ uniformly in $c\in\mathcal{C}$ and $\omega\in[0,1)$ by Lemma~\ref{weylsn2}. For the same reason Lemma~\ref{weylsn2} extends to every $\omega\in[0,1)$ with a constant independent of $\omega$: its proof uses only that $B_\omega$ does not depend on $c$, and the final bound $\max_i\lambda_n^{c_i}(\omega)\leqslant\max_i\lambda_n^{c_i}(0)\lesssim n^2$. The weight $s=2$ is the smallest for which $\Phi^\perp_0(\mathcal{C})\subset\ell^\infty_s$.

By Proposition~\ref{lem2}, for each $c\in\mathcal{C}$ the map $\omega\mapsto\Phi^\perp_\omega(c)$ is componentwise analytic on $[0,1)$, and we set
$$D_\omega\Phi^\perp_\omega(c)\big|_{\omega=0}=\{(\lambda_n^c)'(0)\}_{n\in\N}.$$
By Lemma~\ref{eigderabs} this sequence is bounded, so $D_\omega\Phi^\perp_\omega(c)|_{\omega=0}\in\ell^\infty_0$. In this notation the forward spectral map reads
$$\Phi:\mathcal{C}\to\ell^\infty_2\times\ell^\infty_0,\qquad \Phi(c)=\Big(\Phi^\perp_0(c),\ D_\omega\Phi^\perp_\omega(c)\big|_{\omega=0}\Big),$$
where $\mathcal{C}$ is equipped with the $C^2(\overline\Omega)$ norm and $\ell^\infty_2\times\ell^\infty_0$ with the sum of the norms. 

We are now ready to prove the forward stability result, Proposition~\ref{4mainres}, which we restate.

\medskip
\noindent\textbf{Proposition~\ref{4mainres}.}
\textit{For any $c_1,c_2\in\mathcal{C}$,
$$\|\Phi(c_1)-\Phi(c_2)\|_{\ell^\infty_2\times\ell^\infty_0}\lesssim\|c_1-c_2\|_{C^2(\overline\Omega)} .$$}
\begin{proof}
Fix $c_1,c_2\in\mathcal{C}$. By the definition of the norms and of $\Phi$,
\begin{equation}\label{eq:Phi-norm}
\|\Phi(c_1)-\Phi(c_2)\|_{\ell^\infty_2\times\ell^\infty_0}=\sup_{n\in\N}n^{-2}|\lambda_n^{c_1}-\lambda_n^{c_2}|+\sup_{n\in\N}|(\lambda_n^{c_1})'(0)-(\lambda_n^{c_2})'(0)|.
\end{equation}
By Lemma~\ref{weylsn2}, $n^{-2}|\lambda_n^{c_1}-\lambda_n^{c_2}|\lesssim\|c_1-c_2\|_{C^2(\overline\Omega)}$, and by Lemma~\ref{eigderdiff0}, $|(\lambda_n^{c_1})'(0)-(\lambda_n^{c_2})'(0)|\lesssim\|c_1-c_2\|_{C^2(\overline\Omega)}$, both with constants independent of $n$. Taking the suprema over $n\in\N$ gives the claim.
\end{proof}

\medskip

\subsection{Inverse stability}\label{invstabr}
We denote the Neumann--to--Dirichlet map depending on a wave speed function $c(x)$ by $\Lambda^c$. Now, we prove auxiliary statements in order to obtain the inverse stability result.

\begin{Lemma}\label{hiddenreg}
\textit{There is $L>0$ depending only on $C_{min}$, $C_{max}$ and $T$ such that $\|\Lambda^c\|_{L^2(0,2T)\to H^1(0,2T)}\leqslant L$ for all $c\in\mathcal{C}$.}
\end{Lemma}
\begin{proof}
By density of $H^1_0(0,2T)$ in $L^2(0,2T)$ and boundedness of $\Lambda^c$ on $L^2(0,2T)$, it suffices to show $\|\partial_t\Lambda^c[f]\|_{L^2(0,2T)}\lesssim\|f\|_{L^2(0,2T)}$ for $f\in H^1_0(0,2T)$. The proof relies on energy estimates on $(0,s)\times\Omega$ for some $s\geqslant 2T$. Applied on the time interval $(0,2T)$, they would require $s=2T$ and hence would contain terms evaluated at the fixed time $s=2T$, for which the energy estimate of Section~\ref{fpee} provides no bound. We therefore work on $(0,4T)$ and choose $s\in[2T,4T]$.

Fix $f\in H^1_0(0,2T)$ and extend it by zero to $(0,4T)$; the extension belongs to $H^1_\diamond(0,4T)$. Let $u\in H^2((0,4T)\times\Omega)$ be the strong solution of $(\mathcal{W}_f)$ on $(0,4T)$ given by Section~\ref{fpee}; by uniqueness, its restriction to $(0,2T)$ is the solution of the original problem, and $\partial_t\Lambda^c[f](t)=\partial_tu(t,0)$ on $(0,2T)$. The energy estimate on $(0,4T)$ reads
\begin{equation}\label{eq:hr-energy}
\|\partial_xu\|_{L^\infty(0,4T;L^2(\Omega))}+\|\partial_tu\|_{L^2((0,4T)\times\Omega)}\lesssim\|f\|_{L^2(0,2T)}.
\end{equation}

Fix $s\in(0,4T]$ and set $m(x)=1-x$. Multiply the equation $c^{-2}\partial_t^2u-\partial_x^2u=0$ by $m\,\partial_xu$ and integrate over $(0,s)\times\Omega$. Since $u\in H^2((0,4T)\times\Omega)$, the products $c^{-2}m\,\partial_tu\,\partial_xu$, $c^{-2}m|\partial_tu|^2$ and $m|\partial_xu|^2$ belong to $W^{1,1}((0,4T)\times\Omega)$. Hence the integrations by parts below are justified. Writing $c^{-2}m\,\partial_t^2u\,\partial_xu=\partial_t\big(c^{-2}m\,\partial_tu\,\partial_xu\big)-\tfrac12c^{-2}m\,\partial_x\big(|\partial_tu|^2\big)$ and using $u(0)=\partial_tu(0)=0$, and ${c(0)=1}$, we get
\begin{align*}
	\int\limits_0^s\!\!\int\limits_\Omega c^{-2}m\,\partial_t^2u\,\partial_xu\,dx\,dt=&\int\limits_\Omega c^{-2}m\,\partial_tu(s)\,\partial_xu(s)\,dx+\frac12\int\limits_0^s|\partial_tu(t,0)|^2\,dt\\
	&+\frac12\int\limits_0^s\!\!\int\limits_\Omega\partial_x(c^{-2}m)|\partial_tu|^2\,dx\,dt .
\end{align*}
Similarly, $m\,\partial_x^2u\,\partial_xu=\tfrac12\partial_x\big(m|\partial_xu|^2\big)+\tfrac12|\partial_xu|^2$ and $\partial_xu(t,0)=f(t)$ give
$$-\int\limits_0^s\!\!\int\limits_\Omega m\,\partial_x^2u\,\partial_xu\,dx\,dt=\frac12\int\limits_0^s|f|^2\,dt-\frac12\int\limits_0^s\!\!\int\limits_\Omega|\partial_xu|^2\,dx\,dt .$$
Adding the two identities, the left-hand side vanishes, whence
\begin{align*}
	\frac12\int\limits_0^s|\partial_tu(t,0)|^2\,dt=&-\int\limits_\Omega c^{-2}m\,\partial_tu(s)\,\partial_xu(s)\,dx-\frac12\int\limits_0^s\!\!\int\limits_\Omega\partial_x(c^{-2}m)|\partial_tu|^2\,dx\,dt\\
	&-\frac12\int\limits_0^s|f|^2\,dt+\frac12\int\limits_0^s\!\!\int\limits_\Omega|\partial_xu|^2\,dx\,dt .
\end{align*}
Dropping the nonpositive term $-\frac12\int_0^s|f|^2\,dt$, we obtain
\begin{align}
\frac12\int\limits_0^s|\partial_tu(t,0)|^2\,dt\leqslant &\Bigg|\int\limits_\Omega c^{-2}m\,\partial_tu(s)\,\partial_xu(s)\,dx\Bigg|\nonumber\\
&+\Big(\frac{C_{max}}{C_{min}^3}+\frac{1}{2C_{min}^2}\Big)\int\limits_0^s\!\!\int\limits_\Omega|\partial_tu|^2\,dx\,dt\nonumber\\
&+\frac12\int\limits_0^s\!\!\int\limits_\Omega|\partial_xu|^2\,dx\,dt,\label{eq:hr-multiplier}
\end{align}
where $|\partial_x(c^{-2}m)|=|2c'c^{-3}m+c^{-2}|\leqslant 2C_{max}C_{min}^{-3}+C_{min}^{-2}$ since $0\leqslant m\leqslant1$.

Denote $J(s)=\int_\Omega c^{-2}m\,\partial_tu(s)\,\partial_xu(s)\,dx$, $s\in(0,4T]$. By the Cauchy--Schwarz inequality in $x$ and then in $s$, together with \eqref{eq:hr-energy},
\begin{align*}
\int\limits_{2T}^{4T}|J(s)|\,ds&\leqslant C_{min}^{-2}\int\limits_{2T}^{4T}\|\partial_tu(s)\|_{L^2(\Omega)}\|\partial_xu(s)\|_{L^2(\Omega)}\,ds\\
&\lesssim\|\partial_tu\|_{L^2((0,4T)\times\Omega)}\|\partial_xu\|_{L^\infty(0,4T;L^2(\Omega))}\lesssim\|f\|^2_{L^2(0,2T)} .
\end{align*}
Hence there is $s_0\in[2T,4T]$ such that $|J(s_0)|\leqslant\frac{1}{2T}\int_{2T}^{4T}|J(s)|\,ds\lesssim\|f\|^2_{L^2(0,2T)}$ (otherwise $|J|>\frac{1}{2T}\int_{2T}^{4T}|J(s)|\,ds$ on $[2T,4T]$, and integrating over $[2T,4T]$ gives a contradiction). The remaining two terms on the right-hand side of \eqref{eq:hr-multiplier} are bounded, for every $s\in(0,4T]$, by
$$\Big(\frac{C_{max}}{C_{min}^3}+\frac{1}{2C_{min}^2}\Big)\|\partial_tu\|^2_{L^2((0,4T)\times\Omega)}+2T\|\partial_xu\|^2_{L^\infty(0,4T;L^2(\Omega))}\lesssim\|f\|^2_{L^2(0,2T)}$$
due to \eqref{eq:hr-energy}. Since $s_0\geqslant 2T$, the estimate \eqref{eq:hr-multiplier} with $s=s_0$ yields
$$\|\partial_tu(\cdot,0)\|^2_{L^2(0,2T)}\leqslant\int\limits_0^{s_0}|\partial_tu(t,0)|^2\,dt\lesssim\|f\|^2_{L^2(0,2T)} .$$
Since $\partial_t\Lambda^c[f]=\partial_tu(\cdot,0)$ on $(0,2T)$, this gives $\|\partial_t\Lambda^c[f]\|_{L^2(0,2T)}\lesssim\|f\|_{L^2(0,2T)}$ with a constant depending only on $C_{min}$, $C_{max}$ and $T$.
\end{proof}

Now, we prove an auxiliary lemma.
\begin{Lemma}\label{interpolateA}
\textit{Let $A:L^2(0,2T)\to H^1(0,2T)$ be bounded. Then}
$$\|A\|_{L^2(0,2T)\to L^2(0,2T)}\lesssim\|A\|^{1/2}_{L^2(0,2T)\to H^{-1}(0,2T)}\,\|A\|^{1/2}_{L^2(0,2T)\to H^{1}(0,2T)} .$$
\end{Lemma}
\begin{proof}
For $v\in H^1_0(0,2T)$ the inequality $\|v\|^2_{L^2(0,2T)}\leqslant\|v\|_{H^{-1}(0,2T)}\|v\|_{H^1(0,2T)}$ is immediate, since $v$ may be tested against itself in the duality between $H^{-1}(0,2T)$ and $H^1_0(0,2T)$. For $v\in H^1(0,2T)$ this argument is unavailable, as $v$ need not vanish at the endpoints, and we test $v$ against a truncation instead.

Let $v\in H^1(0,2T)$ and $h\in(0,T)$, and let $\zeta_h:[0,2T]\to[0,1]$ be Lipschitz, equal to one on $[h,2T-h]$, vanishing at the endpoints and linear in between, so that $|\zeta_h'|\leqslant h^{-1}$ and $\zeta_h'$ is supported in $E_h=(0,h)\cup(2T-h,2T)$, with $|E_h|=2h$. Since $\zeta_hv\in H^1_0(0,2T)$ and $1-\zeta_h$ vanishes outside $E_h$,
\begin{align*}
\|v\|^2_{L^2(0,2T)}&=\int_{E_h}v^2(1-\zeta_h)\,dt+\int_0^{2T}v\,\zeta_hv\,dt\\
&\leqslant 2h\|v\|^2_{L^\infty(0,2T)}+\|v\|_{H^{-1}(0,2T)}\|\zeta_hv\|_{H^1(0,2T)} .
\end{align*}
By the product rule,
\begin{equation}\label{prrulzethv}
	\|\zeta_hv\|_{H^1(0,2T)}\lesssim\|v\|_{H^1(0,2T)}+h^{-1}\|v\|_{L^2(E_h)}\lesssim\|v\|_{H^1(0,2T)}+h^{-1/2}\|v\|_{L^\infty(0,2T)}.
\end{equation}
By the Gagliardo--Nirenberg inequality on a bounded interval, 
$$\|v\|_{L^\infty(0,2T)}\lesssim\|v\|^{1/2}_{L^2(0,2T)}\|v'\|^{1/2}_{L^2(0,2T)}+\|v\|_{L^2(0,2T)},$$
and squaring this, together with $\|v\|^2_{L^2(0,2T)}\leqslant\|v\|_{L^2(0,2T)}\|v\|_{H^1(0,2T)}$, gives $\|v\|^2_{L^\infty(0,2T)}\lesssim\|v\|_{L^2(0,2T)}\|v\|_{H^1(0,2T)}$. Inserting this into \eqref{prrulzethv} and both into the first display for $\|v\|^2_{L^2(0,2T)}$, we obtain
\begin{align*}
\|v\|^2_{L^2(0,2T)}\lesssim &\, h\,\|v\|_{L^2(0,2T)}\|v\|_{H^1(0,2T)}\\
& +\|v\|_{H^{-1}(0,2T)}\Bigl(\|v\|_{H^1(0,2T)}+h^{-1/2}\|v\|^{1/2}_{L^2(0,2T)}\|v\|^{1/2}_{H^1(0,2T)}\Bigr).
\end{align*}
Assume $v\neq0$ and take $h=\rho\,\|v\|_{L^2(0,2T)}/\|v\|_{H^1(0,2T)}$ with a constant $\rho\in(0,T]$; this is admissible, since the quotient does not exceed one. The first term on the right becomes $\rho\,\|v\|^2_{L^2(0,2T)}$ up to the implicit constant, so for $\rho$ small enough it is absorbed into the left-hand side; and $h^{-1/2}\|v\|^{1/2}_{L^2(0,2T)}\|v\|^{1/2}_{H^1(0,2T)}=\rho^{-1/2}\|v\|_{H^1(0,2T)}$. Hence
$$\|v\|^2_{L^2(0,2T)}\lesssim\|v\|_{H^{-1}(0,2T)}\|v\|_{H^1(0,2T)},\qquad v\in H^1(0,2T),$$
with a constant depending only on $T$. Let $v=A[f]$, $f\in L^2(0,2T)$, so taking the supremum over $\|f\|_{L^2(0,2T)}\leqslant1$ gives the required estimate.
\end{proof}

Next, we show that the Neumann-to-Dirichlet map depends continuously on the spectral data. Recall that $\lesssim_\theta$ allows the constant to depend on $\theta$.

\begin{Proposition}\label{NDstab}
\textit{The map $\mathcal{D}\ni\Phi(c)\mapsto\Lambda^c\in\mathscr{B}(L^2(0,2T))$ is H\"older continuous with every exponent $\theta<\frac12$, that is, for any $\theta\in(0,\frac12)$ and $c_1,c_2\in\mathcal{C}$,
$$\|\Lambda^{c_1}-\Lambda^{c_2}\|_{L^2(0,2T)\to L^2(0,2T)}\lesssim_\theta\|\Phi(c_1)-\Phi(c_2)\|_{\ell^\infty_2\times\ell^\infty_0}^{\theta},$$
where the constant depends on $\theta$ and is not bounded as $\theta\to\frac12$. Moreover, if $\varepsilon=\|\Phi(c_1)-\Phi(c_2)\|_{\ell^\infty_2\times\ell^\infty_0}\leqslant\frac12$, then
$$\|\Lambda^{c_1}-\Lambda^{c_2}\|_{L^2(0,2T)\to L^2(0,2T)}\lesssim\big(\varepsilon\log(1/\varepsilon)\big)^{1/2}.$$}
\end{Proposition}
\begin{proof}
Set $\varepsilon=\|\Phi(c_1)-\Phi(c_2)\|_{\ell^\infty_2\times\ell^\infty_0}$. By \eqref{eq:Phi-norm}, $|\lambda^{c_1}_n-\lambda^{c_2}_n|\leqslant n^2\varepsilon$ and $|(\lambda^{c_1}_n)'(0)-(\lambda^{c_2}_n)'(0)|\leqslant\varepsilon$ for all $n\in\N$. Besides this, we use only $|(\lambda^{c_i}_n)'(0)|\leqslant C_{abs}$ from Lemma~\ref{eigderabs} and $\lambda^{c_i}_n\asymp n^2$ from Lemma~\ref{weylsn2}.

We first show
\begin{equation}\label{eq:Hminus1}
\|\Lambda^{c_1}-\Lambda^{c_2}\|_{L^2(0,2T)\to H^{-1}(0,2T)}\lesssim\varepsilon\log(1/\varepsilon),\qquad\varepsilon\leqslant\tfrac12 .
\end{equation}
Fix $f\in L^2(0,2T)$ and $\varphi\in H^1_0(0,2T)$ with $\|\varphi\|_{H^1(0,2T)}\leqslant1$. Write $S_\lambda[f](t)=\int_0^t k_\lambda(t-s)f(s)\,ds$ with $k_\lambda(\rho)=\frac{\sin(\rho\sqrt\lambda)}{\sqrt\lambda}$, and set
$$g_\lambda(\rho)=\int\limits_0^\rho k_\lambda(\zeta)\,d\zeta=\frac{1-\cos(\rho\sqrt\lambda)}{\lambda},\qquad \rho\geqslant0,$$
so that $g_\lambda'=k_\lambda$ and $g_\lambda(0)=0$. The integral $\int_0^{2T}S_\lambda[f](t)\varphi(t)\,dt$ is a double integral over the triangle $\{(s,t):0<s<t<2T\}$. By Fubini's theorem, integrating first in $t$, which for fixed $s$ ranges over $(s,2T)$, then substituting $\rho=t-s$ in the inner integral, and finally integrating by parts in $\rho$, where the boundary terms vanish since $g_\lambda(0)=0$ and $\varphi(2T)=0$, we get
\begin{align}
\int\limits_0^{2T}S_\lambda[f](t)\varphi(t)\,dt&=\int\limits_0^{2T}f(s)\int\limits_s^{2T}k_\lambda(t-s)\varphi(t)\,dt\,ds\nonumber\\
&=\int\limits_0^{2T}f(s)\int\limits_0^{2T-s}k_\lambda(\rho)\varphi(\rho+s)\,d\rho\,ds\nonumber\\
&=-\int\limits_0^{2T}f(s)\int\limits_0^{2T-s}g_\lambda(\rho)\,\varphi'(\rho+s)\,d\rho\,ds .\label{eq:S-ibp}
\end{align}
Let $\lambda,\mu\asymp n^2$. Since $|g_\lambda|\leqslant 2\lambda^{-1}$, $\|g_\lambda\|_{L^\infty(0,2T)}\lesssim n^{-2}$ and $\|g_\lambda-g_\mu\|_{L^\infty(0,2T)}\lesssim n^{-2}$. On the other hand, for fixed $\rho\in[0,2T]$ the map $\lambda\mapsto g_\lambda(\rho)$ is $C^1$ on $(0,+\infty)$ with
$$|\partial_\lambda g_\lambda(\rho)|=\Big|-\frac{1-\cos(\rho\sqrt\lambda)}{\lambda^2}+\frac{\rho\sin(\rho\sqrt\lambda)}{2\lambda^{3/2}}\Big|\leqslant\frac{2}{\lambda^2}+\frac{T}{\lambda^{3/2}}.$$
By the mean value theorem, $g_\lambda(\rho)-g_\mu(\rho)=\partial_\lambda g_{\tilde\lambda}(\rho)\,(\lambda-\mu)$ for some $\tilde\lambda$ between $\lambda$ and $\mu$; as $\tilde\lambda\asymp n^2$, we have $|\partial_\lambda g_{\tilde\lambda}(\rho)|\lesssim n^{-3}$. Together with the bound $\|g_\lambda-g_\mu\|_{L^\infty(0,2T)}\lesssim n^{-2}$ above,
$$\|g_\lambda-g_\mu\|_{L^\infty(0,2T)}\lesssim\min\{n^{-2},\,n^{-3}|\lambda-\mu|\}.$$
Now take $\lambda=\lambda^{c_1}_n$ and $\mu=\lambda^{c_2}_n$, so that $|\lambda-\mu|\lesssim n^2\varepsilon$ and the second bound becomes $\min\{n^{-2},\,n^{-1}\varepsilon\}$. 

We now return to \eqref{eq:S-ibp}. For $s\in(0,2T)$ the inner integral satisfies
$$\Bigg|\int\limits_0^{2T-s}g_\lambda(\rho)\,\varphi'(\rho+s)\,d\rho\Bigg|\leqslant\|g_\lambda\|_{L^\infty(0,2T)},$$
since $\|\varphi'\|_{L^1(0,2T)}\leqslant\|\varphi\|_{H^1(0,2T)}$. Hence, 
$$\Bigg|\int\limits_0^{2T}S_\lambda[f](t)\varphi(t)\,dt\Bigg|\lesssim\|g_\lambda\|_{L^\infty(0,2T)}\|f\|_{L^2(0,2T)} .$$
Subtracting \eqref{eq:S-ibp} for $\mu$ from \eqref{eq:S-ibp} for $\lambda$ replaces $g_\lambda$ by $g_\lambda-g_\mu$ in the inner integral, and the same argument gives
$$\Bigg|\int\limits_0^{2T}\big(S_\lambda-S_\mu\big)[f](t)\varphi(t)\,dt\Bigg|\lesssim\|g_\lambda-g_\mu\|_{L^\infty(0,2T)}\|f\|_{L^2(0,2T)} .$$
Taking the supremum over $\varphi$ and inserting the bounds on $g_\lambda$ and $g_\lambda-g_\mu$ for $\lambda=\lambda^{c_1}_n$, $\mu=\lambda^{c_2}_n$, we obtain
\begin{align}
\|S_{\lambda^{c_1}_n}[f]\|_{H^{-1}(0,2T)}&\lesssim n^{-2}\|f\|_{L^2(0,2T)},\nonumber\\
\big\|\big(S_{\lambda^{c_1}_n}-S_{\lambda^{c_2}_n}\big)[f]\big\|_{H^{-1}(0,2T)}&\lesssim\min\{n^{-2},\,n^{-1}\varepsilon\}\,\|f\|_{L^2(0,2T)} .\label{eq:S-Hminus1}
\end{align}

By Proposition~\ref{1mainres}, $\Lambda^{c_1}[f]-\Lambda^{c_2}[f]=I_1+I_2$, where
$$I_1=\sum_{n\in\N}\big((\lambda^{c_1}_n)'(0)-(\lambda^{c_2}_n)'(0)\big)\,S_{\lambda^{c_1}_n}[f],\quad I_2=\sum_{n\in\N}(\lambda^{c_2}_n)'(0)\,\big(S_{\lambda^{c_1}_n}-S_{\lambda^{c_2}_n}\big)[f].$$
The first bound in \eqref{eq:S-Hminus1} gives $\|I_1\|_{H^{-1}(0,2T)}\lesssim\varepsilon\sum_{n\in\N} n^{-2}\|f\|_{L^2(0,2T)}\lesssim\varepsilon\|f\|_{L^2(0,2T)}$. For $I_2$, let $N=\lceil\varepsilon^{-1}\rceil$ and use the second bound in \eqref{eq:S-Hminus1} with $n^{-1}\varepsilon$ for $n\leqslant N$ and with $n^{-2}$ for $n>N$:
\begin{align*}
\|I_2\|_{H^{-1}(0,2T)}&\lesssim C_{abs}\Big(\varepsilon\sum_{n\leqslant N}\frac1n+\sum_{n>N}\frac1{n^2}\Big)\|f\|_{L^2(0,2T)}\\
&\lesssim\big(\varepsilon(1+\log N)+N^{-1}\big)\|f\|_{L^2(0,2T)}\\
&\lesssim\varepsilon\log(1/\varepsilon)\|f\|_{L^2(0,2T)},
\end{align*}
where the last step uses $\varepsilon\leqslant\frac12$. This proves \eqref{eq:Hminus1}. By Lemma~\ref{hiddenreg}, the operator $A=\Lambda^{c_1}-\Lambda^{c_2}$ maps $L^2(0,2T)$ into $H^1(0,2T)$ with $\|A\|_{L^2(0,2T)\to H^1(0,2T)}\leqslant 2L$, so Lemma~\ref{interpolateA} applies and, together with \eqref{eq:Hminus1}, gives
$$\|\Lambda^{c_1}-\Lambda^{c_2}\|_{L^2(0,2T)\to L^2(0,2T)} \lesssim\big(\varepsilon\log(1/\varepsilon)\big)^{1/2},$$
which is the second estimate of the proposition.

Finally, fix $\theta\in(0,\frac12)$. For $\varepsilon\leqslant\frac12$, the second estimate and $\varepsilon\log(1/\varepsilon)\leqslant\big(e(1-2\theta)\big)^{-1}\varepsilon^{2\theta}$ give the first estimate with a constant proportional to $(1-2\theta)^{-1/2}$, which is unbounded as $\theta\to\frac12$. For $\varepsilon>\frac12$, Lemma~\ref{hiddenreg} and $H^1(0,2T)\hookrightarrow L^2(0,2T)$ give
$$\|\Lambda^{c_1}-\Lambda^{c_2}\|_{L^2(0,2T)\to L^2(0,2T)}\leqslant2L\leqslant2^{1+\theta}L\,\varepsilon^\theta .$$
The proposition is proven.
\end{proof}

We combine Proposition~\ref{NDstab} with the regularization strategy of \cite{korpela_lassas_oksanen}, and restrict the results of this subsection to the subclass $\mathcal{C}'\subset\mathcal{C}$ defined in Section~\ref{intro}. We first explain why the restriction is needed.

That strategy is formulated on the half line $\Omega'=(0,+\infty)$, where the wave speed equals one outside a compact subinterval of the interior and the waves emitted at the boundary never return. The reconstruction runs over control parameters $r\in[0,T]$ and recovers the value of the wave speed at the point $\chi(r)$ for each such $r$. In the present model the Dirichlet condition at $x=1$ reflects the waves, and the Blagove\v{s}\v{c}enski\u{i} identity $(u^f(T),1)_{L^2_c}=(f,B1)_{L^2(0,2T)}$, on which the reconstruction of the travel-time volume rests, holds only for $f\in H(r)$ with $r\leqslant\tau(1)$, see Theorem~\ref{blagid}. The admissible range of the control parameter is thus limited by the travel time $\tau(1)$, which is itself unknown, as it depends on the wave speed one is trying to recover. The reconstruction of \cite{korpela_lassas_oksanen} is assembled from the measured Neumann-to-Dirichlet map alone, with the control parameter ranging over an interval fixed before the measurement, so the travel time $\tau(1)$ cannot enter it. That interval must therefore be chosen once and for all, so that its right endpoint (range) $r_{\max}$ satisfies $r_{\max}<\tau(1)$ and $\Omega(r_{\max})\supseteq\operatorname{supp}(c-1)$.

No such range exists for the class $\mathcal{C}$. For $c\in\mathcal{C}$ the set $\operatorname{supp}(c-1)$ is closed and may have $x=1$ as its right endpoint. For such a wave speed the requirement $\Omega(r_{\max})\supseteq\operatorname{supp}(c-1)$ forces $r_{\max}\geqslant\tau(1)$, which is incompatible with $r_{\max}<\tau(1)$. The two requirements are therefore not satisfiable on $\mathcal{C}$. This is what the subclass $\mathcal{C}'$ repairs. Its first condition places $\operatorname{supp}(c-1)$ inside $[0,1-d]$, and its second one calibrates $d$ so that a single $r_{\max}$ serves every admissible wave speed. Indeed, put $r_{\max}=\frac{1-d}{C_{min}}$ and let $c\in\mathcal{C}'$. Since $c\equiv 1$ on $[1-d,1]$ and $C_{min}\leqslant c\leqslant C_{max}$, we have
$$\tau(1-d)\leqslant r_{\max}<\tau(1).$$
The last inequality holds due to $d>\frac{C'}{C'+1}$. Hence $\operatorname{supp}(c-1)\subseteq\Omega(r_{\max})$. 

We therefore run the construction of \cite{korpela_lassas_oksanen} over $r\in[0,r_{\max}]$ instead of $r\in[0,T]$. Every control $f$ used in that construction is then supported in $[T-r_{\max},T]$, so by the first inequality above the corresponding waves do not reach the
endpoint $x=1$ before the observation time $T$. The identity $(u^f(T),1)_{L^2_c}=(f,B1)_{L^2(0,2T)}$ is available throughout and each quantity entering the construction coincides with its counterpart on the half line, even though $\Lambda$ itself differs from the half line map at times beyond $2\tau(1)$, which lie in the observation window since $T\geqslant\tau(1)$. Furthermore, $\operatorname{supp}(c-1)\subseteq \Omega(r_{\max})$, so the values $c(\chi(r))$, $r\in[0,r_{\max}]$, already determine $c$ on all of $\Omega$, and nothing is lost by stopping at $r_{\max}$.

The estimates of \cite{korpela_lassas_oksanen} are uniform in $r$ and are therefore unaffected by the restriction to $\mathcal{C}'$. Some of the auxiliary controls constructed there are supported in $[T-r',T]$ with $r'\leqslant r+C_{min}^{-1}\alpha^{1/2}$, where $\alpha$ is the regularization parameter. Since $r\leqslant r_{\max}$ and $\tau(1)-r_{\max}\geqslant\frac{1-d}{C_{max}}+d-\frac{1-d}{C_{min}}>0$, these controls also satisfy $r'\leqslant\tau(1)$ as soon as $\alpha$ is below a threshold determined by $C_{min}$, $C_{max}$ and $d$. 

Let $c\in\mathcal{C}'$ be the true wave speed and let $\Lambda\in\mathscr{B}(L^2(0,2T))$ be a measurement of $\Lambda^{c}$ with noise level $\delta>0$, that is, $\|\Lambda-\Lambda^{c}\|_{L^2(0,2T)\to L^2(0,2T)}\leqslant\delta$. Applying to $\Lambda$ the regularization strategy of \cite{korpela_lassas_oksanen} produces a wave speed $\widehat{c}\in C^2(\overline\Omega)$, which in general does not belong to $\mathcal{C}$, and Theorem~3 of \cite{korpela_lassas_oksanen} states that
\begin{equation}\label{eq:KLO}
\|\widehat{c}-c\|_{C^2(\overline\Omega)}\leqslant C\delta^{\kappa},
\end{equation}
where $\kappa=1/54$ and $C>0$ depends only on $C_{min}$, $C_{max}$ and $T$. Since $\widehat{c}$ is determined by $\Lambda$ and $\delta$ alone, \eqref{eq:KLO} holds for every $c\in\mathcal{C}'$ whose Neumann-to-Dirichlet map $\Lambda^c$ lies within $\delta$ of $\Lambda$. We now prove the inverse stability result, Theorem~\ref{5mainres}, which we restate.

\medskip
\noindent\textbf{Theorem~\ref{5mainres}.}
\textit{The inverse spectral map $\Phi^{-1}:\mathcal{D}'\to\mathcal{C}'$ is H\"older continuous from $\ell^\infty_2\times\ell^\infty_0$ to $C^2(\overline\Omega)$ with every exponent $\theta<\frac{\kappa}{2}$, that is, for any $c_1,c_2\in\mathcal{C}'$,
$$\|c_1-c_2\|_{C^2(\overline\Omega)}\lesssim_{\theta}\|\Phi(c_1)-\Phi(c_2)\|^{\theta}_{\ell^\infty_2\times\ell^\infty_0},\quad \theta\in \left(0,\frac\kappa2\right).$$
Moreover, if $\varepsilon=\|\Phi(c_1)-\Phi(c_2)\|_{\ell^\infty_2\times\ell^\infty_0}\leqslant\frac12$, then 
$$\|c_1-c_2\|_{C^2(\overline\Omega)}\lesssim\big(\varepsilon\log(1/\varepsilon)\big)^{\kappa/2}.$$}
\begin{proof}
Fix $c_1,c_2\in\mathcal{C}'$ and set $\delta=\|\Lambda^{c_1}-\Lambda^{c_2}\|_{L^2(0,2T)\to L^2(0,2T)}$. Then $\Lambda^{c_2}$ is a measurement of $\Lambda^{c_1}$ with noise level $\delta$; let $\widehat{c}$ be the wave speed reconstructed from $\Lambda^{c_2}$ at noise level $\delta$ by the strategy of \cite{korpela_lassas_oksanen}. Since $\Lambda^{c_2}$ lies within $\delta$ of both $\Lambda^{c_1}$ and $\Lambda^{c_2}$, estimate \eqref{eq:KLO} applies with $c=c_1$ and with $c=c_2$, giving $\|\widehat{c}-c_1\|_{C^2(\overline\Omega)}\leqslant C\delta^\kappa$ and $\|\widehat{c}-c_2\|_{C^2(\overline\Omega)}\leqslant C\delta^\kappa$, hence
$$\|c_1-c_2\|_{C^2(\overline\Omega)}\leqslant2C\delta^\kappa$$
by the triangle inequality. It remains to bound $\delta$ by $\varepsilon=\|\Phi(c_1)-\Phi(c_2)\|_{\ell^\infty_2\times\ell^\infty_0}$. For $\theta\in(0,\frac\kappa2)$, the first estimate of Proposition~\ref{NDstab} with the exponent $\theta/\kappa\in(0,\frac12)$ gives $\delta\lesssim_\theta\varepsilon^{\theta/\kappa}$, hence $\|c_1-c_2\|_{C^2(\overline\Omega)}\lesssim_{\theta}\varepsilon^{\theta}$. If $\varepsilon\leqslant\frac12$, the second estimate of Proposition~\ref{NDstab} gives $\delta\lesssim(\varepsilon\log(1/\varepsilon))^{1/2}$, hence $\|c_1-c_2\|_{C^2(\overline\Omega)}\lesssim (\varepsilon\log(1/\varepsilon))^{\kappa/2}$.
\end{proof}

Finally, we prove Corollary~\ref{6mainres}. Since $\|c\|_{C^3(\R)}\leqslant C_{max}$ for the wave speeds considered here, the second derivative of the difference $w=c_1-c_2$ is uniformly Lipschitz. Interpolating between the uniform bound on $w''$ and this Lipschitz bound controls the H\"older seminorm of $w''$ of any order $\eta<1$ by the power $1-\eta$ of $\|w\|_{C^2(\overline\Omega)}$, and thereby converts the $C^2(\overline\Omega)$ estimate of Theorem~\ref{5mainres} into a $C^{2,\eta}(\overline\Omega)$ estimate with the exponent multiplied by $1-\eta$, decreasing linearly to zero as $\eta\to1$. Nothing better is available: the bound $\|c\|_{C^3(\R)}\leqslant C_{max}$ does not force the third derivatives of two uniformly close wave speeds to be close, so no estimate in $C^{2,1}(\overline\Omega)$ or $C^3(\overline\Omega)$ can follow from Theorem~\ref{5mainres} alone.

\medskip
\noindent\textbf{Corollary~\ref{6mainres}.}
\textit{For any $c_1,c_2\in\mathcal{C}'$ and $\eta\in[0,1)$,}
$$\|c_1-c_2\|_{C^{2,\eta}(\overline\Omega)}\lesssim_{\theta}
\|\Phi(c_1)-\Phi(c_2)\|^{\theta(1-\eta)}_{\ell^\infty_2\times\ell^\infty_0} ,\quad \theta\in \left(0,\frac\kappa2\right).$$
\begin{proof}
Write $w=c_1-c_2$ and $A=\|w\|_{C^2(\overline\Omega)}$. Since $\|c_j\|_{C^3(\R)}\leqslant C_{max}$, $j=1,2$, we have $A\leqslant 2C_{max}$ and
$\|w'''\|_{L^\infty(\Omega)}\leqslant 2C_{max}$.

Let $\eta\in(0,1)$ and $x,y\in\overline\Omega$, $x\neq y$. Bounding the increment of $w''$
in two ways, by its supremum and by the mean value theorem, gives
$$|w''(x)-w''(y)|\leqslant\min\bigl\{2A,\ 2C_{max}|x-y|\bigr\}
\leqslant(2A)^{1-\eta}\bigl(2C_{max}|x-y|\bigr)^{\eta},$$
where the last step uses $\min\{a,b\}\leqslant a^{1-\eta}b^{\eta}$ for $a,b\geqslant0$.
Dividing by $|x-y|^\eta$ and taking the supremum,
$$\sup_{x\neq y}\frac{|w''(x)-w''(y)|}{|x-y|^{\eta}}\leqslant
2(2C_{max})^{\eta}A^{1-\eta}.$$
Moreover $A=A^{1-\eta}A^{\eta}\leqslant(2C_{max})^{\eta}A^{1-\eta}$, so
$$\|w\|_{C^{2,\eta}(\overline\Omega)}\leqslant 3(2C_{max})^{\eta}A^{1-\eta}\lesssim A^{1-\eta},$$
the constant being bounded by $3\max\{1,2C_{max}\}$, hence independent of $\eta$.

By Theorem~\ref{5mainres}, $A\lesssim_\theta\varepsilon^{\theta}$ for
$\theta\in(0,\frac\kappa2)$, where
$\varepsilon=\|\Phi(c_1)-\Phi(c_2)\|_{\ell^\infty_2\times\ell^\infty_0}$. Raising this to
the power $1-\eta\in(0,1)$ and combining with the previous display gives $\|w\|_{C^{2,\eta}(\overline\Omega)}\lesssim_\theta\varepsilon^{\theta(1-\eta)}$. The case $\eta=0$ is Theorem~\ref{5mainres} itself.
\end{proof}

\begin{Remark}
The exponent $\kappa=1/54$ is inherited from the regularization strategy of \cite{korpela_lassas_oksanen}. We do not address lower bounds for the modulus of continuity of $\Phi^{-1}$; in particular, the question whether $\Phi^{-1}$ is H\"older continuous with an exponent independent of the regularization strategy, or even Lipschitz on suitable subclasses, is left open.
\end{Remark}

\medskip

\section{Computational implementation}\label{compimpl}
\subsection{Forward problem: finite element method}
We now describe the finite element discretization used to compute the spectral data for the forward problem. The goal is to approximate the eigenvalues of the weak formulation of $(\mathcal{S}_\omega)$. The numerical solution of the forward spectral problem is carried out using the finite element method (FEM). The finite element method is a Galerkin approximation scheme in which the infinite-dimensional variational problem is replaced by a finite-dimensional one on a suitable subspace $V_h\subset V$. For the general theory of finite element approximations of eigenvalue problems, we refer to D. Boffi~\cite{boffi}.

Let $N\geqslant 2$ and set $h=1/N$. We consider the uniform partition
\[
0=x_0<x_1<\cdots<x_N=1,\qquad x_k=kh,
\]
and denote the elements by
\[
I_k=(x_k,x_{k+1}),\qquad k=0,\dots,N-1.
\]
We define
\[
V_h
=\left\{
v_h\in C(\overline{\Omega}) :
v_h(1)=0,
v_h|_{I_k}\in\mathcal{P}^1(I_k),\ k=0,\dots,N-1
\right\}.
\]
Here $\mathcal{P}^1(I_k)$ denotes the space of polynomials of degree at most one on $I_k$. Equivalently, one first considers the larger space
\[
\widehat V_h=
\left\{
\widehat v_h\in C(\overline{\Omega}):
\widehat v_h|_{I_k}\in\mathcal{P}^1(I_k),\ k=0,\dots,N
\right\},
\]
with standard nodal basis of
$\widehat V_h$ is given by the hat functions $\{\widehat\varphi_\ell\}_{\ell=0}^N$, defined by their nodal values
\[
\widehat\varphi_\ell(x_k)=
\begin{cases}
	1, & \ell=k,\\
	0, & \ell\neq k.
\end{cases}
\]
Each $\widehat\varphi_\ell$ is a continuous piecewise linear ``tent'' function. For an interior node
$x_\ell$, $1\leqslant \ell\leqslant N-1$, the function $\widehat\varphi_\ell$ is supported on $I_{\ell-1}\cup I_\ell=(x_{\ell-1},x_\ell)\cup(x_\ell,x_{\ell+1})$. Thus $\widehat\varphi_\ell$ is triangular: it increases linearly from $0$ to $1$ on
$I_{\ell-1}$, decreases linearly from $1$ to $0$ on $I_\ell$, and vanishes outside
$I_{\ell-1}\cup I_\ell$. At the boundary nodes, $\supp\widehat\varphi_0\subset I_0$ and $\supp\widehat\varphi_N\subset I_{N-1}$. Since the Dirichlet condition at $x_N=1$ requires $v_h(1)=0$, the last basis function
$\widehat\varphi_N$ is omitted from the approximation space $V_h$. Hence the active degrees of
freedom are associated with the nodes $x_0,\dots,x_{N-1}$, and
\[
V_h=\operatorname{span}\{\widehat\varphi_0,\dots,\widehat\varphi_{N-1}\}.
\]

The finite element approximation of $(\mathcal{S}_\omega)$, $\omega\in [0,1)$, is obtained by the Ritz--Galerkin method. Namely, we seek
$(\lambda_h,\phi_h)\in\mathbb{R}\times V_h$, with $\phi_h\neq 0$, satisfying
\[
B_\omega(\phi_h,v_h)=\lambda_h(\phi_h,v_h)_{L^2_c}
\qquad  v_h\in V_h.
\]
Every $\phi_h\in V_h$ can be written as
\[
\phi_h=\sum_{\ell=0}^{N-1} \alpha_\ell\widehat\varphi_\ell.
\]
Substituting this expansion into the Galerkin formulation gives
\[
\sum_{\ell=0}^{N-1} \alpha_\ell B_\omega(\widehat\varphi_\ell,v_h)
=
\lambda_h
\sum_{\ell=0}^{N-1}\alpha_\ell(\widehat\varphi_\ell,v_h)_{L^2_c},
\qquad v_h\in V_h.
\]
Taking $v_h=\widehat\varphi_j$, $j=0,\dots,N-1$, we obtain the finite system
\[
\sum_{\ell=0}^{N-1} \alpha_\ell B_\omega(\widehat\varphi_\ell,\widehat\varphi_j)
=
\lambda_h
\sum_{\ell=0}^{N-1} \alpha_\ell(\widehat\varphi_\ell,\widehat\varphi_j)_{L^2_c},
\qquad j=1,\dots,L.
\]
Thus, for the stiffness matrix $\mathsf{B}^\omega_{j\ell}=B_\omega(\widehat\varphi_\ell,\widehat\varphi_j)$ and the mass matrix $\mathsf{M}_{j\ell}=(\widehat\varphi_\ell,\widehat\varphi_j)_{L^2_c}$, the discrete eigenvalue problem takes the generalized matrix form
\[
\mathsf{B}^\omega\alpha=\lambda_h\mathsf{M}\alpha,
\]
where $\alpha=(\alpha_1,\dots,\alpha_L)^\top$. Solving this generalized eigenvalue problem gives approximations of the eigenvalues and eigenfunctions of $(\mathcal{S}_\omega)$. To assemble these matrices, we compute local element contributions. 

Fix an interval $I_k$ for some $k=0,\dots,N-1$. On $I_k$, the two local basis functions are the restrictions of the global hat functions associated with the endpoints $x_k$ and $x_{k+1}$. We denote them by
\[
\varphi^{(k)}_0(x)=\frac{x_{k+1}-x}{h},
\qquad
\varphi^{(k)}_1(x)=\frac{x-x_k}{h},
\qquad x\in I_k.
\]
Thus $\varphi^{(k)}_0$ is equal to $1$ at the left endpoint $x_k$ and $0$ at the right endpoint $x_{k+1}$, while $\varphi^{(k)}_1$ is equal to $0$ at $x_k$ and $1$ at $x_{k+1}$. These are the element shape functions. They are simply the restrictions of the corresponding global hat functions to the interval $I_k$. The local stiffness matrix on $I_k$ is
\[
\mathsf{B}^{(k)}_{ab}
=\int_{x_k}^{x_{k+1}}
(\varphi^{(k)}_a)'(x)(\varphi^{(k)}_b)'(x)\,dx
,\qquad a=0,1,\;b=0,1.
\]
Since
\[
(\varphi^{(k)}_0)'(x)=-\frac{1}{h},
\qquad
(\varphi^{(k)}_1)'(x)=\frac{1}{h},
\]
we obtain the standard formula
\[
\mathsf{B}^{(k)}
=
\frac{1}{h}
{\renewcommand{\arraystretch}{0.85}\setlength{\arraycolsep}{3pt}
\begin{pmatrix}
	\phantom{-}1 & -1\\
	-1 & \phantom{-}1
\end{pmatrix}}.
\]
The local stiffness matrices above first give the stiffness matrix $\mathsf{B}^0$ for the unperturbed case
$\omega=0$. It is assembled in the standard way:
for each element $I_k=(x_k,x_{k+1})$, the local nodes $0$ and $1$ correspond to the global
nodes $k$ and $k+1$. Hence the local entry $\mathsf{B}^{(k)}_{ab}$ is added to the global
entry $\mathsf{B}^0_{k+a,k+b}$, whenever the global indices $k+a$ and $k+b$ belong to $\{0,\dots,N-1\}$. The node $x_N=1$ is not active, since the Dirichlet condition
$\phi(1)=0$ is imposed. Therefore the assembled matrix $\mathsf{B}^0$ has size
$N\times N$.

For $\omega>0$, the bilinear form is $B_\omega(u,v)
=\int_0^1 u'(x)v'(x)\,dx-\omega u(0)v(0)$. The first term is exactly the one used to assemble $\mathsf{B}^0$. The only additional
contribution is the Robin boundary term $-\omega u(0)v(0)$. Since this term is supported
only at the endpoint $x=0$, it affects only the basis function associated with the node
$x_0$, since $\widehat\varphi_0(0)=1$ and $\widehat\varphi_j(0)=0$ for $j\neq 0$. Thus the perturbed stiffness matrix $\mathsf{B}^\omega$ differs from the unperturbed one
$\mathsf{B}^0$ only in its first diagonal entry:
\[
\mathsf{B}^{\omega}_{00}=\mathsf{B}^{0}_{00}-\omega,
\qquad
\mathsf{B}^{\omega}_{j\ell}=\mathsf{B}^{0}_{j\ell},
\quad (j,\ell)\neq(0,0).
\]
Equivalently, all dependence on the Robin impedance parameter $\omega$ is contained in
the single entry $\mathsf{B}^{\omega}_{00}$.

The local mass matrix on $I_k$ is
\[
\mathsf{M}^{(k)}_{ab}
=\int_{x_k}^{x_{k+1}}
\varphi^{(k)}_a(x)\varphi^{(k)}_b(x)\,\frac{dx}{c^2(x)}
,\qquad a=0,1,\;b=0,1.
\]
This is the only place where the wave speed \(c\) enters the computation of the spectral data. To compute the latter integral, we use the two-point Gauss--Legendre quadrature rule with
\[
\xi_1=-\frac{1}{\sqrt{3}},
\qquad
\xi_2=\frac{1}{\sqrt{3}},
\qquad
w_1=w_2=1,
\]
and let
\[
x(\xi)=\frac{x_k+x_{k+1}}{2}+\frac{h}{2}\xi,\qquad -1\leqslant \xi\leqslant 1.
\]
We then approximate the entries of the local mass matrix on $I_k$ by
\[
\mathsf{M}^{(k)}_{ab}
\approx
\frac{h}{2}
\sum_{q=1}^{2}
\varphi_a^{(k)}(x(\xi_q))\varphi_b^{(k)}(x(\xi_q))\,
\frac{w_q}
{c^2(x(\xi_q))},
\qquad a=0,1,\;b=0,1.
\]
This gives the local $2\times 2$ contribution to the global mass matrix.

The global mass matrix $\mathsf{M}\in\mathbb{R}^{N\times N}$ is assembled in the same way as the stiffness matrix. For each element $I_k=(x_k,x_{k+1})$, the local nodes $0$ and $1$ correspond to the global nodes $k$ and $k+1$. Hence the local entry $\mathsf{M}^{(k)}_{ab}$ is added to the global entry $\mathsf{M}_{k+a,k+b}$, whenever the global indices $k+a$ and $k+b$ are active degrees of freedom, that is, belong to $\{0,\dots,N-1\}$. Since the Dirichlet condition at $x_N=1$ is imposed, the last node is omitted from the set of unknowns. Therefore, on the last element $I_{N-1}=(x_{N-1},x_N)$, only the contributions involving the active node $x_{N-1}$ are retained. After all element contributions have been added, the assembled matrix $\mathsf{M}$ represents the weighted inner product $(\cdot,\cdot)_{L^2_c}$ on the finite element space $V_h$.

\medskip

After assembly, we obtain matrices $\mathsf{B}^\omega,\mathsf{M}\in\mathbb{R}^{N\times N}$ and solve the discrete spectral problem
\[
\mathsf{B}^\omega\alpha=\lambda_h\mathsf{M}\alpha
\]
for $\omega=0$ and some sufficiently small $\omega\in (0,1)$. Since $\mathsf{M}$ is symmetric positive-definite and $\mathsf{B}^\omega$ is symmetric, the discrete eigenvalues are real, and the eigenvectors may be chosen $\mathsf{M}$-orthogonal. We now have, for a prescribed wave speed $c(x)$ and a fixed $\omega$, two finite sets of eigenvalues:
\[
\{\lambda_n(0)\}_{n=1}^L
\qquad\text{and}\qquad
\{\lambda_n(\omega)\}_{n=1}^L.
\]
This is the simulated spectral data.

\medskip

\subsection{Inverse problem: reconstruction of the wave speed via the boundary control method}\label{iprotwsvtbcm}
We now describe the numerical reconstruction of the wave speed from the spectral data by the boundary control method, following Theorem~\ref{3mainres}. The first step is to fix the observation time \(T\). It depends on the travel time \(\tau(1)=\int_0^1 c(x)^{-1}\,dx\), which is unknown and has to be estimated from the data.

By the Liouville transform of Section~\ref{stabsec}, $(\mathcal{S}_0)$ is equivalent to $-\psi''+q\psi=\lambda\psi$ on $(0,\tau(1))$, with a bounded potential $q$, $\psi'(0)=0$ and $\psi(\tau(1))=0$. For $q=0$ the eigenvalues are $\bigl(\pi(n-\tfrac12)/\tau(1)\bigr)^2$, and a bounded $q$ shifts them by its mean value, which gives the classical asymptotics
\[
\sqrt{\lambda_n(0)}=\frac{\pi\left(n-\tfrac12\right)}{\tau(1)}+\frac{b}{n}+o\!\left(n^{-1}\right),\qquad b=\frac{1}{2\pi}\int_0^{\tau(1)}q\,d\tau .
\]
We fit the model $\sqrt{\lambda_n(0)}\approx\pi(n-\tfrac12)/\tau(1)+a+b/n$ of the three parameters \(\tau(1)\), \(a\), \(b\) by least squares to the computed eigenvalues \(\lambda_n(0)\), $5\leqslant n\leqslant 40$. The additional constant $a$ absorbs the discretization error of the finite element eigenvalues; for the exact spectrum $a=0$. This gives an estimate \(\hat\tau(1)\) that uses the spectral data only.

The theory takes \(T\geqslant\tau(1)\) in order to recover \(c\) on all of \(\Omega\). In the computations we take instead
\[
T=(1-\eta)\,\hat\tau(1),\qquad \eta=0.03,
\]
slightly below \(\tau(1)\). This is admissible by Theorem~\ref{blagid}, since \(r\leqslant T<\tau(1)\). We then recover \(c\) on \(\Omega(T)=[0,\chi(T)]\) and continue it on the short remaining interval \((\chi(T),1]\) by its last computed value \(\hat c(\chi(T))\).

The margin \(\eta\) keeps the front reflected at \(x=1\), which returns to \(x=0\) at \(t=2\tau(1)\), outside the observation window \([0,2T]\). It also keeps out the Gibbs oscillations that the truncated series \(\Lambda_L^\omega\) produces around this front. These oscillations have period of order \(\tau(1)/L\), set in slightly before \(t=2\tau(1)\), and would otherwise corrupt \(V(r)\) for \(r\) close to \(\tau(1)\). Hence \(\eta\) must exceed both the relative error of \(\hat\tau(1)\), below \(1\%\) in our experiments, and the width of these oscillations. The value \(\eta=0.03\) minimizes the largest relative \(L^2\) error over the smooth test profiles among \(\eta\in[0.005,0.065]\). Smaller values make the errors grow rapidly, while larger values lengthen the interval \((\chi(T),1]\) on which \(c\) is only continued constantly.

With \(T\) fixed, we pass from the spectral data to the Neumann-to-Dirichlet map on \(L^2(0,2T)\). The link is given by Proposition~\ref{1mainres}, whose formula for \(\Lambda\) uses the derivatives \(\lambda_n'(0)\). In the numerical implementation these derivatives are replaced by the finite-difference approximations
\[
\lambda_n'(0)\approx\frac{\lambda_n(\omega)-\lambda_n(0)}{\omega},\qquad n=1,\dots,L,
\]
whose error is \(O(\omega\sup_{[0,\omega]}|\lambda_n''|)\) by Taylor's theorem, since the second derivative exists by Proposition~\ref{lem2}. This modelling error is introduced by the numerical scheme alone and is not covered by the stability estimates of Section~\ref{stabsec}, which take \(\{\lambda_n(0),\lambda_n'(0)\}\) as exact input. The computed spectral data thus approximate the coefficients of the convolution representation of \(\Lambda\), and in the computations we use the discrete truncated Neumann-to-Dirichlet map
\[
\Lambda_L^\omega[f]=\sum_{n=1}^{L}\frac{\lambda_n(\omega)-\lambda_n(0)}{\omega}\,S_{\lambda_n(0)}[f],\qquad f\in L^2(0,2T).
\]
This operator is the numerical approximation of \(\Lambda\) used in the reconstruction.

The next step is to pass from the boundary operator \(\Lambda\) to the Blagove\v{s}\v{c}enski\u{\i} operator \(K\). In the numerical implementation, the exact map $\Lambda$ is replaced by its finite-rank approximation $\Lambda_L^\omega$, obtained from the two finite element spectra. Thus we form a discrete approximation of
\[
K=R\Lambda RJ-J\Lambda,
\]
where $J$ and $R$ are defined in Section \ref{premip1}.

We discretize the time interval $[0, 2T]$ by a uniform grid with time step $\Delta t$ and use piecewise linear hat functions as a basis for the boundary controls. With respect to this basis, the operators $J$, $R$, $\Lambda_L^\omega$, and hence $K$, are represented by matrices. Fix sufficiently small regularization parameters $\alpha,\beta>0$. By Proposition~\ref{2mainres}, the regularized control
\[
f_{\alpha,\beta,r} = (\mathcal{P}_rK\mathcal{P}_r + \alpha I + \beta S)^{-1}\mathcal{P}_rB1
\]
approximates the indicator $\mathds{1}_{\Omega(r)}$ uniformly in $r\in[0,\tau(1)]$, so that the same pair $(\alpha,\beta)$ is used for all values of $r$, as discussed in Section~\ref{potpamr}. The resulting control $f_{\alpha,\beta,r}$ is used to compute the scalar quantity
\[
V(r) \approx (f_{\alpha,\beta,r}, B1)_{L^2(0,2T)}.
\]
Repeating this procedure for a collection of values $\{r_j\}_{j=0}^m \subset [0,\tau(1)]$ produces a discrete profile $\{(r_j, V(r_j))\}_{j=0}^m$. 

The final step is to recover the wave speed from this volume profile. This step uses the identities
\[
V(r)=\int\limits_{\Omega(r)}\frac{dx}{c^2(x)},
\qquad
\Omega(r)=\{x\in\Omega:d(0,x)\leqslant r\},
\]
where $d(0,x)=\tau(x)$ is the travel-time distance. In the view of the proof of Theorem~\ref{3mainres}, we get
\[
c(\chi(r))=\frac{1}{V'(r)},\quad
\chi(r)=\int_0^r \frac{ds}{V'(s)}.
\]
Therefore the numerical reconstruction of $c(x)$ requires a stable approximation of the derivative $V'(r)$.

In the implementation, the differentiation of the discrete volume function is treated as a regularized inverse problem. Let
\[
Y\,:\,L^2(0,\tau(1))\to L^2(0,\tau(1)),
\qquad
(Yg)(r)=\int_0^r g(t)\,dt,
\]
be the Volterra integration operator. By the fundamental theorem of calculus,
\[
Y\left[V'\right]=V,\qquad V(0)=0.
\]
Hence recovering \(V'\) from noisy discrete values of \(V\) can be viewed as inverting the integration operator \(Y\), which is an ill-conditioned operation.

In the discrete setting, we replace \(Y\) by a cumulative integration matrix, constructed using the trapezoidal rule on the time mesh. Thus, for a vector \(g=(g_j)\) approximating \(V'\), the vector \(Yg\) approximates the cumulative integral of \(g\). In the simplest left-sum form this has the structure
\[
(Yg)_i \approx \sum_{j=0}^{i-1} g_j\,\Delta t,
\qquad
(Yg)_0=0,
\]
while in the actual implementation the endpoint weights are chosen according to the trapezoidal rule. If the computed volume data contain numerical errors, say \(V+\varepsilon\), then numerical differentiation is formulated as the inverse problem of finding \(g\) such that $Yg\approx V$.

To stabilize this inversion, three options are available. The first is to fit a smoothing spline to the data \(V\) and then differentiate the spline analytically. The second is an \(H^1\)-Tikhonov regularization. The third option is total variation (TV) regularization, which is useful when \(V'\), and therefore the reconstructed wave speed, may have sharp transitions. In this case we solve
\[
\min_g \left\{ \|Yg-V\|_2^2+\alpha_{\mathrm{TV}}\|D_t g\|_1 \right\},
\]
where \(D_t\) is the forward finite-difference operator on the time mesh and
\(\|D_t g\|_1\) is the discrete total variation of \(g\). The parameter $\alpha_{\mathrm{TV}}$ stabilizes the differentiation of $V$ and is unrelated to the parameters $\alpha,\beta$ of the control problem. Thus the first term enforces consistency with the computed volume data, while the TV penalty suppresses oscillations in the numerical derivative without overly smoothing jumps. In the numerical experiments of Section~\ref{numres}, the spline and $H^1$-Tikhonov options are used for smooth wave speed profiles, while the TV option is reserved for piecewise constant profiles (these profiles lie outside the admissible class $\mathcal{C}$ and are included to
probe the behaviour of the method). In the latter case, $V'$ is itself discontinuous and spline or $H^1$ regularization would artificially smear the~jumps.

After computing a stable approximation of $V'$, we evaluate
\[
\chi(r)=\int_0^r\frac{ds}{V'(s)}
\]
by numerical quadrature. The map $\chi$ is inverted by interpolation to obtain $r=\tau(x)$, and the wave speed is finally reconstructed on a spatial grid by
\[
c(x)=\frac{1}{V'(\tau(x))}.
\]

The scheme involves several parameters: the truncation $L$ of the spectra, the impedance step $\omega$ used to approximate $\lambda_n'(0)$, the spatial mesh size $h=1/N$, the time step $\Delta t$, the margin $\eta$ in $T=(1-\eta)\hat\tau(1)$, and the regularization parameters $\alpha,\beta$ of the control problem. We do not prove convergence of the fully discrete scheme, which would require coupling these parameters to one another. Such an analysis, coupling the discretization and the regularization parameter to the noise level, is carried out in \cite{KLO1} for a related problem on the half line; the truncation of the spectra and the impedance step, which are specific to the present scheme, are not covered there. In the following section the parameters are chosen empirically.

\medskip 

\section{Numerical results}\label{numres}
In this section we examine the performance of the reconstruction algorithm described in Section~\ref{compimpl} on simulated data. The forward problem is solved numerically by the finite element method, and the resulting synthetic spectral data $\{\lambda_n(0)\}_{n=1}^{L}$ and $\{\lambda_n(\omega)\}_{n=1}^{L}$ serve as input to the inverse solver. The experiments of Sections~\ref{sec:sources} and~\ref{sec:profiles} are carried out on noise-free data, so that the results reflect the intrinsic accuracy of the method. Its sensitivity to measurement noise is studied separately in Section~\ref{sec:noise}. There, independent random perturbations are added to both spectra, and the derivatives $\lambda_n'(0)$ are computed from the perturbed eigenvalues exactly as in the noise-free case.

\medskip

\subsection{Reconstruction of Neumann boundary sources}\label{sec:sources}
The reconstruction of the wave speed rests on Proposition~\ref{2mainres}, namely the source $f_{\alpha,\beta,r}$, computed from the spectral data alone, generates a wave whose snapshot $u^{f_{\alpha,\beta,r}}(T)$ approximates the indicator $\mathds{1}_{\Omega(r)}$ of the domain of influence. We verify this property numerically. The aim is not to verify the proposition itself, which is proved above. Rather, we check whether the source actually computed from the discrete and truncated spectral data produces a wave $u^{f_{\alpha,\beta,r}}(T)$ that is close to $\mathds{1}_{\Omega(r)}$. 

We carry out this test in the simple case \(c(x)\equiv 1\) on \(\Omega\). In this case \(\tau(x)=x\), so \(\Omega(r)=[0,r]\), and \(u=u^f\) solves
\[
\begin{cases}
\partial_t^2 u-\partial_x^2 u=0, & (t,x)\in(0,T)\times\Omega,\\
\partial_x u(t,0)=f(t),\quad u(t,1)=0, & t\in(0,T),\\
u(0,x)=\partial_t u(0,x)=0, & x\in\Omega .
\end{cases}
\]
Let \(G:\mathbb{R}\to\mathbb{R}\) be defined by $G(s)=-\int_0^{s} f(s')\,ds'$ for $s>0$, $G(s)=0$ for $s\leqslant 0$, and set \(u(t,x)=G(t-x)\), a wave travelling to the right from \(x=0\) with speed one. This function satisfies the wave equation. It vanishes together with \(\partial_t u\) at \(t=0\), since \(t-x=-x\leqslant 0\) there. It satisfies \(u(t,1)=G(t-1)=0\) for \(t\leqslant T\leqslant 1\), because the wave has not yet reached \(x=1\). Finally, \(\partial_x u(t,0)=-G'(t)=f(t)\). By uniqueness, \(u=u^f\), and therefore
\[
u^f(T,x)=G(T-x)=-\int_0^{T-x} f(t)\,dt, \qquad x\in\Omega .
\]
Choosing \(f=-\delta_{T-r}\), where \(\delta_{T-r}\) denotes the Dirac delta at \(t=T-r\), we obtain
\[
u^f(T,x)=\int_0^{T-x}\delta_{T-r}(t)\,dt=\mathds{1}_{\Omega(r)}.
\]

Hence \(-\delta_{T-r}\) is the exact source, and \(f_{\alpha,\beta,r}\) is expected to approximate it. Note that \(f_{\alpha,\beta,r}\) is computed from the spectral data alone, which are simulated for \(c\equiv 1\). Deviations of \(u^{f_{\alpha,\beta,r}}(T)\) from \(\mathds{1}_{\Omega(r)}\) could, in principle, come either from the source or from the algorithm used to compute the wave. We therefore use the finite difference method (FDM) with equal space and time steps, which in the present case introduces no numerical oscillations. The FDM does not use the spectral data, so it provides an independent check, and we compare the resulting snapshot $u^{f_{\alpha,\beta,r}}(T)$ with theoretical $\mathds{1}_{\Omega(r)}$. Any oscillations in the figures below thus come from \(f_{\alpha,\beta,r}\) itself. Fixing the travel-time parameter \(r=0.7T\), we compute the source \(f_{\alpha,\beta,r}\) for several choices of \((\alpha,\beta)\). Each figure below has two panels: the source \(f_{\alpha,\beta,r}\), compared with \(-\delta_{T-r}\), and the wave snapshot \(u^{f_{\alpha,\beta,r}}(T)\), compared with the indicator \(\mathds{1}_{\Omega(r)}\). Throughout these plots, dotted orange lines denote exact quantities, while solid blue lines denote the computed ones.

We first set \(\beta=0\), so that only the \(L^2\)-penalty controlled by \(\alpha\) is used. For \(\alpha=\texttt{1e-3}\), the recovery is poor for both the source and the indicator. Reducing \(\alpha\) improves the reconstruction (Figure~\ref{fig:figure10}, \(\alpha=\texttt{2e-4}\)), but for \(\alpha=\texttt{1e-5}\) Gibbs-type oscillations appear near the discontinuity of \(f_{\alpha,\beta,r}\), as shown in Figure~\ref{fig:figure11}. Thus the oscillation threshold lies between \(\texttt{1e-5}\) and \(\texttt{2e-4}\). A manual search gives \(\alpha=\texttt{3e-5}\) as the best oscillation-free choice for pure \(L^2\)-regularization; see Figure~\ref{fig:figure12}.

\begin{figure}[H]
	\centering	
	\includegraphics[scale=1]{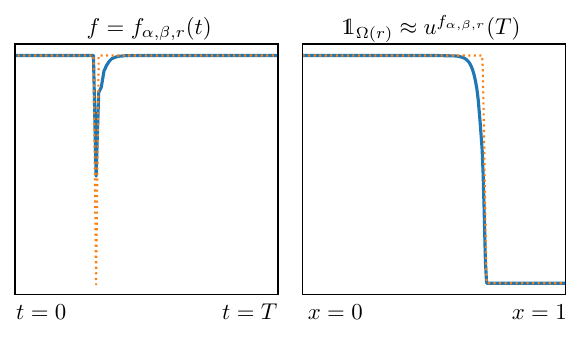}
	\caption{$\alpha=\texttt{2e-4}$, $\beta=0$. Reconstruction remains smoother than the true one.}
	\label{fig:figure10}
\end{figure}

\begin{figure}[H]
	\centering	
	\includegraphics[scale=1]{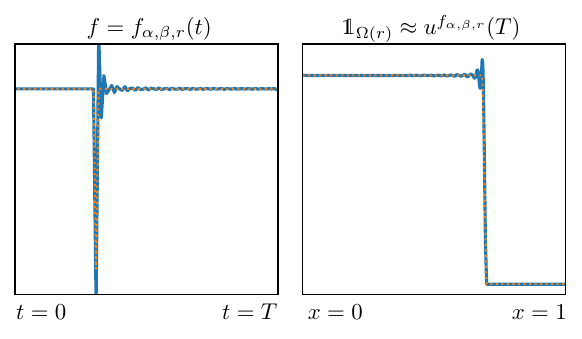}
	\caption{$\alpha=\texttt{1e-5}$, $\beta=0$. Reconstruction is under-regularized and the Gibbs phenomenon appears near the discontinuity of $f_{\alpha,\beta, r}$.}
	\label{fig:figure11}
\end{figure}

\begin{figure}[H]
	\centering	
	\includegraphics[scale=1]{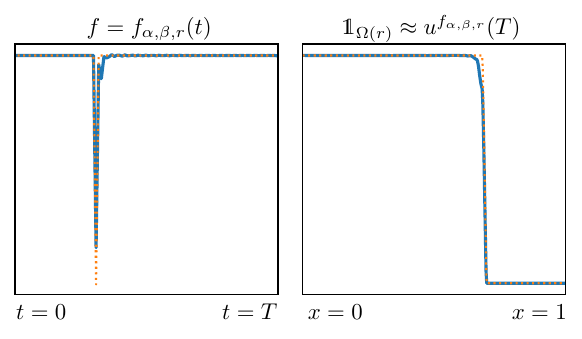}
	\caption{Fine tuned $\alpha=\texttt{6e-5}$, $\beta=0$. The parameter
		$\alpha=\texttt{6e-5}$ is the smallest value in the oscillation-free regime, and therefore yields the best reconstruction achievable with $L^2$ regularization alone.}
	\label{fig:figure12}
\end{figure}

We next investigate whether the $H^1$ regularization term, controlled by $\beta$, can suppress the Gibbs oscillations more effectively than a coarser $\alpha$ alone. We fix $\alpha = \texttt{1e-5}$, at which point the reconstructed source already oscillates with $\beta = 0$, and progressively reduce $\beta$ from $\texttt{1e-9}$ downward. As shown in Figure \ref{fig:figure14} and Figure \ref{fig:figure15}, increasing $\beta$ damps the oscillations at the cost of smoothing the source, while decreasing $\beta$ recovers the sharpness of $f_{\alpha,\beta, r}$ at the expense of reintroducing the Gibbs phenomenon. The threshold is reached at $\beta = \texttt{1e-11}$ (Figure~\ref{fig:figure15}), below which oscillations reappear.

\begin{figure}[H]
	\centering	
	\includegraphics[scale=1]{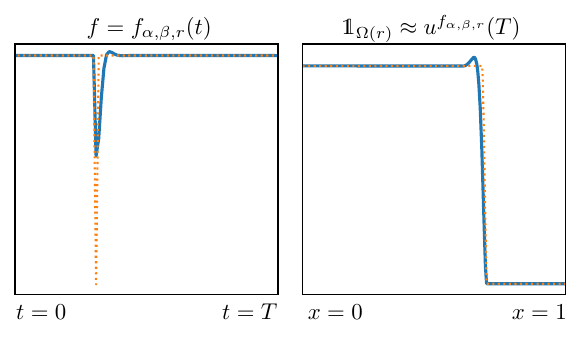}
	\caption{$\alpha=\texttt{1e-5}$, $\beta=\texttt{1e-9}$.
		Reducing $\beta$ recovers sharpness while oscillations remain suppressed.}
	\label{fig:figure14}
\end{figure}

\begin{figure}[H]
	\centering	
	\includegraphics[scale=1]{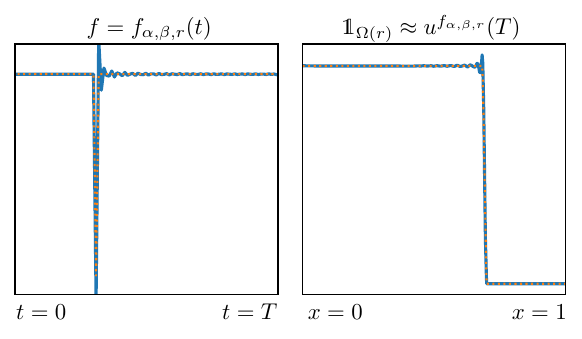}
	\caption{$\alpha=\texttt{1e-5}$, $\beta=\texttt{1e-11}$.
		The Gibbs phenomenon reappears, marking the lower bound of the admissible $\beta$-regime.}
	\label{fig:figure15}
\end{figure}

A joint search over both parameters in the neighborhood of these threshold values yields the pair $\alpha = \texttt{3e-5}$, $\beta = \texttt{3e-10}$ as the optimal combination, shown in Figure~\ref{fig:figure16}. The resulting reconstruction is oscillation-free and resolves the discontinuity of $f_{\alpha,\beta, r}$ more sharply than the $L^2$-only optimum of Figure~\ref{fig:figure12}; accordingly, the recovered indicator $\mathds{1}_{\Omega(r)}$ exhibits a steeper transition at the wavefront and a closer agreement with the true characteristic function. This confirms that the combined $L^2$--$H^1$ regularization norm $\|\cdot\|_{\alpha,\beta}$ outperforms pure $L^2$ penalization when the target control contains near-discontinuous features.

\begin{figure}[H]
	\centering	
	\includegraphics[scale=1]{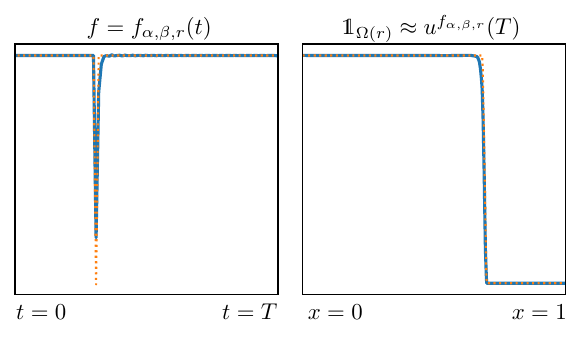}
	\caption{Fine tuned $\alpha=\texttt{3e-5}$, $\beta=\texttt{3e-10}$.
		The joint $L^2$--$H^1$ regularization yields a sharper and more accurate
		reconstruction than $L^2$ penalization alone.}
	\label{fig:figure16}
\end{figure}

\medskip

\subsection{Wave speed profiles}\label{sec:profiles}
Figure~\ref{fig:paperfigure} compares the reconstructed profiles (black solid) with the exact ones (orange dashed): a constant profile (a), smooth profiles in \(\mathcal{C}\) (b), (c), a smooth profile not in \(\mathcal{C}\) (d), and piecewise constant profiles (e), (f). For each profile, \(\alpha\), \(\beta\), the method for differentiating \(V\) and its regularization parameter are chosen on a grid by minimizing the relative \(L^2\) error, and (d) includes the linear extrapolation at the right end. Table~\ref{tab:tab4} reports the relative errors on \(\overline\Omega\) in the \(TV\), \(L^\infty\), \(L^2\) and \(L^1\) norms.

\begin{figure}[H]
	\centering
	\includegraphics[scale=0.8]{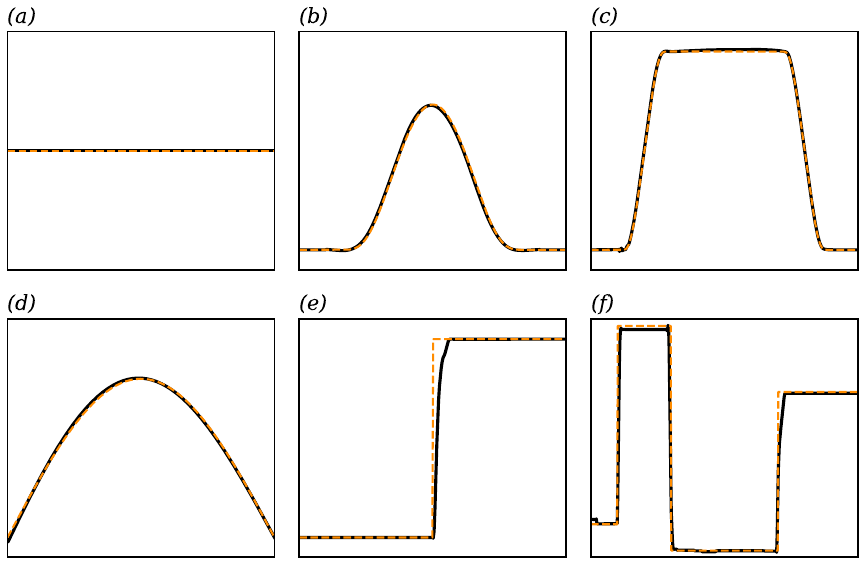}
	\caption{Exact wave speed profiles (orange dashed) vs.\ reconstructed profiles (black solid). (a)--(c)~smooth profiles in $\mathcal{C}$; (d)~smooth profile not in $\mathcal{C}$; (e)~single-jump piecewise constant; (f)~multiple-jumps piecewise constant.}
	\label{fig:paperfigure}
\end{figure}
\begin{table}[H]
	\centering
	\small
	\renewcommand{\arraystretch}{1.2}
	\begin{tabular}{lcccccc}
		\toprule
		Metric & (a) & (b) & (c) & (d) & (e) & (f) \\
		\midrule
		$TV$ & N/A & $3.27\%$ & $5.59\%$ & $0.647\%$ & $102\%$ & $175\%$\\
		$L^\infty$ & $0.01\%$ & $1.11\%$ & $2.6\%$ & $0.17\%$ & $16\%$ & $38.4\%$\\
		$L^2$ & $0.004\%$ & $0.2\%$ & $0.323\%$ & $0.095\%$ & $1.29\%$ & $4.65\%$\\
		$L^1$ & $0.003\%$ & $0.072\%$ & $0.15\%$ & $0.087\%$ & $0.344\%$ & $1\%$ \\
		\bottomrule
	\end{tabular}
	\caption{Relative reconstruction errors for the six profiles shown in Figure~\ref{fig:paperfigure}.}
	\label{tab:tab4}
\end{table}

Plots~(a)--(d) correspond to smooth wave speed profiles. The constant profile \(c(x)=1\), shown in~(a), serves as a baseline test: in this case the travel-time volume is linear, \(V(r)=r\). The reconstruction is essentially exact, with relative \(L^2\) and \(L^1\) errors below \(0.004\%\).

For the smooth profiles~(b) and~(c) in \(\mathcal{C}\), the main source of error is the numerical recovery of \(V'\) from the computed volume function. The relative \(L^2\) and \(L^1\) errors remain below \(0.33\%\). The errors are concentrated near the endpoints of \(\Omega\), where the regularized differentiation of \(V\) has end effects. The \(TV\) and \(L^\infty\) errors are more sensitive to these localized deviations .

Profile~(d) is smooth but does not belong to \(\mathcal{C}\), since \(c'(0)\neq0\) and \(c'(1)\neq0\). Its reconstruction is accurate in the interior, and the violation of the class assumptions affects only the right end. There, the values of \(\hat c\) just before \(\hat\chi(T)\) are affected by the end effect of the differentiation of \(V\), and on \((\hat\chi(T),1]\) the reconstruction is only continued constantly. We therefore apply a linear extrapolation at the right end: \(\hat c\) is kept on \([0,\hat\chi(T)-\delta]\), and on \((\hat\chi(T)-\delta,1]\) it is replaced by the straight line fitted by least squares to its last \(20\) values before \(\hat\chi(T)-\delta\). The values \(\delta=0.02\) and \(20\) were chosen empirically. This simple correction is particularly effective here.

Plots~(e)--(f) show piecewise constant profiles with jump discontinuities. In both cases the constant levels are recovered accurately, and the errors are concentrated at the jumps. A small error in a jump location produces a pointwise discrepancy of the size of the jump on a short interval. This explains the high \(TV\) and \(L^\infty\) errors. The \(L^2\) and \(L^1\) errors, which weight the size of the interval where the discrepancy occurs, remain moderate. Thus, for discontinuous profiles, the \(L^2\) and \(L^1\) errors provide a more representative measure of the overall reconstruction quality.

The deterioration of \(TV\) and \(L^\infty\) accuracy in the multiple-jump case is consistent with the difficulty of resolving several closely spaced discontinuities. We do not have a definitive explanation for all local features observed in this example. A possible reason is the interaction between regularization, numerical differentiation of \(V\), and small errors in the estimated jump positions. A more detailed study of this effect, including a comparison with \(TV\)-based regularization methods such as the ADMM solver, is left for future work.

\medskip

\subsection{Sensitivity to noise in the spectral data}\label{sec:noise}
The measured data are the two spectra. We perturb each of them by independent multiplicative Gaussian noise,
\[
\tilde\lambda_n(0)=\lambda_n(0)\,(1+\sigma\xi_n),\qquad
\tilde\lambda_n(\omega)=\lambda_n(\omega)\,(1+\sigma\xi_n'),\qquad n=1,\dots,L,
\]
where \(\xi_n,\xi_n'\sim\mathcal{N}(0,1)\) are independent and \(\sigma>0\) is the relative noise level. The derivatives are computed from the perturbed eigenvalues as in the noise-free case, \(\tilde\lambda_n'(0)=(\tilde\lambda_n(\omega)-\tilde\lambda_n(0))/\omega\). The perturbed data enter every step of the reconstruction, including the estimate \(\hat\tau(1)\) and hence \(T\). All other parameters are those of the noise-free experiments, and each noise level is tested with three realizations.

We first estimate the error \(e_n=\tilde\lambda_n'(0)-\lambda_n'(0)\) of the derivative data. Fix \(\omega_0\in (0,1)\) and let \(M_n=\sup_{[0,\omega_0]}|\lambda_n''|\), which is finite by Proposition~\ref{lem2}. For \(\omega\in (0,\omega_0]\), Taylor's theorem gives \(\lambda_n(\omega)=\lambda_n(0)+\omega\lambda_n'(0)+\frac12\omega^2\lambda_n''(\theta_n)\) for some \(\theta_n\in(0,\omega)\), and hence
\[
e_n=\underbrace{\frac{\omega}{2}\,\lambda_n''(\theta_n)}_{\text{modelling error}}+\underbrace{\frac{\sigma}{\omega}\bigl(\lambda_n(\omega)\xi_n'-\lambda_n(0)\xi_n\bigr)}_{\text{amplified noise}} .
\]
Since \(\xi_n,\xi_n'\) are independent with zero mean and unit variance,
\[
\mathbb{E}\,e_n=\frac{\omega}{2}\,\lambda_n''(\theta_n),\quad
\operatorname{Var}e_n=\frac{\sigma^2}{\omega^2}\bigl(\lambda_n(\omega)^2+\lambda_n(0)^2\bigr)=\frac{2\sigma^2\lambda_n(0)^2}{\omega^2}\bigl(1+O(\omega)\bigr).
\]
Without noise, \(e_n=\mathbb{E}\,e_n\), only the modelling error remains, and the error of the derivative data is of order \(\omega\), since \(|e_n|\leqslant\tfrac12\omega M_n\).

The noise term determines when the data become unusable. Relative to the exact value, its standard deviation is
\[
s_n=\frac{\sqrt{\operatorname{Var}e_n}}{|\lambda_n'(0)|}\approx\frac{\sqrt2\,\sigma\lambda_n(0)}{\omega\,|\lambda_n'(0)|}.
\]
Since \(\lambda_n(0)\asymp n^2\) increases with \(n\) while \(|\lambda_n'(0)|=|\phi_n(0)|^2\asymp1\), the largest relative error occurs at \(n=L\), and
\[
\rho=\max_{n\leqslant L}s_n\approx\frac{\sigma\lambda_L(0)}{\omega}
\]
up to a factor of order one. For \(\rho\ll1\), every coefficient \(\tilde\lambda_n'(0)\) of \(\Lambda_L^\omega\) is accurate. For \(\rho\gg1\), the coefficients with \(n\) close to \(L\) are dominated by noise, and even their sign is random. We therefore expect the reconstruction to break down at \(\rho\approx1\), that is, at the noise level
\[
\sigma^*\approx\frac{\omega}{\lambda_L(0)} .
\]
This threshold grows linearly with \(\omega\). Since \(\lambda_L(0)\asymp L^2\), it is also expected to decrease like \(1/L^{2}\), and thereby using more eigenvalues improves the resolution of \(\Lambda_L^\omega\), but lowers the admissible noise level.

Table~\ref{tab:noise_rec} tests this prediction for profiles~(c), (d) and~(e) at two noise levels for each \(\omega\), one below the expected threshold and one above it. Below the threshold, the reconstruction is as accurate as without noise; above it, the reconstruction deteriorates strongly or breaks down, for every \(\omega\). A breakdown means that, for at least one realization, the relative error exceeds \(100\%\) or the computed \(V'\) is not positive, so that \(\chi\) cannot be inverted.

\begin{table}[H]
	\centering
	\small
	\renewcommand{\arraystretch}{1.2}
	\begin{tabular}{ccccccc}
		\toprule
		$\omega$ & $\sigma$ & $\rho$ & data error & (c) & (d) & (e)\\
		\midrule
		$10^{-4}$ & $10^{-10}$ & $0.47$ & $0.34$ & $0.388\%$ & $0.461\%$ & $1.51\%$\\
		          & $10^{-9}$  & $4.7$  & $3.4$  & -- & -- & $15.1\%$\\
		\midrule
		$10^{-3}$ & $10^{-9}$  & $0.47$ & $0.34$ & $0.384\%$ & $0.381\%$ & $1.47\%$\\
		          & $10^{-8}$  & $4.7$  & $3.4$  & -- & -- & $19.4\%$\\
		\midrule
		$10^{-2}$ & $10^{-8}$  & $0.47$ & $0.34$ & $0.390\%$ & $0.679\%$ & $1.01\%$\\
		          & $10^{-7}$  & $4.7$  & $3.4$  & -- & $34.7\%$ & $25.8\%$\\
		\midrule
		$10^{-1}$ & $10^{-7}$  & $0.47$ & $0.34$ & $1.91\%$ & $2.14\%$ & $4.64\%$\\
		          & $10^{-6}$  & $4.7$  & $3.4$  & -- & -- & --\\
		\bottomrule
	\end{tabular}
	\caption{Reconstruction under noise below and above the threshold \(\rho\approx1\). The values of \(\rho=\sigma\lambda_L(0)/\omega\) and of the data error \(\max_n|\tilde\lambda_n'(0)-\lambda_n'(0)|/|\lambda_n'(0)|\) are means over~(c)--(e). Last three columns: relative \(L^2\) errors; a dash marks a breakdown.}
	\label{tab:noise_rec}
\end{table}

The method is therefore robust to noise in the spectral data as long as \(\rho\ll1\), that is, as long as for every \(n\leqslant L\) the noise in \(\tilde\lambda_n'(0)\) is much smaller than \(\lambda_n'(0)\) itself. In the experiments, \(\rho\approx0.5\) already suffices for the reconstruction to be as accurate as without noise, while for \(\rho\approx5\) it deteriorates rapidly. The trade-off in \(\omega\) remains: a larger \(\omega\) admits more noise, but its modelling error increases the reconstruction error, as seen in the last block of Table~\ref{tab:noise_rec}. The required accuracies of the eigenvalues are realistic only if the two spectra are measured very precisely, or if their noise is strongly correlated and largely cancels in the difference. A direct measurement of \(\lambda_n'(0)\), or a regularized estimate from several values of \(\omega\), would avoid the amplification by \(1/\omega\).

\medskip

\subsection{Conclusions}
The numerical experiments confirm that the boundary control method reconstructs the wave speed accurately from the two finite spectra alone. For \(c\equiv1\), the sources computed from the spectral data generate waves that approximate the indicators of the domains of influence, as predicted by Proposition~\ref{2mainres}; this was verified with an independent finite difference solver. The combined \(L^2\)--\(H^1\) penalty permits sharper sources without Gibbs oscillations and yields a smaller misfit than the pure \(L^2\) penalty. The smooth profiles are reconstructed with small errors, which are concentrated near the endpoints of \(\Omega\), where the regularized differentiation of \(V\) has end effects. For the smooth profile outside \(\mathcal{C}\), the violation of the class assumptions affects only the right end and is removed by a simple linear extrapolation. For the piecewise constant profiles, the constant levels are recovered accurately, and the errors are concentrated at the jumps. The noise study shows that the reconstruction is stable as long as the noise in the eigenvalues is small compared with the Robin parameter \(\omega\). A larger \(\omega\) tolerates more noise, at the cost of a larger modelling error of the finite difference approximation of~\(\lambda_n'(0)\).

These results should be read with the choice of the parameters in mind. The mesh size \(N\), the regularization parameters \(\alpha\) and \(\beta\), the method for the differentiation of \(V\) and its regularization parameter were selected by a grid search minimizing the relative \(L^2\) error with respect to the exact wave speed. The margin \(\eta\) and the parameters of the linear extrapolation for profile~(d) were chosen in the same way. The discretization of the data, that is, \(L\), \(\omega\) and \(h\), and the range of indices in the fit for \(\hat\tau(1)\) were fixed a priori. The reported errors therefore show the accuracy attainable with the method, and a finer or systematic optimization may improve them further.

In practice such a calibration is usually not available. The true wave speed is unknown, so the error minimized by the grid search cannot be computed. Data-driven parameter choice rules avoid this, but they have their own requirements: the discrepancy principle needs the noise level, and heuristic rules such as the L-curve or generalized cross-validation may fail for the strongly ill-posed differentiation of \(V\). Moreover, the parameters interact, so they cannot be tuned one at a time. The optimal regularization depends on \(N\) and \(\eta\), and the optimal \(\eta\) depends on the regularization. Some parameters are not free at all: \(L\) and \(\omega\) are dictated by the measurement, and by the noise study, \(\omega\) also determines the admissible noise level. Future work therefore includes a data-driven choice of the parameters, an analysis of how they should be coupled to the discretization and the noise level, and a direct or regularized estimation of \(\lambda_n'(0)\) that avoids the amplification of noise by \(1/\omega\).

\medskip

\section*{Acknowledgement}
The authors thank Professor Nuutti Hyv\"onen of Aalto University for helping to establish the collaboration that led to this paper.

This work was supported by the Finnish Ministry of Education
and Culture’s Pilot for Doctoral Programmes (Pilot project Mathematics of Sensing, Imaging and Modelling), the FAME flagship of the Research Council of Finland (Flagship of Advanced Mathematics for Sensing Imaging and Modelling, grant 359182, grant 359181), the Centre of Excellence of Inverse Modelling and Imaging of Research Council of Finland (grant 336786), and the European Research Council (the Advanced Grant project 101097198). Views and opinions expressed are those of the authors only and do not necessarily reflect those of the European Union or the other funding organizations.

\section*{Data availability}
The code implementing the reconstruction algorithm of Section~\ref{compimpl} and reproducing the experiments of Section~\ref{numres} is openly available at \url{https://github.com/pkulikov239/bcm-inverse-spectral} and archived at Zenodo, \url{https://doi.org/10.5281/zenodo.22961134}. The results reported here were obtained with version~1.0, \url{https://doi.org/10.5281/zenodo.22961135}. All spectral data are simulated by this code; no experimental data were used.

\bibliographystyle{plain} 
\bibliography{references.bib}	

\vspace{1.5cm}
\begin{flushleft}
	\begin{footnotesize}
		\sc Antti Hannukainen:
		
		Department of Mathematics and Systems Analysis, Aalto University
		
		Otakaari 1B,
		
		Espoo, 02150, Finland
		
		\textit{E-mail address:} \tt{antti.hannukainen@aalto.fi}
	\end{footnotesize}
\end{flushleft}
\vspace{0.3cm}
\begin{flushleft}
	\begin{footnotesize}
		\sc Petr Kulikov:
		
		Department of Mathematics and Statistics, University of Helsinki
		
		Pietari Kalmin katu 5,
		
		Helsinki, 00560, Finland
		
		\textit{E-mail address:} \tt{petr.kulikov@helsinki.fi}
	\end{footnotesize}
\end{flushleft}
\vspace{0.3cm}
\begin{flushleft}
	\begin{footnotesize}
		\sc Matti Lassas:
		
		Department of Mathematics and Statistics, University of Helsinki
		
		Pietari Kalmin katu 5,
		
		Helsinki, 00560, Finland
		
		\textit{E-mail address:} \tt{matti.lassas@helsinki.fi}
	\end{footnotesize}
\end{flushleft}
\vspace{0.3cm}
\begin{flushleft}
	\begin{footnotesize}
		\sc Lauri Oksanen:
		
		Department of Mathematics and Statistics, University of Helsinki
		
		Pietari Kalmin katu 5,
		
		Helsinki, 00560, Finland
		
		\textit{E-mail address:} \tt{lauri.oksanen@helsinki.fi}
	\end{footnotesize}
\end{flushleft}
\end{document}